\documentclass[a4paper,11pt]{amsart}

\usepackage{amsmath}
\usepackage{amssymb}
\usepackage{amsthm}
\usepackage{verbatim}
\usepackage[all]{xy}
\usepackage{enumerate}
\usepackage{dsfont}
\usepackage{color}
\usepackage{hyperref}
\usepackage{enumitem}

\makeatletter
\@namedef{subjclassname@2020}{%
	\textup{2020} Mathematics Subject Classification}
\makeatother

\newtheorem{thm}{Theorem}[section]
\newtheorem{prop}[thm]{Proposition}
\newtheorem{cor}[thm]{Corollary}
\newtheorem{lem}[thm]{Lemma}

\theoremstyle{definition}

\newtheorem{dfn}[thm]{Definition}

\newtheorem{rmk}[thm]{Remark}

\newtheorem{rem}[thm]{Remark}

\newtheorem*{acknow}{Acknowledgement}
\numberwithin{equation}{section}

\newtheorem{introtheorem}{Theorem}

\newtheorem{introcorollary}[introtheorem]{Corollary}
\newtheorem{introproblem}[introtheorem]{Problem}

\numberwithin{equation}{section}
\newcommand{\cK}{\mathcal{K}}

\newcommand{\cH}{\mathcal{H}}

\newcommand{\Id}{\textrm{Id}}

\newcommand{\N}{\mathbb{N}}
\newcommand{\R}{\mathbb{R}}

\newcommand{\Pol}{{\rm Pol}}

\newcommand{\bj}{{\bf{j}}}

\newcommand{\cS}{\mathcal{S}}

\begin{document}
\title[Relative Haagerup property for arbitrary von Neumann algebras]{Relative Haagerup property for arbitrary von Neumann algebras}
\begin{abstract}
	We introduce the relative Haagerup approximation property for a unital, expected inclusion of arbitrary von Neumann algebras and show that if the smaller algebra is finite then the notion only depends on the inclusion itself, and not on the choice of the conditional expectation.
	Several variations of the definition are shown to be equivalent in this case, and in particular the approximating maps can be chosen to be unital and preserving the reference state.
	The concept is then applied to amalgamated free products of von Neumann algebras and used to deduce that the standard Haagerup property for a von Neumann algebra is stable under taking free products with amalgamation over finite-dimensional subalgebras. The general results are illustrated by examples coming from $q$-deformed Hecke-von Neumann algebras and von Neumann algebras of quantum orthogonal groups.
	
\end{abstract}

\date{\noindent \today}
\keywords{relative Haagerup property, von Neumann algebra, amalgamated free product}
\subjclass[2020]{Primary 46L10}

\author{Martijn Caspers}
\author{Mario Klisse}
\author{Adam Skalski}
\author{Gerrit Vos}
\author{Mateusz Wasilewski}

\address{TU Delft, EWI/DIAM,
	P.O.Box 5031,
	2600 GA Delft,
	The Netherlands}

\email{m.p.t.caspers@tudelft.nl}

\email{m.klisse@tudelft.nl, mario.klisse@kuleuven.be }

\curraddr{KU Leuven, Celestijnenlaan 200B, 3001 Leuven, Belgium}

\email{g.m.vos@tudelft.nl}

\address{Institute of Mathematics of the Polish Academy of Sciences, ul.\ \'{S}niadeckich 8, 00--656
Warszawa, Poland}

\email{a.skalski@impan.pl}

\address{KU Leuven, Celestijnenlaan 200B, 3001 Leuven, Belgium}
\curraddr{Institute of Mathematics of the Polish Academy of Sciences, ul.\ \'{S}niadeckich 8, 00--656
	Warszawa, Poland}

\email{mateusz.wasilewski@kuleuven.be, mwasilewski@impan.pl }

\date{}

\maketitle


\section{Introduction}

The story of the group-theoretic Haagerup property began in a celebrated article \cite{HaaFree} by Haagerup, who noted that the free group admits a sequence of positive-definite functions vanishing at infinity which pointwise converges to a constant function equal to $1$; in other words, the free group von Neumann algebra admits a sequence of unital completely positive Herz-Schur multipliers which are in a sense `small' and yet converge to the identity operator. Soon after that Choda gave in \cite{Choda} a definition of the Haagerup property for a von Neumann algebra $M$  equipped with a faithful normal trace in terms of existence of abstract approximating maps on $M$, which behave well with respect to the trace in question. Later Jolissaint proved in \cite{Jolissaint}  that the property does not depend in fact on the choice of a trace as above, and Bannon and Fang showed in \cite{BannonFang} that some of the requirements concerning the approximating maps can be weakened. For several years the study focused on finite von Neumann algebras, mainly as the motivating examples came from discrete groups. This changed with the articles \cite{Brannan},   \cite{CFY}, which established the Haagerup property for the von Neumann algebras of certain discrete quantum groups, and the paper \cite{DFSW}, which introduced and studied the analogous property for quantum groups themselves. Soon after that Okayasu and Tomatsu on one hand, and two authors of this paper on another, gave a definition of the Haagerup property for an arbitrary von Neumann algebra equipped with a faithful normal semifinite weight and proved that in fact the notion does not depend on the choice of the weight in question (see \cite{CS-IMRN}, \cite{CS-CMP}, \cite{OkayasuTomatsu} as well as \cite{COST} and references therein). In all the cases above the Haagerup property should be thought of as a natural weakening of amenability/injectivity, which permits applying several approximation ideas and techniques beyond the class of amenable groups or algebras. The class of discrete groups enjoying the Haagerup property has good permanence properties, among which we would like to stress the fact that it is closed under taking free products amalgamated over finite subgroups (\cite[Section 6]{HaagerupProperty}).

In several group-theoretic and operator algebraic contexts it is important to consider also \emph{relative} properties; for example relative Property (T) is key to showing that $\mathbb{Z}^2 \rtimes SL_2(\mathbb{Z})$ does not have the Haagerup property -- which in turn has several von Neumann algebraic consequences, studied among other places in \cite{YongleAdam}. In the context of finite von Neumann algebras the relative Haagerup property appeared first in \cite{Boca} in the study of Jones' towers associated with irreducible finite index subfactors, and was later applied in \cite{Popa06} as a key tool to obtain deep structural results about algebras admitting a certain type of Cartan inclusion  (i.e.\ maximal abelian subalgebras with a `sufficiently rich' normalizer). Such Cartan inclusions are deeply related to von Neumann algebras of equivalence relations by the celebrated results of Feldman and Moore in \cite{FM}.  The case of Cartan subalgebras was also the first in which a definition of a relative Haagerup property was proposed  beyond finite von Neumann algebras \cite{Ueda}, \cite{ClaireEquiv}.  Notably the latter developments took place even before the usual Haagerup property for arbitrary von Neumann algebras was well understood.

The study in \cite{Ueda} and \cite{ClaireEquiv} was focused on the special case of Cartan subalgebras. In this paper we undertake a systematic investigation of a von Neumann algebraic relative Haagerup property for a unital inclusion $N \subseteq M$ equipped with a faithful normal conditional expectation $\mathbb{E}_N:M \to N$. Again we first define it in terms of a fixed faithful normal state (preserved by $\mathbb{E}_N$) but then quickly show that it depends only on the conditional expectation in question. Much more can be said in the case where $N$ is assumed to be finite; here we obtain the following theorem, which is one of the main results of the paper. Relevant definitions can be found in Section 2; essentially for a triple $(M, N, \mathbb{E}_N)$ to have the relative Haagerup property we require the existence of completely positive, normal, $N$-bimodular maps on $M$ which are $\mathbb{E}_N$-decreasing, $L^2$-compact (in the sense determined by the conditional expectation $\mathbb{E}_N$), uniformly bounded and converge point-strongly to the identity, see Definition \ref{Dfn=RHAPState}.

\begin{introtheorem} \label{thm:A}
Suppose that $N \subseteq M$ is a unital, expected inclusion of  von Neumann algebras and assume that $N$ is finite. Then the relative Haagerup property of the triple $(M, N, \mathbb{E}_N)$ does not depend on the choice of a faithful normal conditional expectation $\mathbb{E}_N:M \to N$. 	
	Moreover if $(M, N, \mathbb{E}_N)$ has the relative Haagerup property and we are given a fixed state $\tau \in N_*$ we can always assume that the approximating maps are unital and $\tau \circ \mathbb{E}_N$-preserving.
	\end{introtheorem}

The key idea of the proof once again, as in \cite{CS-IMRN}, uses crossed products by modular actions and the passage to the semifinite setting that (Takai-)Takesaki duality permits. However, the relative context makes the technical details much more demanding and makes adapting the earlier methods -- including those developed in \cite{BannonFang} -- significantly more complicated. On the other hand, allowing non-trivial inclusions allows us to significantly broaden the class of examples fitting into our framework and yields certain facts which are new even in the context of the standard Haagerup property of finite von Neumann algebras. This is exemplified by the next key result of this work and its corollary (which also requires proving a general theorem on the relative Haagerup property of amalgamated free products).

\begin{introtheorem} \label{thm:B}
 	Suppose that $N \subseteq M$ is a unital inclusion of von Neumann algebras equipped with a faithful normal conditional expectation $\mathbb{E}_N:M \to N$ and assume that $N$ is finite-dimensional. Then the relative Haagerup property of the triple $(M, N, \mathbb{E}_N)$ is equivalent to the usual Haagerup property of $M$.
\end{introtheorem}

\begin{introcorollary} \label{cor:C}
	Suppose that $N \subset M_1$ and $N \subset M_2$ are  unital inclusions of  von Neumann algebras equipped with respective faithful normal conditional expectations. If $N$ is finite-dimensional and both $M_1, M_2$ have the Haagerup property, then the amalgamated free product $M_1 \ast_N M_2$ also has the Haagerup property.
\end{introcorollary}

We illustrate our results with examples coming on one hand from the class of $q$-deformed Hecke-von Neumann algebras of (virtually free) Coxeter groups, and on the other hand from discrete quantum groups.

Of particular interest is also the elementary case of $M = \mathcal{B}(\mathcal{H})$ which provides us with both triples $(\mathcal{B}(\mathcal{H}), N, \mathbb{E}_N)$ that do and do not have the relative Haagerup property. This  is rather surprising as it gives examples of co-amenable inclusions of von Neumann algebras in the sense of \cite{PopaCorr} (see also \cite{BannonMarrakchiOzawa}) without the relative Haagerup property. 

\vspace{0.3cm}

The detailed plan of the paper is as follows. After this introduction we recall some facts regarding von Neumann algebras, their modular theory and completely positive approximations in Section 2, and introduce (certain variants of) the definition of the relative Haagerup property in Section 3. The latter section also contains the initial discussion of the independence of our notion of various ingredients, mainly in the semifinite setting.
Section 4, the most technical part of the paper, introduces the crossed product arguments allowing us to reduce the general problem to the semifinite case. Section 5 contains the main general results of the paper; in particular the above  Theorem \ref{thm:A} is a combination of Theorem \ref{thm:condexpindep} and Theorem \ref{Unitality+State-Preservation}. In Section \ref{Sect=Examples} we briefly describe the known examples of Haagerup inclusions related to Cartan subalgebras. Here we also study the case $M = \mathcal{B}(\mathcal{H})$ which leads to very interesting counterexamples. In Section \ref{Sect=FD} we show that in the case of a finite-dimensional subalgebra the relative Haagerup property is equivalent to the Haagerup property of the larger algebra and prove Theorem \ref{thm:B} above (which is  Theorem \ref{Thm=HAPrelHAP}). In Section \ref{amfreeprod} we discuss the behaviour of the relative Haagerup property with respect to the amalgamated free product construction and show Corollary \ref{cor:C}  (i.e.\ Corollary \ref{Cor=amalgamated}). In Section \ref{amfreeprod} we also discuss briefly a consequence of these results for the Hecke-von Neumann algebras associated to virtually free Coxeter groups.
Finally in a short Section \ref{Sect=Quantum} we present an example of a Haagerup inclusion coming from quantum groups, which in fact is even strongly of finite index.

\begin{acknow}
MC and GV are supported by the NWO Vidi grant `Non-commutative harmonic analysis and rigidity of operator algebras', VI.Vidi.192.018.  MK was supported by the NWO project `The structure of Hecke-von Neumann algebras', 613.009.125.  MW was supported by the Research Foundation – Flanders (FWO) through a Postdoctoral Fellowship and by long term structural funding - Methusalem grant of the Flemish Government. AS was  partially supported by the National Science Center (NCN) grant no. 2020/39/I/ST1/01566.
\end{acknow}

\section{Preliminaries and notation}

\subsection{General notation}

We write $\mathbb{N}:=\{1,2,...\}$, 
$C(\mathbb{R})$ denotes the continuous functions $\mathbb{R} \rightarrow \mathbb{C}$ and let $C_c(\mathbb{R})$ denote the functions in $C(\mathbb{R})$ with compact support.  For $f,g \in L^1(\mathbb{R})$ we define the convolution and involution by
\[
(f \ast g)(s) = \int_\mathbb{R} f(t) g(s-t) dt, \qquad f^\ast(s) = \overline{f(s^{-1})}, \qquad s \in \mathbb{R}.
\]

 The bounded operators on a Hilbert space $\mathcal{H}$ will be denoted by $\mathcal{B}(\mathcal{H})$ and $\otimes$
is the von Neumann algebraic tensor product or the tensor product of Hilbert spaces, as will be clear from the context.
Hilbert space inner products are linear on the left side.

\subsection{General von Neumann algebra theory}
For standard results on the theory von Neumann algebras we refer to  \cite{StratilaZsido}, \cite{Takesaki1}, \cite{Takesaki2}, \cite{PopaDelarocheBook}. For the theory of operator spaces and completely bounded maps we refer to \cite{EffrosRuan}, \cite{PisierBook}. For several results on approximation properties and the Haagerup property we refer to \cite{NateTaka}, \cite{HaagerupProperty}.

\vspace{0.3cm}

\noindent {\bf Assumption:} We will assume throughout that the von Neumann algebras we study are $\sigma$-finite, i.e. they admit faithful normal states.

\vspace{0.3cm}

Let $M\subseteq\mathcal{B}(\mathcal{H})$ be a von Neumann algebra.
A self-adjoint (possibly unbounded) operator $h$ on $\mathcal{H}$
is said to be \emph{affiliated with $M$} if for all $k\in\mathbb{N}$
the corresponding spectral projection $E_{\left[-k,k\right]}(h)$
is an element of $M$. Equivalently, $h$ is affiliated with $M$
if and only if $h$ commutes with all unitaries in the commutant $M^{\prime}\subseteq\mathcal{B}(\mathcal{H})$.

We will always assume inclusions of von Neumann algebras $N\subseteq M$
to be unital in the sense that $1_{M}\in N$, and conditional expectations
to be faithful and normal. We will usually repeat these conditions throughout
the text. For a functional $\varphi\in M_{\ast}$ in the predual of
$M$ and elements $a,b\in M$ we denote by $a\varphi b\in M_{\ast}$
the normal functional given by $(a\varphi b)(x):=\varphi(bax)$, $x\in M$, and further write $a\varphi$ for $a\varphi1$ and $\varphi b$ for $1\varphi b$.
If $\varphi\in M_{\ast}$ is faithful, normal and positive, we write $L^{2}(M,\varphi)$
for the GNS Hilbert space associated with $\varphi$ and $\Omega_{\varphi}$
for the corresponding cyclic vector. We will usually identify $M$
with its image under the GNS representation, so  $M\subseteq\mathcal{B}(L^{2}(M, \varphi))$.
We further write $\left\Vert x\right\Vert _{2,\varphi}:=\varphi(x^{\ast}x)^{1/2}$
for $x\in M$ or, if $\varphi$ is clear from the context, $\Vert x \Vert_2 := \Vert x \Vert_{2,\varphi}$.

An \emph{action $\mathbb{R}\overset{\alpha}{\curvearrowright}M$}
of the group $\mathbb{R}$ on a von Neumann algebra $M$ is a group homomorphism
$\alpha:\mathbb{R}\rightarrow\text{Aut}(M)$, $t\mapsto\alpha_{t}$
such that for every $x\in M$ the map $t\mapsto\alpha_{t}(x)$ is
strongly continuous. We denote the corresponding \emph{crossed product
von Neumann algebra} by $M\rtimes_{\alpha}\mathbb{R}$.\\

The following lemma is standard.


\begin{lem} \label{Lem=SOTversusL2} Let $\varphi$ be a faithful
normal state on a von Neumann algebra $M$. Then, on bounded subsets
of $M$ the strong topology coincides with the topology induced by
the norm $\Vert x\Vert_{2,\varphi}=\varphi(x^{\ast}x)^{1/2}$, $x\in M$.
\end{lem}
\begin{proof}
Recall that by \cite{Takesaki1} we may assume $M\subseteq\mathcal{B}(L^{2}(M, \varphi))$ and that on bounded sets the strong topology of $M$ does not change under the choice of a faithful representation. Now if $(x_i)_{i \in I}$ is a  net in $M$ converging strongly to $x$, then $\Vert x_i - x \Vert_{2,\varphi} =  \Vert (x_i  - x) \Omega_\varphi \Vert \rightarrow 0$. Conversely, suppose that  $(x_i)_{i \in I}$ is a bounded net in $M$ such that $\Vert x_i - x \Vert_{2,\varphi} = \Vert (x_i - x) \Omega_\varphi \Vert \rightarrow 0$. Then  for $a \in M$ analytic for the modular group $\sigma^\varphi$ (see Section \ref{Sect=Modular}) we have by \cite[Lemma 3.18 (i)]{Takesaki2} that
\[
\Vert (x_i - x) a \Omega_\varphi \Vert \leq \Vert \sigma_{i/2}^\varphi(a) \Vert  \Vert (x_i - x) \Omega_\varphi \Vert \rightarrow 0.
\]
 Since by \cite[Lemma VIII.2.3]{Takesaki2} such elements $a \Omega_\varphi$, $a \in M$ are dense in $L^{2}(M, \varphi)$ and $(x_i)_{i \in I}$ is bounded, we conclude by a $2\epsilon$-estimate that $x_i \rightarrow x$ strongly.
\end{proof}

Note that one implication above naturally does not require the uniform boundedness assumption, provided that we assume that $M$ is represented on its standard Hilbert space $L^{2}(M, \varphi)$.

\subsection{Tomita-Takesaki modular theory}\label{Sect=Modular}
Let $M$ be a von Neumann algebra with a faithful normal
positive functional $\varphi\in M_{\ast}$. We let $S_\varphi$ be the closure of the operator
\[
L^2(M, \varphi) \rightarrow L^2(M, \varphi): x \Omega_\varphi \mapsto x^\ast \Omega_\varphi, \qquad x \in M.
\]
Let $S_\varphi = J_{\varphi}\Delta_{\varphi}^{1/2}$ be the (anti-linear) polar decomposition where $J_\varphi$ is the modular conjugation and  $\Delta_{\varphi}$ the modular operator. We have the modular automorphism group $\sigma^\varphi_t(x) = \Delta_\varphi^{it} x \Delta_\varphi^{-it}$.
  Then $(M,L^{2}(M,\varphi),J_{\varphi},L^{2}(M,\varphi)^{+})$
is the standard form of $M$, where the positive cone is given by
\[
L^{2}(M,\varphi)^{+}:=\overline{\{x(JxJ)\Omega_{\varphi}\mid x\in M\}}\subseteq L^{2}(M,\varphi)\text{.}
\]
The standard form is uniquely determined up to a unique (unitarily implemented) isomorphism.
For $x\in M$ and $\xi\in L^{2}(M,\varphi)$ we write $\xi x:=J_{\varphi}x^{\ast}J_{\varphi}\xi$.
    An element $x\in M$
is called \emph{analytic} for $\sigma^{\varphi}$ if the function
$\mathbb{R}\ni t\mapsto\sigma_{t}^{\varphi}(x)\in M$ extends to
a (necessarily unique) analytic function on the complex plane $\mathbb{C}$. In this case
we write $\sigma_{z}^{\varphi}(x)$ for the extension at $z\in\mathbb{C}$.

The \emph{centralizer} of a von Neumann algebra with respect to a faithful normal state $\varphi$ is the set $M^{\varphi}:=\{x \in M \mid \forall_{y \in M} \varphi(xy)= \varphi(yx)\}$, by \cite[Theorem VIII.2.6]{Takesaki2} equivalently described as $\{x \in M \mid \sigma^\varphi_t(x) = x, \;\; t \in \mathbb{R} \}$. We will often consider the situation where $N\subseteq M$ is a unital embedding, equipped with a faithful normal conditional expectation  $\mathbb{E}_N:M \to N$,   $\tau \in N_*$ is a faithful tracial state, and $\varphi = \tau \circ \mathbb{E}_N$. Then an easy computation shows that $N \subseteq M^\varphi$.


\subsection{Completely positive maps}

Let $A,B$ be von Neumann algebras with faithful normal positive functionals
$\varphi$ and $\psi$ respectively. For a linear map $\Phi:A\rightarrow B$
we say that its \emph{$L^{2}$-implementation} $\Phi^{(2)}$ (with
respect to $\varphi$ and $\psi$) exists if the map $x\Omega_{\varphi}\mapsto\Phi(x)\Omega_{\psi}$
extends to a bounded operator $\Phi^{(2)}:L^{2}(A,\varphi)\rightarrow L^{2}(B,\psi)$.
This is the case if and only if there exists a constant $C>0$ such
that for all $x\in A$,
\[
\psi(\Phi(x)^{\ast}\Phi(x))\leq C\varphi(x^{\ast}x).
\]
In particular, if $\Phi$ is $2$-positive with $\psi\circ\Phi\leq\varphi$,
the Kadison-Schwarz inequality implies that $\Phi^{(2)}$ exists with
$\Vert\Phi^{(2)}\Vert\leq\Vert\Phi(1)\Vert^{1/2}$. Indeed, for $x\in A$
\begin{eqnarray*}
\Vert\Phi(x)\Omega_{\psi}\Vert_{2}^{2} & = & \psi(\Phi(x)^{\ast}\Phi(x))\leq\Vert\Phi(1)\Vert\psi(\Phi(x^{\ast}x))\\
 & \leq & \Vert\Phi(1)\Vert\varphi(x^{\ast}x)=\Vert\Phi(1)\Vert\Vert x\Omega_{\varphi}\Vert_{2}^{2}.
\end{eqnarray*}
 This implies that if $\Phi$ is contractive then so is $\Phi^{(2)}$.

The general principle of the following lemma was used as part of a
proof in \cite{CS-IMRN} and \cite{Jolissaint} a number of times.
Here we present it separately. We will also need several straightforward
variations of this lemma. Because they can be proved in a very similar
way, we shall not state them here. The essence of the result is that,
given two nets of maps with suitable properties that strongly converge
to the identity, the composition of these maps gives rise to a net
that also converges to the identity in the strong operator topology.

\begin{lem} \label{Lem=ApproxSOT}
Let $(A,\varphi)$
and $(B,\varphi_{j})$, $j\in\mathbb{N}$ be pairs of von Neumann
algebras equipped with faithful normal states. Consider a normal completely
positive map $\pi:A\rightarrow B$, a bounded sequence of normal completely
positive maps $(\Psi_{j}:B\rightarrow A)_{j\in\mathbb{N}}$ and for
every $j\in\mathbb{N}$ a bounded net of  completely positive maps $\ensuremath{(\Phi_{j,k}:B\rightarrow B})_{k\in K_{j}}$.
Assume that for all $j\in\mathbb{N}$, $k\in K_{j}$ the inequalities
$\varphi_{j}\circ\pi\leq\varphi$, $\varphi\circ\Psi_{j}\leq\varphi_{j}$
and $\varphi_{j}\circ\Phi_{j,k}\leq\varphi_{j}$ hold, that $\Psi_{j}\circ\pi(x)\rightarrow x$
strongly in $j$ for every $x\in A$ and that for every $j\in\mathbb{N}$, $x \in B$
we have $\Phi_{j,k}(x)\rightarrow x$ strongly in $k$. Then there
exists a directed set $\mathcal{F}$ and a function $(\widetilde{j},\widetilde{k})\text{: }\mathcal{F}\rightarrow\{(j,k)\mid j\in \mathbb{N},k\in K_{j}\}\text{, }F\mapsto(\widetilde{j}(F),\widetilde{k}(F))$
such that $\Psi_{\widetilde{j}(F)}\circ\Phi_{\widetilde{j}(F),\widetilde{k}(F)}\circ\pi(x)\rightarrow x$
strongly in $F$ for every $x\in A$.
\end{lem}

\begin{proof}
For $j\in\mathbb{N}$ and $k\in K_{j}$ write
\begin{eqnarray}
\nonumber
\pi_{j}^{(2)}&:& L^{2}(A,\varphi)\rightarrow L^{2}(B,\varphi_{j}),x\Omega_{\varphi}\mapsto\pi(x)\Omega_{\varphi_j},\\
\nonumber
\Psi_{j}^{(2)}&:&L^{2}(B,\varphi_{j})\rightarrow L^{2}(A,\varphi),x\Omega_{\varphi_{j}}\mapsto\Phi_{j,k}(x)\Omega_{\varphi},\\
\nonumber
\Phi_{j,k}^{(2)}&:&L^{2}(B,\varphi_{j})\rightarrow L^{2}(B,\varphi_{j}),x\Omega_{\varphi_{j}}\mapsto\Phi_{j,k}(x)\Omega_{\varphi_{j}}
\end{eqnarray}
for the corresponding $L^2$-implementations with respect to $\varphi$ and $\varphi_{j}$.
Let $C\geq1$ be a bound for the norms of $(\Psi_{j})_{j\in\mathbb{N}}$
and hence for the norms of $(\Psi_{j}^{(2)})_{j\in\mathbb{N}}$. We shall
make use of the fact that on bounded sets the strong topology coincides
with the $L^{2}$-topology determined by a state, see Lemma \ref{Lem=SOTversusL2}.
Therefore we have strong limits $\Psi_{j}^{(2)}\pi_{j}^{(2)}\rightarrow1$
in $\mathcal{B}(L^{2}(A,\varphi))$ and $\Phi_{j,k}^{(2)}\rightarrow1$
in $\mathcal{B}(L^{2}(B,\varphi_{j}))$. Now let $F\subseteq L^{2}(A,\varphi)$
be a finite subset. We may find $j=\widetilde{j}(F)\in\mathbb{N}$ such
that for all $\xi\in F$,
\[
\Vert\Psi_{j}^{(2)}\pi_{j}^{(2)}\xi-\xi\Vert_{2}<|F|^{-1}\text{.}
\]
In turn, we may find $k=\widetilde{k}(j,F)=\widetilde{k}(F)$ such that for
all $\xi\in F$,
\[
\Vert\Phi_{j,k}^{(2)}\pi_{j}^{(2)}\xi-\pi_{j}^{(2)}\xi\Vert_{2}<|F|^{-1}\text{.}
\]
From the triangle inequality and by using that the operator norm of
$\Psi_{j}^{(2)}$ is bounded by $C$,
\begin{eqnarray}
\nonumber
\Vert\Psi_{j}^{(2)}\Phi_{j,k}^{(2)}\pi_{j}^{(2)}\xi-\xi\Vert_{2} &\leq& \Vert\Psi_{j}^{(2)}\Phi_{j,k}^{(2)}\pi_{j}^{(2)}\xi-\Psi_{j}^{(2)}\pi_{j}^{(2)}\xi\Vert_{2}+\Vert\Psi_{j}^{(2)}\pi_{j}^{(2)}\xi-\xi\Vert_{2}\\
\nonumber
&\leq& \Vert\Psi_{j}^{(2)}\Vert\Vert\Phi_{j,k}^{(2)}\pi_{j}^{(2)}\xi-\pi_{j}^{(2)}\xi\Vert_{2}+\Vert\Psi_{j}^{(2)}\pi_{j}^{(2)}\xi-\xi\Vert_{2}\\
\nonumber
&<& (1+C)\vert F\vert^{-1}.
\end{eqnarray}
This implies that $\Psi_{\widetilde{j}(F)}^{(2)}\Phi_{\widetilde{j}(F),\widetilde{k}(F)}^{(2)}\pi_{\widetilde{j}(F)}^{(2)}\rightarrow1$
strongly in $\mathcal{B}(L^{2}(A,\varphi))$ where the net is indexed
by all finite subsets of $L^{2}(A,\varphi)$ partially ordered by
inclusion. Using once more Lemma \ref{Lem=SOTversusL2}, one sees that for $x\in A$
we have that $\Psi_{\widetilde{j}(F)}\circ\Phi_{\widetilde{j}(F), \widetilde{k}(F)}\circ\pi(x)\rightarrow x$
strongly. The claim follows.
\end{proof}

\vspace{1mm}


\section{Relative Haagerup property} \label{relhap}

In this section we introduce the relative Haagerup property for inclusions
of general $\sigma$-finite von Neumann algebras and consider natural
variations of the definition. For this, fix a triple $(M,N,\varphi)$
where $N\subseteq M$ is a unital inclusion of von Neumann
algebras and where $\varphi$ is a faithful normal positive functional
on $M$ whose corresponding modular automorphism group $(\sigma_{t}^{\varphi})_{t\in\mathbb{R}}$
satisfies $\sigma_{t}^{\varphi}(N)\subseteq N$ for all $t\in\mathbb{R}$.
To keep the notation short, we will often just write $(M,N,\varphi)$
and will implicitly assume that the triple satisfies the mentioned
conditions. By \cite[Theorem IX.4.2]{Takesaki2} the assumption $\sigma_{t}^{\varphi}(N)\subseteq N$,
$t\in\mathbb{R}$ is equivalent to the existence of a (uniquely determined)
$\varphi$-preserving (necessarily faithful) normal conditional expectation
$\mathbb{E}_{N}^{\varphi}:M\rightarrow N$. If the corresponding
functional $\varphi$ is clear, we will often just write $\mathbb{E}_{N}$
instead of $\mathbb{E}_{N}^{\varphi}$ (compare also with Subsection
\ref{dependence}).


\subsection{First definition of relative Haagerup property}

For a triple $(M,N,\varphi)$ as before define the \emph{Jones projection}
\[
e_{N}^{\varphi}:=\mathbb{E}_{N}^{(2)}:L^{2}(M,\varphi)\rightarrow L^{2}(M,\varphi)
\]
which is the orthogonal projection onto $L^{2}(N,\varphi)\subseteq L^{2}(M,\varphi)$
and let $\left\langle M,N\right\rangle \subseteq\mathcal{B}(L^{2}(M,\varphi))$
be the von Neumann subalgebra generated by $e_{N}^{\varphi}$ and
$M$. This is the \emph{Jones construction}. We will usually write
$e_{N}$ instead of $e_{N}^{\varphi}$ if there is no ambiguity. Further
set
\[
\mathcal{K}_{00}(M,N,\varphi):=\text{Span}\{xe_{N}y\mid x,y\in M\}\subseteq\mathcal{B}(L^{2}(M,\varphi))
\]
and
\[
\mathcal{K}(M,N,\varphi):=\overline{\mathcal{K}_{00}(M,N,\varphi)}\text{.}
\]
Then $\mathcal{K}_{00}(M,N,\varphi)$ is a (not necessarily closed)
two-sided ideal in $\left\langle M,N\right\rangle $ whose elements
are called the \emph{finite rank operators} relative to $N$. Similarly,
$\mathcal{K}(M,N,\varphi)$ is a closed two-sided ideal in $\left\langle M,N\right\rangle $
whose elements are called the \emph{compact operators} relative to
$N$. Note that if $N=\mathbb{C}1_{M}$, then $e_{N}$ is a rank one
projection and the operators in $\mathcal{K}(M,N,\varphi)$ are precisely
the compact operators on $L^{2}(M,\varphi)$.

\begin{rem}
In the following it is often convenient to identify
a finite rank operator $ae_{N}b\in\mathcal{K}_{00}(M,N,\varphi)$,
$a,b\in M$ with the map $a\mathbb{E}_{N}(b\:\cdot\:):M\rightarrow M$.
The latter does not depend on $\varphi$ (but only on the conditional
expectation $\mathbb{E}_{N}$), and the notation is naturally compatible with the inclusion $M \subset L^2(M, \varphi)$. We will often write $a\mathbb{E}_{N}b:=a\mathbb{E}_{N}(b\:\cdot\:)$.
\end{rem}

\begin{dfn} \label{Dfn=RHAPState}
Let $ N\subseteq M$
be a unital inclusion of von Neumann algebras and let $\varphi$ be
a faithful normal positive functional on $M$ with $\sigma_{t}^{\varphi}(N)\subseteq N$
for all $t\in\mathbb{R}$. We say that the triple $(M,N,\varphi)$
has the \emph{relative Haagerup property} (or just \emph{property (rHAP)})
if there exists a net $(\Phi_{i})_{i\in I}$ of   normal maps $\Phi_{i}:M\rightarrow M$
such that

\begin{enumerate}
\item  \label{Item=RHAP1} $\Phi_{i}$
is completely positive and $\sup_{i}\Vert\Phi_{i}\Vert<\infty$ for all $i \in I$;
\item  \label{Eqn=RHAP2} $\Phi_{i}$ is
an $N$-$N$-bimodule map for all $i \in I$;
\item  \label{Eqn=RHAP3} $\Phi_{i}(x)\rightarrow x$
strongly for every $x \in M$;
\item  \label{Eqn=RHAP4}  $\varphi\circ\Phi_{i}\leq\varphi$ for all $i \in I$;
\item  \label{Eqn=RHAP5} For every $i\in I$ the $L^{2}$-implementation
\[
\Phi_{i}^{(2)}:L^{2}(M,\varphi)\rightarrow L^{2}(M,\varphi), x\Omega_{\varphi}\mapsto\Phi_{i}(x)\Omega_{\varphi},
\]
 is contained in $\mathcal{K}(M,N,\varphi)$.
\end{enumerate}
\end{dfn}

\begin{rem} \label{Rmk=HAPforPosFunctionals} {(1)} In many applications
$\varphi$ will be a faithful normal state, but for notational convenience
we shall rather work in the more general setting. Note that we may
always normalize $\varphi$ to be a state and that the definition
of the relative Haagerup property does not change under this normalization.

\noindent
{(2)} Note that in \cite{Popa06} (see also \cite{JolissaintETDS}) a
different notion of relative compactness is used to define the relative
Haagerup property. It coincides with ours in case $N^{\prime}\cap M\subseteq N$.
However, the alternative notion is not very suitable beyond the tracial
situation since it requires the existence of finite projections; we will return to this issue in Subsection \ref{Subsec:Examples}.

\noindent
{(3)} In the case where $N=\mathbb{C}1_{M}$ Definition \ref{Dfn=RHAPState} recovers the usual definition of the (non-relative) Haagerup property, see \cite[Definition 3.1]{CS-IMRN}.
\end{rem}

There is a number of immediate variations of Definition \ref{Dfn=RHAPState}.
For instance, one may replace the condition \eqref{Item=RHAP1} by
one of the following stronger conditions:

\vspace{1mm}
\begin{itemize}
\item[(1$^{\prime}$)] For every $i\in I$ the map $\Phi_{i}$
is contractive completely positive.
\item[(1$^{\prime\prime}$)] For every $i\in I$ the map $\Phi_{i}$
is unital completely positive.
\end{itemize}

\noindent We may also replace the condition \eqref{Eqn=RHAP4} by the following
condition:

\begin{itemize}
\item[(4$^{\prime}$)] $\varphi\circ\Phi_{i}=\varphi$.
\end{itemize}
\vspace{1mm}
One of the results that we shall prove is that if the subalgebra $N$
is finite, then condition \eqref{Eqn=RHAP4} is redundant. We will
further prove that in this setting the approximating maps $\Phi_{i}$, $i\in I$ can
be chosen to be unital and state-preserving implying that all the
variations of the relative Haagerup property from above coincide.
To simplify the statements of the following sections, let us introduce
the following auxiliary notion, which is a priori weaker (see Section 2).

\begin{dfn}  \label{Dfn=RHAPState-}
Let $ N\subseteq M$
be a unital inclusion of von Neumann algebras and let $\varphi$ be
a faithful normal positive functional on $M$ with $\sigma_{t}^{\varphi}(N)\subseteq N$
for all $t\in\mathbb{R}$. We say that the triple $(M,N,\varphi)$
has \emph{property $\text{(rHAP)}^{-}$} if there exists a net $(\Phi_{i})_{i\in I}$
of normal maps $\Phi_{i}:M\rightarrow M$ such that

\begin{enumerate}
\item  \label{Item=RHAP1-} $\Phi_{i}$
is completely positive for all $i \in I$;
\item  \label{Eqn=RHAP2-} $\Phi_{i}$
is an $N$-$N$-bimodule map for all $i \in I$;
\item  \label{Eqn=RHAP3-} $\Vert\Phi_{i}(x)-x\Vert_{2,\varphi}\rightarrow0$ for every $x \in M$;
\item  \label{Eqn=RHAP4-} For every $i\in I$ the $L^{2}$-implementation
\[
\Phi_{i}^{(2)}\text{: }L^{2}(M,\varphi)\rightarrow L^{2}(M,\varphi), x\Omega_{\varphi}\mapsto\Phi_{i}(x)\Omega_{\varphi},
\]
exists and is contained in $\mathcal{K}(M,N,\varphi)$.
\end{enumerate}
\end{dfn}


\subsection{Dependence on the positive functional: reduction to the dependence
on the conditional expectation} \label{dependence}

Let $N\subseteq M$ be a unital inclusion of von Neumann
algebras which admits a faithful normal conditional expectation $\mathbb{E}_{N}:M\rightarrow N$.
Recall that for every faithful normal positive functional $\varphi$
on $M$ with $\varphi\circ\mathbb{E}_{N}=\varphi$ the corresponding
modular automorphism group $(\sigma_{t}^{\varphi})_{t\in\mathbb{R}}$
satisfies $\sigma_{t}^{\varphi}(N)\subseteq N$ for all $t\in\mathbb{R}$. Note that such a functional always exists, as it suffices to pick a faithful normal state $\omega\in N_*$ (which exists by our standing $\sigma$-finiteness assumption) and set $\varphi = \omega \circ \mathbb{E}_N$.
In this subsection we will examine the dependence of the relative
Haagerup property of $(M,N,\varphi)$ on the functional $\varphi$.
We shall prove that the property rather depends on the conditional
expectation $\mathbb{E}_{N}$ than on $\varphi$.

\begin{lem} \label{Lem=StatePreserving}
Let $\Phi:M\rightarrow M$
be $N$-$N$-bimodular. Then the following statements are equivalent:

\begin{enumerate}
\item \label{Item=Epreserving1} $\mathbb{E}_{N}\circ\Phi\leq\mathbb{E}_{N}$
(resp. $\mathbb{E}_{N}\circ\Phi=\mathbb{E}_{N}$).
\item \label{Item=Epreserving2} For all $\varphi\in M_{\ast}^{+}$
with $\varphi\circ\mathbb{E}_{N}=\varphi$ we have $\varphi\circ\Phi\leq\varphi$
(resp. $\varphi\circ\Phi=\varphi$).
\item \label{Item=Epreserving3} There exists a faithful functional
$\varphi\in M_{\ast}^{+}$ with $\varphi\circ\mathbb{E}_{N}=\varphi$
such that $\varphi\circ\Phi\leq\varphi$ (resp. $\varphi\circ\Phi=\varphi$).
\end{enumerate}
Further, the following statements are equivalent:
\begin{enumerate}\setcounter{enumi}{3}
\item \label{Item=Epreserving1-} There exists $C>0$
such that $\mathbb{E}_{N}(\Phi(x)^{\ast}\Phi(x))\leq C\mathbb{E}_{N}(x^{\ast}x)$
for all $x\in M$.
\item \label{Item=Epreserving2-} There exists $C>0$ such that for
all $\varphi\in M_{\ast}^{+}$ with $\varphi\circ\mathbb{E}_{N}=\varphi$
and $x\in M$ we have $\varphi(\Phi(x)^{\ast}\Phi(x))\leq C\varphi(x^{\ast}x)$.
\item \label{Item=Epreserving3-} There exists $C>0$ and a faithful functional
$\varphi\in M_{\ast}^{+}$ with $\varphi\circ\mathbb{E}_{N}=\varphi$
such that for all $x\in M$ we have $\varphi(\Phi(x)^{\ast}\Phi(x))\leq C\varphi(x^{\ast}x)$.
\end{enumerate}
In particular, if the $L^{2}$-implementation of $\Phi$ with respect
to $\varphi$ exists, then it exists with respect to any other $\psi$
with $\psi\circ\mathbb{E}_{N}=\psi$.
\end{lem}

\begin{proof}
We prove the statements for the inequalities; the
respective cases with equalities follow similarly. The implications \eqref{Item=Epreserving1} $\Leftrightarrow$
\eqref{Item=Epreserving2} $\Rightarrow$ \eqref{Item=Epreserving3}
of the first three statements are trivial. For the implication \eqref{Item=Epreserving3}
$\Rightarrow$ \eqref{Item=Epreserving1} take $\varphi$ as in \eqref{Item=Epreserving3}.
For $x\in N$ consider the positive functional $x^{\ast}\varphi x\in M_{\ast}^{+}$
which again satisfies $(x^{\ast}\varphi x)\circ\mathbb{E}_{N}=x^{\ast}\varphi x$.
Then, for $y\in M^{+}$
\[
(x^{\ast}\varphi x)\circ\mathbb{E}_{N}\circ\Phi(y)=\varphi\circ\mathbb{E}_{N}\circ\Phi(xyx^{\ast})\leq\varphi\circ\mathbb{E}_{N}(xyx^{\ast})=(x^{\ast}\varphi x)\circ\mathbb{E}_{N}(y)\text{.}
\]
 Since the restrictions of functionals $x^{\ast}\varphi x$, $x\in N$
to $N$ are dense in $N_{\ast}^{+}$ we conclude that $\mathbb{E}_{N}\circ\Phi\leq\mathbb{E}_{N}$.

The equivalence of the statements \eqref{Item=Epreserving1-}, \eqref{Item=Epreserving2-} and \eqref{Item=Epreserving3-} follows in a similar
way.
\end{proof}

The following lemma  shows that in good circumstances compactness of the $L^2$-implementations does not depend on the choice of the state.

\begin{lem} \label{Lem=StandardForm}
Let $\varphi,\psi\in M_{\ast}^{+}$ be faithful with
$\varphi\circ\mathbb{E}_{N}=\varphi$ and $\psi\circ\mathbb{E}_{N}=\psi$.
Let further $\Phi:M\rightarrow M$ be a completely positive $N$-$N$-bimodule
map whose $L^{2}$-implementation $\Phi_{\varphi}^{(2)}$ with respect
to $\varphi$ exists (hence, by Lemma \ref{Lem=StatePreserving},
the $L^{2}$-implementation $\Phi_{\psi}^{(2)}$ of $\Phi$ with respect
to $\psi$ exists as well). Then, $\Phi_{\varphi}^{(2)}\in\mathcal{K}(M,N,\varphi)$
if and only if $\Phi_{\psi}^{(2)}\in\mathcal{K}(M,N,\psi)$.
\end{lem}

\begin{proof}
Let $U$ be the unique unitary mapping the standard
form $(M,L^{2}(M,\varphi),J_{\varphi},P_{\varphi})$ to the standard
form $(M,L^{2}(M,\psi),J_{\psi},P_{\psi})$, see \cite[Theorem
2.3]{Haagerup75}. It restricts to the unique unitary
map between the standard forms $(N,L^{2}(N,\varphi),J_{\varphi|_{N}},P_{\varphi|_{N}})$
and $(N,L^{2}(N,\psi),J_{\psi|_{N}},P_{\psi|_{N}})$. Indeed,
for all $x\in M$, $\varphi(x)=\langle xU\Omega_{\varphi},U\Omega_{\varphi}\rangle$
and by \cite[Lemma 2.10]{Haagerup75}, $U\Omega_{\varphi}$
is the unique element in $L^{2}(M,\psi)$ satisfying this equation.
On the other hand, applying \cite[Lemma 2.10]{Haagerup75}
to $\varphi|_{N}$ implies the existence of a unique vector $\xi\in L^{2}(N,\psi)$
such that $\varphi(x)=\langle x\xi,\xi\rangle$ for all $x\in N$.
By approximating $\xi$ by elements in $N\Omega_{\psi}$ and by using
the assumptions $\varphi\circ\mathbb{E}_{N}=\varphi$ and $\psi\circ\mathbb{E}_{N}=\psi$
one deduces that for all $x\in M$,
\[
\varphi(x)=\varphi\circ\mathbb{E}_{N}(x)=\langle\mathbb{E}_{N}(x)\xi,\xi\rangle=\langle x\xi,\xi\rangle
\]
and hence $U\Omega_{\varphi}=\xi\in L^{2}(N,\psi)$. This implies
$U(N\Omega_{\varphi})\subseteq L^{2}(N,\psi)$ and therefore (by density
and symmetry) $U(L^{2}(N,\psi))=L^{2}(N,\psi)$. Finally, using that
$\sigma_{t}^{\varphi}(N)=N$ and hence $J_{\varphi|_{N}}=(J_{\varphi})|_{L^{2}(N,\varphi)}$
(and similarly $J_{\psi|_{N}}=(J_{\psi})|_{L^{2}(N,\psi)}$) it is
straightforward to check that the restriction of $U$ satisfies the
other properties of the unique unitary mapping between the standard
forms $(N,L^{2}(N,\varphi),J_{\varphi|_{N}},P_{\varphi|_{N}})$ and
$(N,L^{2}(N,\psi),J_{\psi|_{N}},P_{\psi|_{N}})$.

Since $e_{N}^{\varphi}=(\mathbb{E}_{N}^{\varphi})^{(2)}$ is the orthogonal
projection of $L^{2}(M,\varphi)$ onto $L^{2}(N,\varphi)$ and $e_{N}^{\psi}=(\mathbb{E}_{N}^{\psi})^{(2)}$
is the orthogonal projection of $L^{2}(M,\psi)$ onto $L^{2}(N,\psi)$,
we see that $U^{\ast}e_{N}^{\psi}U=e_{N}^{\varphi}$. Hence, for every
map $\Lambda$ of the form $\Lambda=a\mathbb{E}_{N}b$ with $a,b\in M$
the $L^{2}$-implementation $\Lambda_{\varphi}^{(2)}$ with respect
to $\varphi$ and the $L^{2}$-implementation $\Lambda_{\psi}^{(2)}$
with respect to $\psi$ exist with $\Lambda_{\varphi}^{(2)}=ae_{N}^{\varphi}b=U^{\ast}ae_{N}^{\psi}bU=U^{\ast}\Lambda_{\psi}^{(2)}U$.

 Now assume that $\Phi_{\varphi}^{(2)}\in\mathcal{K}(M,N,\varphi)$.
Then there exists a sequence $(\Phi_{k}:M\rightarrow M)_{k\in\mathbb{N}}$
of maps of the form $\Phi_{k}=\sum_{i=1}^{N_{k}}a_{i,k}\mathbb{E}_{N}b_{i,k}$
with $N_{k}\in\mathbb{N}$ and $a_{1,k},b_{1,k},...,a_{N_{k},k},b_{N_{k},k}\in M$
whose $L^{2}$-implementations $\Phi_{k,\varphi}^{(2)}\in\mathcal{K}_{00}(M,N,\varphi)$
(with respect to $\varphi$) norm-converge to $\Phi_{\varphi}^{(2)}$.
By the above, $U\Phi_{k,\varphi}^{(2)}U^{\ast}\in\mathcal{K}_{00}(M,N,\psi)$
is given by $x\Omega_{\psi}\mapsto\Phi_{k}(x)\Omega_{\psi}$ for $x\in M$.
We claim that the sequence $(U\Phi_{k,\varphi}^{(2)}U^{\ast})_{k\in\mathbb{N}}\subseteq\mathcal{K}_{00}(M,N,\psi)$
norm-converges to $\Phi_{\psi}^{(2)}$. Indeed, by the density of the set of all elements of the form
$x(\varphi|_{N})x^{\ast}$, $x \in N$ in $N_{\ast}^{+}$ we find
a net $(x_{i})_{i\in I}\subseteq N$ such that $x_{i}(\varphi|_{N})x_{i}^{\ast}\rightarrow\psi|_{N}$.
In combination with $\varphi\circ\mathbb{E}_{N}=\varphi$ and $\psi\circ\mathbb{E}_{N}=\psi$
this also implies $x_{i}\varphi x_{i}^{\ast}\rightarrow\psi$. For
$y\in M$ and $k \in \mathbb{N}$,
\begin{eqnarray}
\nonumber
\|(\Phi_{\psi}^{(2)}-U\Phi_{k,\varphi}^{(2)}U^{\ast})y\Omega_{\psi}\|_{2,\psi}^{2} &=& \Vert(\Phi(y)-\Phi_{k}(y))\Omega_{\psi}\Vert_{2,\psi}^{2} \\
\nonumber
&=& \lim_{i}\|(\Phi(y)-\Phi_{k}(y))x_{i}\Omega_{\varphi}\|_{2,\varphi}^{2}\\
\nonumber
 & = & \lim_{i}\|(\Phi_{\varphi}^{(2)}-\Phi_{k,\varphi}^{(2)})yx_{i}\Omega_{\varphi}\|_{2,\varphi}^{2} \\
\nonumber
&\leq& \lim_{i}\|\Phi_{\varphi}^{(2)}-\Phi_{k,\varphi}^{(2)}\|^{2}\varphi(x_{i}^{*}y^{*}yx_{i})\\
\nonumber
&=& \|\Phi_{\varphi}^{(2)}-\Phi_{k,\varphi}^{(2)}\|^2 \psi(y^{\ast}y),
\end{eqnarray}
where in the third step we used the $N$-$N$-bimodularity of $\Phi$ and the right $N$-modularity of $\Phi_k$. Now, $\Phi_{k,\varphi}^{(2)}\rightarrow\Phi_{\varphi}^{(2)}$ and
$(U\Phi_{k,\varphi}^{(2)}U^{\ast})_{k\in\mathbb{N}}$ is a Cauchy
sequence, hence the above inequality leads to $U\Phi_{k,\varphi}^{(2)}U^{\ast}\rightarrow\Phi_{\psi}^{(2)}$
as claimed. In particular, $\Phi_{\psi}^{(2)}\in\mathcal{K}(M,N,\psi)$ which finishes the proof.
\end{proof}

\begin{thm}  \label{StateIndependence}
Let $ N\subseteq M$
be a unital inclusion of von Neumann algebras which admits a faithful
normal conditional expectation $\mathbb{E}_{N}:M\rightarrow N$.
Further, let $\varphi,\psi\in M_{\ast}^{+}$ be faithful normal positive functionals with $\varphi\circ\mathbb{E}_{N}=\varphi$
and $\psi\circ\mathbb{E}_{N}=\psi$. Then the triple $(M,N,\varphi)$
has property $\text{(rHAP)}$ (resp. property $\text{(rHAP)}^{-}$)
if and only if the triple $(M,N,\psi)$ has property $\text{(rHAP)}$
(resp. property $\text{(rHAP)}^{-}$). In particular, property $\text{(rHAP)}$
(resp. property $\text{(rHAP)}^{-}$) only depends on the triple $(M,N,\mathbb{E}_{N})$.
\end{thm}

\begin{proof}
It follows from Lemma \ref{Lem=StatePreserving} and
Lemma \ref{Lem=StandardForm} that if $(\Phi_{j})_{j\in J}$ is a
net of approximating maps witnessing the relative Haagerup property
of $(M,N,\varphi)$ (resp. property $\text{(rHAP)}^{-}$ of $(M,N,\varphi)$),
then it also witnesses the Haagerup property of $(M,N,\psi)$ (resp.
property $\text{(rHAP)}^{-}$ of $(M,N,\psi)$) and vice versa.
\end{proof}

We will later see that in the case where the von Neumann subalgebra
$N$ is finite the statement in Theorem \ref{StateIndependence}
can be strengthened: in this case property $\text{(rHAP)}$ (and equivalently
property $\text{(rHAP)}^{-}$) does not even depend on the choice
of the conditional expectation $\mathbb{E}_{N}$.

Motivated by Theorem \ref{StateIndependence} we introduce the following
natural definition.

\begin{dfn}
Let $N\subseteq M$ be a unital
inclusion of von Neumann algebras which admits a faithful normal conditional
expectation $\mathbb{E}_{N}:M\rightarrow N$. We say that
the triple $(M,N,\mathbb{E}_{N})$ has the \emph{relative Haagerup property}
(or just \emph{property (rHAP)}) if $(M,N,\varphi)$ has the relative
Haagerup property for some (equivalently any) faithful normal positive
functional $\varphi\in M_{\ast}^+$ with $\varphi \circ \mathbb{E}_N=\varphi$. The same terminology shall be adopted for property (rHAP)$^-$.
\end{dfn}


\subsection{State preservation, contractivity and unitality of the approximating
maps in a special case} \label{SpecialCase}

In this subsection we will prove that the relative Haagerup property
of certain triples $(M,N,\mathbb{E}_{N})$ may be witnessed by approximating
maps that satisfy extra conditions, such as state-preservation, contractivity and unitality. This will play a crucial role in Section \ref{MainResults}.
The approach is inspired by \cite[Section 2]{BannonFang}, where
ideas from \cite{Jolissaint} were used.

\begin{lem} \label{Lem=Commute}
Let $M$ be a von Neumann algebra,
$\varphi\in M_{\ast}$ a faithful normal state and $y\in M$. If $y\varphi=\varphi y$ (i.e.\ $y \in M^\varphi$),
then $y\Omega_{\varphi}=\Omega_{\varphi}y$.
\end{lem}

\begin{proof}
As mentioned in Section \ref{Sect=Modular}, by \cite[Theorem VIII.2.6]{Takesaki2} we have that $\sigma_{t}^{\varphi}(y)=y$
for all $t\in\mathbb{R}$. But then $y$ is analytic and moreover $\sigma_{-i/2}^{\varphi}(y)=y$.
Hence
\[
\Omega_{\varphi}y=J_{\varphi}y^{\ast}J_{\varphi}\Omega_{\varphi}=J_{\varphi}\sigma_{-i/2}^{\varphi}(y^{\ast})J_{\varphi}\Omega_{\varphi}=J_{\varphi}\Delta_{\varphi}^{1/2}y^{\ast}\Delta_{\varphi}^{-1/2}J_{\varphi}\Omega_{\varphi}=S_{\varphi}y^{\ast}S_{\varphi}\Omega_{\varphi}=y\Omega_{\varphi}\text{.}
\]
The claim follows.
\end{proof}

\begin{prop} \label{Prop=BannonFang}
Let $ N\subseteq M$
be a unital inclusion of von Neumann algebras that admits a faithful
normal conditional expectation $\mathbb{E}_{N}$. Assume that $N$
is finite and let $\tau\in N_{\ast}$ be a faithful normal tracial
state that we extend to a state $\varphi:=\tau\circ\mathbb{E}_{N}$
on $M$. Let further $\Phi:M\rightarrow M$ be a normal, completely positive, $N$-$N$-bimodular
map for which there exists $\delta>0$ such that $c:=\Phi(1)\leq1-\delta$
and $\varphi\circ\Phi\leq(1-\delta)\varphi$. Then one can find $a,b\in N^{\prime}\cap M$
such that $a \geq 0$, $\mathbb{E}_{N}(a)=1$, $a\mathbb{E}_{N}(b^{\ast}b)=\mathbb{E}_{N}(b^{\ast}b)a=1-c$
and $b\varphi b^{\ast}=\varphi-\varphi\circ\Phi$.
\end{prop}

\begin{proof}
The complete positivity of $\Phi$ implies that   $0\leq\left\Vert \Phi\right\Vert = \|\Phi(1)\| = \|c\| \leq 1-\delta$,
hence the map $\Phi$ must be contractive. It is clear that $c = \Phi(1) \geq 0$. Further,
since $\mathbb{E}_{N}(1-c)>\mathbb{E}_{N}(\delta)=\delta$, the element
$\mathbb{E}_{N}(1-c)\in N$ is boundedly invertible.  Additionally, the $N$-$N$-bimodularity of
$\Phi$ implies that for every $n\in N$,
\[
nc=n\Phi(1)=\Phi(n)=\Phi(1)n=cn\text{,}
\]
so $c\in N^{\prime}\cap M$. The latter two observations imply that for
\[
a:=(1-c)(\mathbb{E}_{N}(1-c))^{-1}
\]
we have $a \in N' \cap M$, $a \geq 0$ and $\mathbb{E}_{N}(a)=1$.

Consider the positive normal functional $\varphi-\varphi\circ\Phi\in M_{\ast}$.
By \cite[Lemma 2.10]{Haagerup75} there exists a unique vector $\xi\in L^{2}(M,\varphi)^{+}$
such that $(\varphi-\varphi\circ\Phi)(x)=\left\langle x\xi,\xi\right\rangle $
for all $x\in M$. Note that $\{J_{\varphi}x\Omega_{\varphi}\mid x\in M\}$ is dense
in $L^{2}(M,\varphi)$ and define the linear map
\[
b:L^{2}(M,\varphi)\rightarrow L^{2}(M,\varphi)\text{, }J_{\varphi}x\Omega_{\varphi}\mapsto J_{\varphi}x\xi\text{.}
\]
It is contractive since
\[
\Vert b(J_{\varphi}x\Omega_{\varphi})\Vert_{2}^{2}=\Vert J_{\varphi}x\xi\Vert_{2}^{2}=\Vert x\xi\Vert_{2}^{2}=(\varphi-\varphi\circ\Phi)(x^{\ast}x)\leq\varphi(x^{\ast}x)=\Vert x\Omega_{\varphi}\Vert_{2}^{2} = \Vert J_\varphi x\Omega_{\varphi}\Vert_{2}^{2}
\]
for all $x\in M$. Further, for $x,y\in M$,
\[
bJ_{\varphi}xJ_{\varphi}(J_{\varphi}y\Omega_{\varphi})=bJ_{\varphi}xy\Omega_{\varphi}=J_{\varphi}xy\xi=J_{\varphi}xJ_{\varphi}J_{\varphi}y\xi=J_{\varphi}xJ_{\varphi}b(J_{\varphi}y\Omega_{\varphi})\text{.}
\]
It hence follows that $b$ and $J_\varphi x J_\varphi$ commute and therefore that $b\in(J_\varphi M  J_\varphi)^{\prime}=M^{\prime\prime}=M$.\\

We claim that $a$ and $b$ from above satisfy the required conditions.
It remains to show that $b\in N^{\prime}\cap M$, \emph{$b\varphi b^{\ast}=\varphi-\varphi\circ\Phi$}
and $a\mathbb{E}_{N}(b^{\ast}b)=\mathbb{E}_{N}(b^{\ast}b)a=1-c$ .
\begin{itemize}
\item \emph{$b\in N^{\prime}\cap M$:} By the assumption we have $\varphi-\varphi\circ\Phi\geq\delta\varphi$
and therefore $\varphi-\varphi\circ\Phi$ is a faithful normal functional.
For $x\in M$, $n\in N$ the $N$-$N$-bimodularity of $\Phi$ and
the traciality of $\varphi$ on $N$ (implying that $n$ is contained
in the centralizer of $\varphi$) imply that $\varphi\circ\Phi(xn)=\varphi(\Phi(x)n)=\varphi(n\Phi(x))=\varphi\circ\Phi(nx)$,
hence $n(\varphi-\varphi\circ\Phi)=(\varphi-\varphi\circ\phi)n$. The unique isomorphism between the standard forms induced by $\varphi$ and $\varphi -\varphi\circ\Phi$ maps $\xi$ to the canonical cyclic vector in $L^2(M,\varphi-\varphi\circ \Phi)$.
Hence, from Lemma \ref{Lem=Commute} applied to $\varphi-\varphi\circ\Phi$
we get $n\xi=\xi n$ for all $n\in N$,
which, together with the fact that $J_{\varphi}n\Omega_{\varphi} = n^*\Omega_{\varphi}$, implies that for $x\in M$
\begin{eqnarray}
\nonumber
bn(J_{\varphi}x\Omega_{\varphi})&=&bJ_{\varphi}xJ_{\varphi}n\Omega_{\varphi}=bJ_{\varphi}xn^{\ast}\Omega_{\varphi}=J_{\varphi}xn^{\ast}\xi\\
\nonumber
&=&J_{\varphi}x\xi n^{\ast}=J_{\varphi}xJ_{\varphi}nJ_{\varphi}\xi=nJ_{\varphi}x\xi=nb(J_{\varphi}x\Omega_{\varphi})\text{,}
\end{eqnarray}
so $b\in N^{\prime}\cap M$ by the density of $\{J_{\varphi}x\Omega_{\varphi}\mid x\in M\}$
in $L^{2}(M,\varphi)$.
\item \emph{$b\varphi b^{\ast}=\varphi-\varphi\circ\Phi$:} For every $x\in M$
the equality
\[
(b\varphi b^{\ast})(x)=\left\langle xb\Omega_{\varphi},b\Omega_{\varphi}\right\rangle =\left\langle x\xi,\xi\right\rangle =(\varphi-\varphi\circ\Phi)(x)
\]
holds, i.e. $b\varphi b^{\ast}=\varphi-\varphi\circ\Phi$.
\item \emph{$a\mathbb{E}_{N}(b^{\ast}b)=\mathbb{E}_{N}(b^{\ast}b)a=1-c$:}
For $x\in M$ we find by $b\in N^{\prime}\cap M$ and $b\varphi b^{\ast}=\varphi-\varphi\circ\Phi$
that
\begin{eqnarray}
\nonumber
& & \varphi(x\mathbb{E}_{N}(b^{\ast}b))=\varphi(\mathbb{E}_{N}(x)b^{\ast}b)=\varphi(b^{\ast}\mathbb{E}_{N}(x)b)=(\varphi-\varphi\circ\Phi)(\mathbb{E}_{N}(x))\\
\nonumber
&=& \varphi(\mathbb{E}_{N}(x))-\varphi(\mathbb{E}_{N}(x)\Phi(1))=\varphi(x)-\varphi(x\mathbb{E}_{N}(\Phi(1)))=\varphi(x\mathbb{E}_{N}(1-c))
\end{eqnarray}
and hence $\mathbb{E}_{N}(1-c)=\mathbb{E}_{N}(b^{\ast}b)$. It follows
by the definition of $a$ that $a\mathbb{E}_{N}(b^{\ast}b)=a\mathbb{E}_{N}(1-c)=1-c$
and similarly,  as  $a \in N' \cap M$, we have  $\mathbb{E}_{N}(b^{\ast}b)a=1-c$.
\end{itemize}
\end{proof}

\begin{lem} \label{Lem=TomTakExp}
Let $N\subseteq M$
be a unital inclusion of von Neumann algebras which admits a faithful
normal conditional expectation $\mathbb{E}_{N}$. Assume that $N$
is finite and let $\tau\in N_{\ast}$ be a faithful normal tracial
state that we extend to a state $\varphi:=\tau\circ\mathbb{E}_{N}$
on $M$. Let further $x\in N^{\prime}\cap M$ be an element which
is analytic for $\sigma^{\varphi}$. Then $\mathbb{E}_{N}(yx)=\mathbb{E}_{N}(\sigma_{i}^{\varphi}(x)y)$
for all $y\in M$.
\end{lem}

\begin{proof}
For $n\in N$ we have by the traciality of $\varphi$
on $N$ (implying that $n$ is contained in the centralizer of
$\varphi$) that
\[
n\sigma_{z}^{\varphi}(x)=\sigma_{z}^{\varphi}(n)\sigma_{z}^{\varphi}(x)=\sigma_{z}^{\varphi}(nx)=\sigma_{z}^{\varphi}(xn)=\sigma_{z}^{\varphi}(x)\sigma_{z}^{\varphi}(n)=\sigma_{z}^{\varphi}(x)n\text{.}
\]
for all $z\in\mathbb{C}$. Therefore, $\sigma_{z}^{\varphi}(x)\in N^{\prime}\cap M$
and in particular $\sigma_{i}^{\varphi}(x)\in N^{\prime}\cap M$.
One further calculates that for $y\in M$,
\begin{eqnarray}
\nonumber
(\varphi n)(\mathbb{E}_{N}(yx))&=&\varphi(\mathbb{E}_{N}(nyx))=\varphi(nyx)=\varphi(\sigma_{i}^{\varphi}(x)ny)\\
\nonumber
&=&  \varphi(n\sigma_{i}^{\varphi}(x)y)=(\varphi n)(\mathbb{E}_{N}(\sigma_{i}^{\varphi}(x)y))\text{.}
\end{eqnarray}
Since the set of all functionals of the form $\varphi n$, $n\in N$
is dense in $N_{\ast}$ we find that $\mathbb{E}_{N}(yx)=\mathbb{E}_{N}(\sigma_{i}^{\varphi}(x)y)$,
as claimed.
\end{proof}

\begin{lem}\label{Lem=CompactMaps}
 Let $ N\subseteq M$
be a unital inclusion of von Neumann algebras that admits a faithful
normal conditional expectation $\mathbb{E}_{N}$. Assume that $N$
is finite and let $\tau\in N_{\ast}$ be a faithful normal tracial
state that we extend to a state $\varphi:=\tau\circ\mathbb{E}_{N}$
on $M$.
Let $h_1, h_2 \in M$ and let $h_3, h_4 \in N'\cap M$  be analytic for $\sigma^\varphi$. Suppose that $\Phi:M \to M$ is a normal map such that $\Phi^{(2)}_\varphi \in \cK(M, N, \varphi)$ and define the map $\tilde{\Phi}:=h_1 \Phi( h_2 \: \cdot \: h_3) h_4$. Then we also have that   $\tilde{\Phi}^{(2)}_\varphi \in  \cK(M, N, \varphi)$.
\end{lem}
\begin{proof}

Note first that by \cite[VIII.3.18(i)]{Takesaki2}  (and its proof) $\tilde{\Phi}$ has a bounded $L^2$-implementation, with $\|\tilde{\Phi}^{(2)}_\varphi\|\leq C \|\Phi^{(2)}_\varphi\| $, with the constant $C>0$ depending on $h_1, h_2, h_3, h_4$.
It thus suffices to show that the passage $\Phi \to \tilde{\Phi}$ preserves the property of having a finite-rank implementation.
Let then $a,b \in M$ so that $a e_N b$ is in $\cK_{00}(M, N, \varphi)$. So for $x \in M$ we have by Lemma \ref{Lem=TomTakExp},
\[
 h_1 (a e_N b  )( h_2 x h_3) h_4  \Omega_\varphi=  h_1  a h_4 e_N (  \sigma^\varphi_i(h_3) b  h_2 x ) \Omega_\varphi,
\]
and so this map is in $\cK_{00}(M, N, \varphi)$. 
\end{proof}

We are now ready to formulate the main result of this subsection.
In combination with Lemma \ref{Lem=ContractiveMaps} it will later
allow us to deduce that the relative Haagerup property of a triple
$(M,N,\mathbb{E}_{N})$ with finite $N$ may be witnessed by unital
and state-preserving maps. Its proof is inspired by \cite[Section
2]{BannonFang}.

\begin{thm} \label{BannonFangApproach}
Let $N\subseteq M$
be a unital inclusion of von Neumann algebras which admits a faithful
normal conditional expectation $\mathbb{E}_{N}$. Assume that $N$
is finite, let $\tau\in N_{\ast}$ be a faithful normal tracial state
that we extend to a state $\varphi:=\tau\circ\mathbb{E}_{N}$ on $M$
and suppose that the triple $(M,N,\varphi)$ has property $\text{(rHAP)}$
witnessed by contractive approximating maps. Then,  if all the elements of $M$ are analytic with respect to the modular automorphism group of $\varphi$ -- for example if there exists
a boundedly invertible element $h\in M^{+}$ with $\sigma_{t}^{\varphi}(x)=h^{it}xh^{-it}$
for all $t\in\mathbb{R}$, $x \in M$, property $\text{(rHAP)}$ of $(M,N,\varphi)$
may be witnessed by unital and state-preserving approximating maps,
i.e.\ we may assume that ($1^{\prime\prime}$) and ($4^{\prime}$)
hold.
\end{thm}

\begin{proof}

Let $(\Phi_{j})_{j\in J_1}$ be a net of contractive
approximating maps witnessing property $\text{(rHAP)}$ of the triple
$(M,N,\varphi)$ and choose a net $(\delta_{j})_{j\in J_2}$ with $\delta_{j}\rightarrow0$. We now set $J = J_1 \times J_2$ with the product partial order and for $j = (j_1, j_2) \in J$ we set $\Phi_j = \Phi_{j_1}$ and $\delta_j = \delta_{j_2}$.
Then for all $j \in J$,
\begin{eqnarray}
\nonumber
c_{j}:=(1-\delta_{j})\Phi_{j}(1)\leq1-\delta_{j} \qquad \text{and} \qquad (1-\delta_{j})(\varphi\circ\Phi_{j})\leq(1-\delta_{j})\varphi.
\end{eqnarray}
 In particular, we may apply Proposition \ref{Prop=BannonFang}
to $(1-\delta_{j})\Phi_{j}$ to find elements $a_{j},b_{j}\in N^{\prime}\cap M$
with $a_j \geq 0$, $\mathbb{E}_{N}(a_{j})=1$, $a_{j}\mathbb{E}_{N}(b_{j}^{\ast}b_{j})=\mathbb{E}_{N}(b_{j}^{\ast}b_{j})a_{j}=1-c_{j}$ and $b_{j}\varphi b_{j}^{\ast}=\varphi-(1-\delta_{j})(\varphi\circ\Phi_{j})$. For $j\in J$ define
\[
\Psi_{j}:M\rightarrow M\text{, }\Psi_{j}(x):=(1-\delta_{j})\Phi_{j}(x)+a_{j}\mathbb{E}_{N}(b_{j}^{\ast}xb_{j})\text{.}
\]
It is clear that $\Psi_{j}$ is normal completely positive and $N$-$N$-bimodular.
Further,
\[
\Psi_{j}(1)=(1-\delta_{j})\Phi_{j}(1)+a_{j}\mathbb{E}_{N}(b_{j}^{\ast}b_{j})=c_{j}+(1-c_{j})=1
\]
and for any $x \in M$
\begin{eqnarray}
\nonumber
\varphi\circ\Psi_{j}(x)&=& (1-\delta_{j})\varphi(\Phi_{j}(x))+\varphi(a_{j}\mathbb{E}_{N}(b_{j}^{\ast}xb_{j}))\\
\nonumber
&=& (1-\delta_{j})\varphi(\Phi_{j}(x))+\varphi(\mathbb{E}_{N}(a_{j})b_{j}^{\ast}xb_{j})\\
\nonumber
&=& (1-\delta_{j})\varphi(\Phi_{j}(x))+(b_{j}\varphi b_{j}^{\ast})(x)\\
\nonumber
&=& (1-\delta_{j})\varphi(\Phi_{j}(x))+\varphi(x)-(1-\delta_{j})\varphi(\Phi_{j}(x))\\
\nonumber
&=& \varphi(x)\text{,}
\end{eqnarray}
so the $\Psi_{j}$ are unital and $\varphi$-preserving.\\

For the relative compactness note that by the assumption 
that every element in $M$ is analytic for $\sigma^{\varphi}$,
Lemma \ref{Lem=TomTakExp} implies that for all $x\in M$
\[
\Psi_{j}(x)=(1-\delta_{j})\Phi_{j}(x)+a_{j}\mathbb{E}_{N}(\sigma_{i}^{\varphi}(b_{j})b_{j}^{\ast}x),
\]
hence,
\[
\Psi_{j}^{(2)}=(1-\delta_{j})\Phi_{j}^{(2)}+a_{j}e_{N}\sigma_{i}^{\varphi}(b_{j})b_{j}^{\ast}\in\mathcal{K}(M,N,\varphi).
\]
It remains to show that for every $x\in M$, $\Psi_{j}(x)\rightarrow x$
strongly. For this, estimate for $x\geq0$,
\begin{eqnarray}
\nonumber
(\Psi_{j}-  (1-\delta_j) \Phi_{j})(x) &=& a_{j}^{1/2}\mathbb{E}_{N}(b_{j}^{\ast}xb_{j})a_{j}^{1/2} \\
\nonumber
&\leq& \Vert x\Vert a_{j}^{1/2}\mathbb{E}_{N}(b_{j}^{\ast}b_{j})a_{j}^{1/2}\\
\nonumber
&=& \Vert x\Vert(1-c_{j})\text{.}
\end{eqnarray}
Since $c_{j}=(1-\delta_{j})\Phi_{j}(1)\rightarrow1$ and $ (1-\delta_j)  \Phi_{j}(x)\rightarrow x$
strongly it then follows that
\[
\Psi_{j}(x)=(\Psi_{j}(x)-  (1-\delta_j)  \Phi_{j}(x))+ (1-\delta_j)  \Phi_{j}(x)\rightarrow x
\]
strongly for every $x\in M$. This completes the proof.
\end{proof}

Another important statement that was proved in \cite{CS-CMP} in case
of the usual (non-relative) Haagerup property is the following lemma.
It will later ensure the contractivity of certain approximating
maps and allow us to apply Theorem \ref{BannonFangApproach} in a
suitable setting.

\begin{lem} \label{Lem=ContractiveMaps}
Let $M$ be a finite
von Neumann algebra equipped with a faithful normal tracial state
$\tau\in M_{\ast}$ and let $ N\subseteq M$ be a unital inclusion.
Assume that $h \in N^{\prime}\cap M$ is a boundedly invertible
self-adjoint element and define $\varphi\in M_{\ast}$ by $\varphi(x):=\tau(hxh)$
for $x\in M$. Then, if $(M,N,\varphi)$ has property $\text{(rHAP)}$,
the approximating maps $(\Phi_{i})_{i\in I}$ witnessing property
$\text{(rHAP)}$ may be chosen contractively, i.e.\ we may assume that
($1^{\prime}$) holds.
\end{lem}

\begin{proof}
The proof is given in \cite[Lemma 4.3]{CS-CMP}.
One only needs to check that the condition $h \in N^{\prime}\cap M^{+}$
ensures that the maps $\Phi_{k}^{\prime}$, $\Phi_{k}^{l}$ and $\Psi_{j}$
defined there are $N$-$N$-bimodule maps that are compact relative
to $N$.
 Let us comment on this.

In Step 1 of the proof of \cite[Lemma 4.3]{CS-CMP} it is shown that the approximating maps $\Phi_k$ witnessing the Haagerup property maybe chosen such that $\sup_k \Vert \Phi_k \Vert < \infty$. In the current setup of  $\text{(rHAP)}$  this is automatic (see Definition \ref{Dfn=RHAPState}) and so we may skip this step.

We now turn to Step 2 in the proof of \cite[Lemma 4.3]{CS-CMP}. Let $\Phi_k$ be the approximating maps witnessing the $\text{(rHAP)}$ for $(M, N, \varphi)$. In particular $\Phi_k$ is $N$-$N$-bimodular and $\Phi_k^{(2)} \in \cK(M, N, \varphi)$.
 By \cite[Theorem VIII.2.211]{Takesaki2} we have $\sigma_t^\varphi(x) = h^{it} x h^{-it}$, $t \in \mathbb{R}, x \in M$. Now recall the map defined in  \cite[Lemma 4.3]{CS-CMP} given by
\begin{equation}\label{Eqn=IntegralMap}
\begin{split}
\Phi_k^l(x) = &  \sqrt{   \frac{1}{l \pi }  }  \int_{- \infty}^{\infty}  e^{-t^2/l}  \sigma_t^\varphi( \Phi_k (   \sigma_{-t}^\varphi(x)  )    ) dt \\
= &  \sqrt{   \frac{1}{l \pi }  }  \int_{- \infty}^{\infty}  e^{-t^2/l}   h^{it}  \Phi_k (   h^{-it} x  h^{it}  )  h^{-it}  dt.
\end{split}
\end{equation}
Since $h \in N' \cap M$ this map is $N$-$N$-bimodular. Since $\sigma_t^{\varphi}(h^{is}) = h^{is}, s,t \in \mathbb{R}$ it follows from Lemma \ref{Lem=CompactMaps} that the $L^2$-implementation of
\begin{equation}\label{Eqn=IntegrantMap}
x \mapsto  \sigma_t^\varphi( \Phi_k (   \sigma_{-t}^\varphi(x)  )    ) =  h^{it}  \Phi_k (   h^{-it} x  h^{it}  )  h^{-it}, \qquad t \in \mathbb{R},
\end{equation}
exists and is compact, i.e.\ contained in $\cK(M, N, \varphi)$. By assumption $h$ is boundedly invertible and so $t \mapsto h^{it}$ depends continuously (in norm) on $t$. Hence the map  \eqref{Eqn=IntegrantMap} depends continuously on $t$ and it follows that \eqref{Eqn=IntegralMap} is compact.

 Next, in the proof of \cite[Lemma 4.3]{CS-CMP} the following operators were defined:
 \[
    g_k^l = \Phi_k^l(1), \qquad f_k^{n,l} = F_n(g_k^l),
 \]
 where $F_n(z) = e^{-n (z-1)^2}, z \in \mathbb{C}, n \in \mathbb{N}$. Since $\Phi_k^l$ is $N$-$N$ bimodular it follows  that $g_k^l \in N' \cap M$. Therefore also $f_k^{n,l} \in N' \cap M$. Then the proof of \cite[Lemma 4.3]{CS-CMP} defines for suitable $n(j), k(j), l(j) \in \mathbb{N}, \epsilon_j >0$ depending on some $j$ in a directed set the map $\Psi_j:M \to M$ via the formula:
 \[
   \Psi_j( \: \cdot \: ) = \frac{1}{(1 + \epsilon_j)^2} f_{k(j)}^{n(j), l(j)}  \Phi_{k(j)}^{l(j)}( \: \cdot \: )  f_{k(j)}^{n(j), l(j)}.
 \]
   Since  $f_{k(j)}^{n(j), l(j)} \in N' \cap M$   it follows that $\Psi_j$ is both $N$-$N$-bimodular and compact, i.e. $\Psi_j^{(2)} \in \cK(M,N, \varphi)$. The last part of the proof of \cite[Lemma 4.3]{CS-CMP} shows that $\Psi_j^{(2)} \rightarrow 1$ strongly and this holds true here as well with the same proof.
  By Lemma \ref{Lem=SOTversusL2} this shows that for every $x \in M$ we have  $\Psi_j(x) \rightarrow x$ strongly.

 \end{proof}

\vspace{1mm}


\section{For finite $N$: translation into the finite setting} \label{Preparation}

Let again $N\subseteq M$ be a unital inclusion of von Neumann
algebras which admits a faithful normal conditional expectation $\mathbb{E}_{N}$.
Assume moreover that $N$ is a general $\sigma$-finite von Neumann
algebra, though in many of the statements below we shall add the assumption
that $N$ is finite. The aim of this section is to characterise the
relative Haagerup property (resp.\ property  $\text{(rHAP)}^{-}$) of
the triple $(M,N,\mathbb{E}_{N})$ in terms of the structure of  certain
corners of crossed product von Neumann algebras associated with the
modular automorphism group of some faithful $\varphi\in M_{\ast}^{+}$
with $\varphi\circ\mathbb{E}_{N}=\varphi$. These statements will
play a crucial role in Section \ref{MainResults}.


\subsection{Crossed products}  \label{CrossedProducts}

Let us first recall some of the theory of crossed product von Neumann
algebras and their duality for which we refer to \cite[Section X.2]{Takesaki2}.
For this, fix an action $\mathbb{R}\curvearrowright^{\alpha}M$ on
$M\subseteq\mathcal{B}(\mathcal{H})$, define the corresponding \emph{fixed
point algebra}
\[
M^{\alpha}:=\{x\in M\mid\alpha_{t}(x)=x\text{ for all }t\in\mathbb{R}\}
\]
and let $M\rtimes_{\alpha}\mathbb{R}\subseteq\mathcal{B}(\mathcal{H}\otimes L^{2}(\mathbb{R}))\cong\mathcal{B}(L^{2}(\mathbb{R},\mathcal{H}))$
be the corresponding \emph{crossed product von Neumann algebra}. It is generated by the operators $\pi_{\alpha}(x)$, $x\in M$ and $\lambda_{t}:=\lambda_{t}^{\alpha}$, $t\in\mathbb{R}$ where
\begin{eqnarray}
\nonumber
(\pi_{\alpha}(x)\xi)(t)=\alpha_{-t}(x)\xi(t) \qquad \text{ and } \qquad (\lambda_{t}\xi)(s)=\xi(s-t)\qquad
\end{eqnarray}
for $s,t\in\mathbb{R}$, $x\in M$, $\xi\in\mathcal{H}\otimes L^{2}(\mathbb{R})$;  we will also occasionally use $\lambda$ to denote the left regular representation on $L^2(\mathbb{R})$, which should not cause any confusion.
Recall that this construction does not depend on the choice of the
embedding $M\subseteq\mathcal{B}(\mathcal{H})$ and that $M\cong\pi_{\alpha}(M)$.
For notational convenience we will therefore omit the faithful normal
representation $\pi_{\alpha}$ in our notation and identify $M$ with
$\pi_{\alpha}(M)$ and $N$ with $\pi_{\alpha}(N)$. Note that
 $\pi_{\alpha}(x)=x\otimes1$ for all
$x\in M^\alpha$. Set further $\lambda(f):=\int_{\mathbb{R}}f(t)\lambda_{t}dt$
for $f\in L^{1}(\mathbb{R})$ and
\[
\mathcal{L}(\mathbb{R}):=\{\lambda(f)\mid f\in L^{1}(\mathbb{R})\}^{\prime\prime}=\{\lambda_{s}\mid s\in\mathbb{R}\}^{\prime\prime}\subseteq\mathcal{B}(\mathcal{H}\otimes L^{2}(\mathbb{R}))\text{.}
\]

\begin{rmk} \label{Rmk=Fourier}

For $f\in L^{1}(\mathbb{R})$
we denote by
\[
\widehat{f}(s) =   \int_{\mathbb{R}}  f(t) e^{ist} dt  \in L^{\infty}(\mathbb{R}),
\]
its
Fourier transform. Let $\mathcal{F}_{2}:L^{2}(\mathbb{R})\rightarrow L^{2}(\mathbb{R}): f \mapsto (2\pi)^{  -\frac{1}{2}  } \widehat{f}$
be the unitary Fourier transform operator on $L^{2}(\mathbb{R})$.
Then $\mathcal{F}_{2} \lambda(f)\mathcal{F}_{2}^{\ast}$ is the multiplication
operator with $\widehat{f}$. We shall occasionally extend our notation in the following way. Let $f \in L^2(\mathbb{R})$ be such that its Fourier transform $\widehat{f}$ is in $L^\infty(\mathbb{R})$. We shall write $\lambda(f)$ for $\mathcal{F}_{2}^{\ast} \widehat{f} \mathcal{F}_{2}$ where we view $\widehat{f}$ as a multiplication operator. This is naturally compatible with the earlier notation for $f \in L^{1}(\mathbb{R})$
\end{rmk}

Let $\mathbb{R}\overset{\widehat{\alpha}}{\curvearrowright}M\rtimes_{\alpha}\mathbb{R}$ be the \emph{dual action} determined by
\begin{equation}  \label{Eqn=DualAction}
\widehat{\alpha}_{t}(x)=x, \qquad \text{ and } \qquad \widehat{\alpha}_{t}(\lambda_{s})=\exp(-ist)\lambda_{s},
\end{equation}
for $x\in M$, $s,t\in \mathbb{R}$ and recall that its fixed point algebra is given by
\begin{equation}  \label{Eqn=FixedCross}
M=(M\rtimes_{\alpha}\mathbb{R})^{\widehat{\alpha}}\text{.}
\end{equation}

The expression
\[
T_{\widehat{\alpha}}(x):=\int_{\mathbb{R}}\widehat{\alpha}_{s}(x)ds,\qquad x\in(M\rtimes_{\alpha}\mathbb{R})^{+},
\]
defines a faithful normal semi-finite operator valued weight on $M\rtimes_{\alpha}\mathbb{R}$
which takes values in the extended positive part of $M$. Choose $f\in L^{1}(\mathbb{R})\cap L^{2}(\mathbb{R})$
with $\Vert f\Vert_{2}=1$ such that the support of the Fourier transform
$\widehat{f}$ equals $\mathbb{R}$. We keep $f$ fixed throughout
the whole section. One has $T_{\widehat{\alpha}}(\lambda(f)^{\ast}\lambda(f))=\Vert f\Vert_{2}^{2}=1$,
hence we may define the unital normal completely positive map
\[
T_{f}:=T_{f,\widehat{\alpha}}:M\rtimes_{\alpha}\mathbb{R}\rightarrow M\text{, }x\mapsto T_{\widehat{\alpha}}(\lambda(f)^{\ast}x\lambda(f))\text{.}
\]
By Lemma \ref{Lem=SOTversusL2} $T_{f}$ is strongly continuous on the
unit ball. For a given map $\Phi:M\rtimes_{\alpha}\mathbb{R}\rightarrow M\rtimes_{\alpha}\mathbb{R}$
and a positive normal functional $\varphi\in M_{\ast}$ we further
define
\begin{equation}
\nonumber
\widetilde{\Phi}_{f}:M\rightarrow M\text{, }\;\;\;\widetilde{\Phi}_{f}(x)=T_{f}(\Phi(x))\text{.}
\end{equation}
and
\[
\widehat{\varphi}_{f}:M\rtimes_{\alpha}\mathbb{R}\rightarrow\mathbb{C}\text{, }\;\;\;\widehat{\varphi}_{f}(x)=\varphi(T_{f}(x)).
\]
The functional $\widehat{\varphi}_{f}$ is normal and positive. It
is moreover a state if $\varphi$ is a state. Since we assumed the
support of $\widehat{f}$ to be equal to $\mathbb{R}$, by Remark
\ref{Rmk=Fourier} the support projection of $\lambda(f)$ equals
$1$. It follows that $\widehat{\varphi}_{f}$ is faithful if and
only if $\varphi$ is faithful.

\begin{lem} \label{Lem=EasyLemma}
Assume that $N\subseteq M^{\alpha}$.
Then $T_{f}$ is $N$-$N$-bimodular, meaning that for $x,y\in N$,
$a\in M\rtimes_{\alpha}\mathbb{R}$ we have $T_{f}(xay)=xT_{f}(a)y$.
\end{lem}

\begin{proof}
As $N\subseteq M^{\alpha}$ we have that $N$ and
$\lambda(f)$ commute. From the definition of $T_{\widehat{\alpha}}$
and \eqref{Eqn=FixedCross} we get that for $x,y\in N$ and $a\in M\rtimes_{\alpha}\mathbb{R}$,
\[
T_{\widehat{\alpha}}(\lambda(f)^{\ast}xay\lambda(f))=T_{\widehat{\alpha}}(x\lambda(f)^{\ast}a\lambda(f)y)=xT_{\widehat{\alpha}}(\lambda(f)^{\ast}a\lambda(f))y.
\]
 This concludes the proof.
\end{proof}

We recall the following formula which was proved in \cite[Lemma 5.2]{CS-IMRN} (which extends \cite[Theorem
3.1 (c)]{HaagerupMathScan}) in case $k=g$; the general case then follows from the polarization identity. For $k,g\in L^{2}(\mathbb{R})$  such that $\widehat{k}, \widehat{g} \in L^\infty(\mathbb{R})$ we have:
\begin{equation} \label{Eqn=HaagerupComputation}
T_{\widehat{\alpha}}(\lambda(k)^{\ast}x\lambda(g))=\int_{\mathbb{R}}\overline{k(t)}g(t)\alpha_{-t}(x)dt,\qquad x\in M.
\end{equation}
We shall need the following consequence of it. For $g \in L^1(\R)$ define $g^*(t) := \overline{g(-t)}$, which is the involution for the convolution algebra $L^1(\R)$.

\begin{lem} \label{Lem=Talpha}
Let $h\in C_c(\mathbb{R})$   and let $x\in M$. Then, for $k,g\in L^{1}(\mathbb{R})\cap L^{2}(\mathbb{R})$
and $a:=\lambda(h)x$,
\[
T_{\widehat{\alpha}}(\lambda(k)^{\ast}a\lambda(g))=\int_{\mathbb{R}}\int_{\mathbb{R}}k^{\ast}(s)g(t)h(-s-t)\alpha_{-t}(x)dsdt,
\]
and
\[
T_{\widehat{\alpha}}(\lambda(k)^{\ast}\lambda( g)a)=\int_{\mathbb{R}}\int_{\mathbb{R}k^{\ast}(s) g}(t)h(-s-t)xdsdt\text{.} \qquad
\]
\end{lem}

\begin{proof}
We have $\lambda(k)^{\ast}a=\lambda(h^{\ast}\ast k)^{\ast}x$.
The equality \eqref{Eqn=HaagerupComputation} then implies
\begin{eqnarray}
\nonumber
T_{\widehat{\alpha}}(\lambda(k)^{\ast}a\lambda(g)) &=& T_{\widehat{\alpha}}(\lambda(h^{\ast}\ast k)^{\ast}x\lambda(g))\\
\nonumber
&=& \int_{\mathbb{R}}\overline{(h^{\ast}\ast k)}(t)g(t)\alpha_{-t}(x)dt\\
\nonumber
&=& \int_{\mathbb{R}}\int_{\mathbb{R}}k^{\ast}(s)g(t)h(-s-t)\alpha_{-t}(x)dsdt\text{.}
\end{eqnarray}
This concludes the proof of the first formula.  The second formula follows from the first after observing that $T_{\widehat{\alpha}}(\lambda(k)^{\ast}\lambda(g)a) = T_{\widehat{\alpha}}(\lambda(k)^{\ast}\lambda(g)\lambda(h))x$.
\end{proof}


\subsection{Passage to crossed products} \label{PassageCrossedProducts}

Let us now study the stability of the relative Haagerup property with
respect to certain crossed products. The setting is the same as in
Subsection \ref{CrossedProducts}.

\begin{prop} \label{Prop=Phi}
Let $\Phi:M\rtimes_{\alpha}\mathbb{R}\rightarrow M\rtimes_{\alpha}\mathbb{R}$
be a linear map and fix $f\in L^{1}(\mathbb{R})\cap L^{2}(\mathbb{R})$ as before.
Then the following statements hold:

\begin{enumerate}
\item \label{Item=Phi1} If $\Phi$ is completely positive then so
is $\widetilde{\Phi}_{f}$.
\item \label{Item=Phi2} Assume that $N\subseteq M^{\alpha}$. If $\Phi$
is an $N$-$N$-bimodule map then $\widetilde{\Phi}_{f}$ is an $N$-$N$-bimodule
map.
\end{enumerate}

\noindent In the remaining statements let $\varphi\in M_{\ast}^{+}$
be a faithful normal positive functional with $\varphi\circ\mathbb{E}_{N}=\varphi$
and $\varphi\circ\alpha_{t}=\varphi$ for all $t\in\mathbb{R}$. Then:

\begin{enumerate} \setcounter{enumi}{2}
\item \label{Item=Phi3} If $\widehat{\varphi}_{f}\circ\Phi\leq\widehat{\varphi}_{f}$
(resp. $\widehat{\varphi}_{f}\circ\Phi=\widehat{\varphi}_{f}$) then
$\varphi\circ\widetilde{\Phi}_{f}\leq\varphi$ (resp. $\varphi\circ\widetilde{\Phi}_{f}=\varphi$).
\item \label{Item=Phi3'} If the $L^{2}$-implementation of $\Phi$
with respect to $\widehat{\varphi}_{f}$ exists, then the $L^{2}$-implementation
of $\widetilde{\Phi}_{f}$ with respect to $\varphi$ exists as well.
\end{enumerate}

\noindent Now, if $N\subseteq M^{\alpha}$, $\mathbb{E}_{N}\circ\alpha_{t}=\mathbb{E}_{N}$
for all $t\in\mathbb{R}$ and $f$ is continuous, then:

\begin{enumerate} \setcounter{enumi}{4}
\item  \label{Item=Phi4} If $\Phi\in\mathcal{K}_{00}(M\rtimes_{\alpha}\mathbb{R},N,\widehat{\varphi}_{f})$,
then $\widetilde{\Phi}_{f}\in\mathcal{K}_{00}(M,N,\varphi)$.
\item  \label{Item=Phi5} If $\Phi\in\mathcal{K}(M\rtimes_{\alpha}\mathbb{R},N,\widehat{\varphi}_{f})$,
then $\widetilde{\Phi}_{f}\in\mathcal{K}(M,N,\varphi)$ .
\end{enumerate}

\end{prop}

\begin{proof}
\emph{\eqref{Item=Phi1}} is straightforward from
the constructions and \eqref{Item=Phi2} follows from Lemma \ref{Lem=EasyLemma}.

\emph{\eqref{Item=Phi3}}: If $\widehat{\varphi}_{f}\circ\Phi\leq\widehat{\varphi}_{f}$
we have for $x\in M^{+}$, using \eqref{Eqn=HaagerupComputation}
and the $\alpha$-invariance of $\varphi$,

\begin{eqnarray}
\nonumber
\varphi\circ\widetilde{\Phi}_{f}(x) &=& \varphi(T_{f}(\Phi(x)))= \widehat{\varphi}_{f}(\Phi(x))\leq\widehat{\varphi}_{f}(x) \\
\nonumber
&=& \varphi(T_{\widehat{\alpha}}(\lambda(f)^{\ast}x\lambda(f)))= \int_{\mathbb{R}}\vert f(t)\vert^{2}\varphi(\alpha_{-t}(x))dt= \varphi(x)\text{.}
\end{eqnarray}
Moreover, if $\widehat{\varphi}_{f}\circ\Phi=\widehat{\varphi}_{f}$
then the inequality above is actually an equality.

\emph{\eqref{Item=Phi3'}}: Assume that there exists a constant $C>0$
such that $\widehat{\varphi}_{f}(\Phi(x)^{\ast}\Phi(x))\leq C\widehat{\varphi}_{f}(x^{\ast}x)$
for all $x\in M$. Then, by the Kadison-Schwarz inequality and \eqref{Eqn=HaagerupComputation},
\begin{eqnarray}
\nonumber
\varphi(\widetilde{\Phi}_{f}(x)^{\ast}\widetilde{\Phi}_{f}(x)) &=& \varphi(T_{f}(\Phi(x))^{\ast}T_{f}(\Phi(x)))\leq\widehat{\varphi}_{f}(\Phi(x)^{\ast}\Phi(x))\leq C\widehat{\varphi}_{f}(x^{\ast}x)
=C\varphi(x^{\ast}x)
\end{eqnarray}
for all $x\in M$, where we use the fact (proved above) that $\varphi$ and $\hat{\varphi}_f$ coincide on $M_+$. 
This implies that the $L^{2}$-implementation of
$\widetilde{\Phi}_{f}$ with respect to $\varphi$ exists.

\emph{\eqref{Item=Phi4}}: By Lemma \ref{Lem=EasyLemma} and the discussion before, $\mathbb{F}_{N}=\mathbb{E}_{N}\circ T_{f}$
is the unique faithful normal $\widehat{\varphi}_{f}$-preserving
conditional expectation of $M\rtimes_{\alpha}\mathbb{R}$ onto $N$.
Let $a,b\in M\rtimes_{\alpha}\mathbb{R}$. 
By $N\subseteq M^{\alpha}$ we have for $x\in M$,

\begin{equation}  \label{Eqn=FnRestrict}
\widetilde{(a\mathbb{F}_{N}b)_{f}}(x):=T_{f}(a\mathbb{F}_{N}(bx))=T_{f}(a)\mathbb{F}_{N}(bx)
\end{equation}
We shall show that $\mathbb{F}_{N}(bx)=\mathbb{E}_{N}(\widetilde{b}x)$
for all $x\in M$, where $\widetilde{b}:=T_{\widehat{\alpha}}(\lambda(f)^{\ast}\lambda(f)b)$.
For this it suffices to consider the case where $b=\lambda(h)y$ for
some compactly supported function $h\in C_c(\mathbb{R})$ and $y\in M$,
since such elements span a  $\sigma$-weakly dense subset of $M\rtimes_{\alpha}\mathbb{R}$   and the map $b \mapsto \widetilde{b}$ is $\sigma$-weakly continuous.
Using Lemma \ref{Lem=Talpha} twice and the fact that $\mathbb{E}_{N}\circ\alpha_{t}=\mathbb{E}_{N}$
for all $t\in\mathbb{R}$ one has
\begin{eqnarray}
\nonumber
\mathbb{F}_{N}(bx) &=& \mathbb{E}_{N}\circ T_{f}(bx)\\
\nonumber
&=& \mathbb{E}_{N}\left(\int_{\mathbb{R}}\int_{\mathbb{R}}f^{\ast}(s)f(t)h(-s-t)\alpha_{-t}(yx)dsdt\right)\\
\nonumber
&=& \int_{\mathbb{R}}\int_{\mathbb{R}}f^{\ast}(s)f(t)h(-s-t)\mathbb{E}_{N}(yx)dsdt\\
\nonumber
&=&\mathbb{E}_{N}\left(\int_{\mathbb{R}}\int_{\mathbb{R}}f^{\ast}(s)f(t)h(-s-t)y  dsdt \:\:   x\right)\\
\nonumber
&=& \mathbb{E}_{N}(\widetilde{b}x)\text{,}
\end{eqnarray}
as claimed. Combining this equality and \eqref{Eqn=FnRestrict} we
get that $\widetilde{(a\mathbb{E}_{N}b)_{f}}=T_f(a)\mathbb{F}_{N}\widetilde{b}$.
By considering linear combinations of such expressions one gets that
if $\Phi\in\mathcal{K}_{00}(M\rtimes_{\alpha}\mathbb{R},N,\widehat{\varphi}_{f})$
then also $\widetilde{\Phi}_{f}\in\mathcal{K}_{00}(M,N,\varphi)$.
This proves \eqref{Item=Phi4}.

\emph{\eqref{Item=Phi5}}: The statement follows directly from \eqref{Item=Phi4}
by approximation and the fact that $\Vert\widetilde{\Phi}_{f}\Vert\leq\Vert\Phi\Vert$.
\end{proof}

In the following we will direct our attention to certain choices of
functions $f\in L^{1}(\mathbb{R})\cap L^{2}(\mathbb{R})$ with $\Vert f\Vert_{2}=1$
whose support of the Fourier transform $\widehat{f}$ equals $\mathbb{R}$.
For this, define for $j\in\mathbb{N}$ the $L^{2}$-normalized
\emph{Gaussian}
\[
f_{j}:\mathbb{R}\rightarrow\mathbb{R}\text{, }f_{j}(s):=\left(\frac{j}{\pi}\right)^{1/4}\text{exp}(-js^{2}/2)\text{.}
\]
Further set for a given map $\Phi:M\rtimes_{\alpha}\mathbb{R}\rightarrow M\rtimes_{\alpha}\mathbb{R}$
and a positive normal functional $\varphi\in M_{\ast}$
\[
\widehat{\varphi}_{j}:=\widehat{\varphi}_{f_{j}} \qquad \text{ and } \qquad \text{\ensuremath{\widetilde{\Phi}_{j}:=\widetilde{\Phi}_{f_{j}}\text{.}}}
\]

\begin{thm} \label{Thm=ActionHAP}
Let $ N\subseteq M$
be a unital inclusion of von Neumann algebras with a faithful normal
conditional expectation $\mathbb{E}_{N}:M\rightarrow N$. Let further
$\varphi\in M_{\ast}^{+}$ be a faithful normal positive functional
with $\varphi\circ\mathbb{E}_{N}=\varphi$ and $\mathbb{R}\curvearrowright^{\alpha}M$ be
an action such that $N\subseteq M^{\alpha}$. Finally assume that $\mathbb{E}_{N}\circ\alpha_{t}=\mathbb{E}_{N}$  (or, equivalently under the earlier assumptions, that $\varphi = \varphi \circ \alpha_t$) for all $t\in\mathbb{R}$. Then the following statements hold:

\begin{enumerate}
\item \label{Item=ActionHAP1} If for all $j\in\mathbb{N}$
the triple $(M\rtimes_{\alpha}\mathbb{R},N,\widehat{\varphi}_{j})$
has property $\text{(rHAP)}$ (resp. property $\text{(rHAP)}^{-}$),
then $(M,N,\varphi)$ has property $\text{(rHAP)}$ (resp. property
$\text{(rHAP)}^{-}$).
\item \label{Item=ActionHAP2} If for all $j\in\mathbb{N}$
property $\text{(rHAP)}$ of $(M\rtimes_{\alpha}\mathbb{R},N,\widehat{\varphi}_{j})$
is witnessed by unital (resp. $\widehat{\varphi}_{j}$-preserving)
approximating maps (see (1$^{\prime\prime}$) and (4$^{\prime}$)
in Section \ref{relhap}), then also property $\text{(rHAP)}$
of $(M,N,\varphi)$ may be witnessed by unital (resp. $\varphi$-preserving)
approximating maps.
\end{enumerate}
\end{thm}

\begin{proof}
\emph{\eqref{Item=ActionHAP1}}: For fixed $j\in\mathbb{N}$
let $(\Phi_{j,k})_{k\in K_j}$ be a bounded net of normal
completely positive maps witnessing the relative Haagerup property
of $(M\rtimes_{\alpha}\mathbb{R},N,\widehat{\varphi}_{j})$. In particular,
$\Phi_{j,k}\rightarrow1$ in the point-strong topology in $k$. Set
$\widetilde{\Phi}_{j,k}:=T_{f_{j}}\circ\Phi_{j,k}$. As $s\mapsto\alpha_{s}(x)$
is strongly continuous for $x\in M$ and $f_{j}$ is $L^{2}$-normalized
with mass concentrated around 0, Lemma \ref{Lem=Talpha} shows that
for $x\in M$,
\[
T_{f_{j}}(x)=\int_{\mathbb{R}}\vert f_{j}(s)\vert^{2}\alpha_{s}(x)ds\rightarrow x
\]
as $j\rightarrow\infty$ in the strong topology. Lemma \ref{Lem=ApproxSOT}
then shows that we may find a directed set $\mathcal{F}$ and a function
$(\widetilde{j},\widetilde{k}):\mathcal{F}\rightarrow\{(j,k)\mid j\in \mathbb{N},k\in K_{j}\}\text{, }F\mapsto(\widetilde{j}(F),\widetilde{k}(F))$
such that the net $(\widetilde{\Phi}_{\widetilde{j}(F),\widetilde{k}(F)})_{F\in\mathcal{F}}$
converges to the identity in the point-strong topology. By Proposition
\ref{Prop=Phi} these maps then witness the relative Haagerup property
for $(M,N,\varphi)$. In the same way, using a variant of Lemma \ref{Lem=ApproxSOT}, we can deduce that if $(M\rtimes_{\alpha}\mathbb{R},N,\widehat{\varphi}_{j})$
has property $\text{(rHAP)}^{-}$, then $(M,N,\varphi)$ has property
$\text{(rHAP)}^{-}$ as well.

\emph{\eqref{Item=ActionHAP2}}: Note that if $\Phi_{j,k}$ is unital
for all $k\in\mathbb{N}$, then $\widetilde{\Phi}_{j,k}$ is unital
as well and if $\Phi_{j,k}$ is $\widehat{\varphi}_{j}$-preserving
for all $k\in\mathbb{N}$, then $\widetilde{\Phi}_{j,k}$ is $\varphi$-preserving,
c.f. Proposition \ref{Prop=Phi}.
\end{proof}

We will now apply this theorem to the modular automorphism group $\sigma^{\varphi}$
of $\varphi$ as well as its dual action.

\begin{thm} \label{Thm=CrossToNormal}
Let $N\subseteq M$
be a unital inclusion of von Neumann algebras with a faithful normal
conditional expectation $\mathbb{E}_{N}:M\rightarrow N$. Assume that
$N$ is finite and let $\tau\in N_{\ast}$ be a faithful normal tracial
state. Further define the faithful normal (possibly non-tracial) state
$\varphi:=\tau\circ\mathbb{E}_{N}$ on $M$. Then the following statements
hold:

\begin{enumerate}
\item If for all $j\in\mathbb{N}$ the triple $(M\rtimes_{\sigma^{\varphi}}\mathbb{R},N,\widehat{\varphi}_{j})$
has property $\text{(rHAP)}$ (resp. property $\text{(rHAP)}^{-}$),
then $(M,N,\varphi)$ has property $\text{(rHAP)}$ (resp. property
$\text{(rHAP)}^{-}$).
\item If for all $j\in\mathbb{N}$ property $\text{(rHAP)}$
of $(M\rtimes_{\alpha}\mathbb{R},N,\widehat{\varphi}_{j})$ is witnessed
by unital (resp. $\widehat{\varphi}_{j}$-preserving) approximating
maps, then also property $\text{(rHAP)}$ of $(M,N,\varphi)$ may
be witnessed by unital (resp. $\varphi$-preserving) approximating
maps.
\end{enumerate}
\end{thm}

\begin{proof}
This is Theorem \ref{Thm=ActionHAP} for $\alpha=\sigma^{\varphi}$; the assumptions are satisfied, as follows from the fact that $N \subset M^\varphi$.
\end{proof}

We will also prove the converse of Theorem \ref{Thm=CrossToNormal}
by using crossed product duality.  We first recall the following lemma which is well-known. We will use the fact that every function $g\in L^{\infty}(\mathbb{R})$ may be viewed
as a multiplication operator on $L^{2}(\mathbb{R})$.
\begin{lem} \label{Lem=HilbertSchmidt}
For $g,h\in C_{b}(\mathbb{R})\cap L^{2}(\mathbb{R})$
we have that $g\lambda(h)\in\mathcal{B}(L^{2}(\mathbb{R}))$ is Hilbert-Schmidt with
\[
{\rm Tr}((g\lambda(h))^{\ast}g\lambda(h))=\Vert g\Vert_{2}^{2}\Vert h\Vert_{2}^{2}\text{.}
\]
\end{lem}
\begin{proof}
Let $\cS_2(\cH)$ denote the Hilbert-Schmidt operators on a Hilbert space $\cH$.  We have  linear identifications $\cH \otimes \overline {\cH} \simeq \cS_2(\cH)$ where $\xi \otimes \overline{\eta}$ corresponds to the rank 1 operator $v \mapsto \xi \eta^\ast(v)$. We identify $L^2(\mathbb{R})$ with $\overline{L^2(\mathbb{R})}$ linearly and isometrically through the pairing $\langle \xi, \eta \rangle = \int_{\mathbb{R}} \xi(s) \eta(s) ds$. Therefore we have isometric linear identifications
\begin{equation}\label{Eqn=Isometries}
\cS_2(L^2(\R)) \simeq L^2(\mathbb{R}) \otimes L^2(\mathbb{R}) \simeq L^2(\mathbb{R}^2),
\end{equation}
 where the rank 1 operator $\xi \eta^\ast$ corresponds to the function $(s,t) \mapsto \xi(s) \eta(t)$.

Now, $g\lambda(h)$ is an integral operator on $L^2(\mathbb{R})$ with a square-integrable kernel $K(x,y):= g(x) h(x-y)$.
Then $g\lambda(h)$ is Hilbert-Schmidt and corresponds to $K \in L^2(\mathbb{R}^2)$ in \eqref{Eqn=Isometries}, so that $\Vert g\lambda(h) \Vert_{\cS_2} = \Vert K \Vert_2 = \|g\|_2 \|h\|_2$.
\end{proof}

Further recall that for $j\in\mathbb{N}$ the Gaussian $f_{j}$
was defined by $f_{j}(s):=j^{1/4}\pi^{-1/4}\text{exp}(-js^{2}/2)$,
$s\in\mathbb{R}$ and $\widehat{f_{j}}$ denotes its Fourier transform.
Both these functions are $L^{2}$-normalized by definition and the
Plancherel identity. Define for $i,j\in\mathbb{N}$ a positive linear functional
$\psi_{i,j}$ on $\mathcal{B}(L^{2}(\mathbb{R}))$ by
\[
\psi_{i,j}(x):={\rm Tr}((\widehat{f_{i}}\lambda(f_{j}))^{\ast}x\widehat{f_{i}}\lambda(f_{j}))\text{.}
\]
 It is a state by Lemma \ref{Lem=HilbertSchmidt}.
We will need the following elementary lemma for which we give a short
non-explicit proof following from the results in \cite{CS-IMRN}.

\begin{lem}\label{Lem=BHhap}
For all $i,j\in\mathbb{N}$ the pair $(\mathcal{B}(L^{2}(\mathbb{R})),\psi_{i,j})$
has the Haagerup property in the sense that the triple $(\mathcal{B}(L^{2}(\mathbb{R})), \mathbb{C}, \psi_{i,j})$ has the relative Haagerup property, see \cite[Definition 3.1]{CS-IMRN}.
Moreover, the approximating maps may be chosen to be unital and $\psi_{i,j}$-preserving.
\end{lem}

\begin{proof}
According to \cite[Proposition 3.4]{CS-IMRN},
$(\mathcal{B}(L^{2}(\mathbb{R}),\text{Tr})$ has the Haagerup property.
By \cite[Theorem 1.3]{CS-IMRN} the Haagerup property does not
depend on the choice of the faithful normal semi-finite weight and
hence $(\mathcal{B}(L^{2}(\mathbb{R})),\psi_{i,j})$ has the Haagerup
property for all $i,j\in\mathbb{N}$. In \cite[Theorem 5.1]{CS-CMP}
 it was proved that the approximating maps may be taken
unital and state preserving. This finishes the proof.
\end{proof}

As before, let $N\subseteq M$ be a unital inclusion of von
Neumann algebras which admits a faithful normal conditional expectation
$\mathbb{E}_{N}$ and fix a faithful normal state $\varphi$ on $M$
with $\varphi=\varphi\circ\mathbb{E}_{N}$. Let $\sigma^{\varphi}$
be the corresponding modular automorphism group, $M\rtimes_{\sigma^{\varphi}}\mathbb{R}$
the crossed product von Neumann algebra and let
\[
\theta:=\widehat{\sigma^{\varphi}}:\mathbb{R}\curvearrowright M\rtimes_{\sigma^{\varphi}}\mathbb{R}
\]
be the dual action as defined in \eqref{Eqn=DualAction}. Define for
$j\in\mathbb{N}$ the state $\widehat{\varphi}_{j}:=\varphi\circ T_{f_{j},\theta}$
on $M\rtimes_{\sigma^{\varphi}}\mathbb{R}$ as before and recall that
$M$ (hence also $N$) is invariant under $\theta$. We may in turn
consider the double crossed product which admits an isomorphism of
von Neumann algebras (i.e. a bijective $\ast$-homomorphism, which
is automatically normal by Sakai \cite[Theorem 1.13.2]{Sakai}),
\begin{equation} \label{doubleisom}
(M\rtimes_{\sigma^{\varphi}}\mathbb{R})\rtimes_{\theta}\mathbb{R}\cong M\otimes\mathcal{B}(L^{2}(\mathbb{R})).
\end{equation}
Let us describe what this isomorphism looks like. For $g\in L^{\infty}(\mathbb{R})$
write $\mu(g):=1_{M}\otimes g\in M\otimes\mathcal{B}(L^{2}(\mathbb{R}))$
for the multiplication operator acting in the second tensor leg. The double crossed
product above is generated by $M\rtimes_{\sigma^{\varphi}}\mathbb{R}$
and the left regular representation of  the second copy of $\mathbb{R}$, denoted here by
$\lambda_{t}^{\theta}$, $t\in\mathbb{R}$. Under the isomorphism,
$M\rtimes_{\sigma^{\varphi}}\mathbb{R}$ is identified as a subalgebra
of $M\otimes\mathcal{B}(L^{2}(\mathbb{R}))$ via inclusion. Further,
$\lambda_{t}^{\theta}$ is identified for every $t\in\mathbb{R}$
with $\mu(e_{t})=1_{M}\otimes e_{t}$ where $e_{t}(s):=\exp(-ist)$
for $s\in\mathbb{R}$. Under this correspondence, $\lambda^{\theta}(f_{j})=\mu(\widehat{f_{j}})$.
We find that for $x\in M\otimes\mathcal{B}(L^{2}(\mathbb{R}))$,
\begin{eqnarray}
\nonumber
(\varphi\circ T_{f_{j},\theta}\circ T_{f_{i},\widehat{\theta}})(x) &=& \varphi\left(T_{\theta}\left(\lambda(f_{j})^{\ast}T_{\widehat{\theta}}\left(\mu(\widehat{f_{i}})^{\ast}x\mu(\widehat{f_{i}})\right)\lambda(f_{j})\right)\right)\\
\nonumber
&=& \varphi\left(T_{\theta}\left(T_{\widehat{\theta}} \left(\lambda(f_{j})^{\ast}   \mu(\widehat{f_{i}})^\ast x \mu(\widehat{f_{i}}) \lambda(f_{j})\right)\right)\right)\text{.}
\end{eqnarray}
By \cite[Theorem X.2.3]{Takesaki2}
and the fact that $\varphi\circ\sigma_{t}^{\varphi}=\varphi$
we have that (formally, being imprecise about domains) the normal
semi-finite faithful weight $\varphi\circ T_{\theta}\circ T_{\widehat{\theta}}$
coincides with $\varphi\otimes{\rm Tr}$. Hence, for $i,j\in\mathbb{N}$ we have equality of states
\[
\varphi\circ T_{f_{j},\theta}\circ T_{f_{i},\widehat{\theta}}=\varphi\otimes\psi_{i,j}.
\]
 The following theorem now provides a passage to study the relative
Haagerup property on the continuous core of a von Neumann algebra,
which is semi-finite.

\begin{thm} \label{Thm=NormalToCross}
Let $ N\subseteq M$
be a unital inclusion of von Neumann algebras which admits a faithful
normal conditional expectation $\mathbb{E}_{N}$ and assume that $N$
is finite with a faithful normal tracial state $\tau\in N_{\ast}$.
Set $\varphi=\tau\circ\mathbb{E}_{N}\in M_{\ast}$. Then the following
two statements hold:

\begin{enumerate}
\item  \label{Thm=NormalToCross(1)} The triple $(M,N,\varphi)$ has
property $\text{(rHAP)}$ (resp. property $\text{(rHAP)}^{-}$) if
and only if $(M\rtimes_{\sigma^{\varphi}}\mathbb{R},N,\widehat{\varphi}_{j})$
has property $\text{(rHAP)}$ (resp. property $\text{(rHAP)}^{-}$)
for all $j\in\mathbb{N}$.
\item  \label{Thm=NormalToCross(2)} If property $\text{(rHAP)}$ of
$(M,N,\varphi)$ is witnessed by unital (resp. $\varphi$-preserving)
maps, then for all $j\in\mathbb{N}$ property $\text{(rHAP)}$
of $(M\rtimes_{\alpha}\mathbb{R},N,\widehat{\varphi}_{j})$ is witnessed
by unital (resp. $\widehat{\varphi}_{j}$-preserving) maps, and vice
versa.
\end{enumerate}
\end{thm}

\begin{proof}
The if statements were proven in Theorem \ref{Thm=CrossToNormal}.
For the converse of \eqref{Thm=NormalToCross(1)} assume that $(M,N,\varphi)$
has the relative Haagerup property.  $(\mathcal{B}(L^{2}(\mathbb{R})),\mathbb{C},\psi_{i,j})$
has the relative Haagerup property for all $i,j\in\mathbb{N}$, see Lemma \ref{Lem=BHhap}. Therefore by a suitable modification of \cite[Lemma 3.5]{CS-IMRN},
we see that $(M\otimes\mathcal{B}(L^{2}(\mathbb{R})),N\otimes\mathbb{C},\varphi\otimes\psi_{i,j})$
has the relative Haagerup property for all $i,j\in\mathbb{N}$. It follows
from Theorem \ref{Thm=ActionHAP} and the discussion above that $(M\rtimes_{\sigma^{\varphi}}\mathbb{R},N,\widehat{\varphi}_{j})$
has the relative Haagerup property
\footnote{Note that
in the picture above $\pi(x)=x\otimes1$ and hence $\pi(N)=N\otimes\mathbb{C}$
since $\varphi$ is tracial on $N$. This is used implicitly in the
identifications of $N$ in the double crossed product isomorphism \eqref{doubleisom}.}.

The statements in \eqref{Thm=NormalToCross(2)} and the statement
about property $\text{(rHAP)}^{-}$ follow in the same way.
\end{proof}


\subsection{Passage to corners of crossed products}

In the last subsection we characterised the relative Haagerup property
of the triple $(M,N,\varphi)$ for finite $N$ with a faithful normal
tracial state $\tau\in N_{\ast}$ and $\varphi:=\tau\circ\mathbb{E}_{N}\in M_{\ast}$
in terms of the crossed product triples $(M\rtimes_{\alpha}\mathbb{R},N,\widehat{\varphi}_{j})$,
$j\in\mathbb{N}$. In the following we will pass over to suitable
corners of these crossed products which allows to translate our investigations into the setting
of  finite von Neumann algebras. In this setting the following
lemma will be useful.

\begin{lem} \label{Lem=CompactTwist}
Let $ N\subseteq M$
be a unital inclusion of finite von Neumann algebras, let $\tau\in M_{\ast}$
be a faithful normal tracial state and let $\mathbb{E}_{N}:M\rightarrow N$
be the unique $\tau$-preserving faithful normal conditional expectation
onto $N$. Further, let $h\in N^{\prime}\cap M$ be self-adjoint and boundedly
invertible. For a linear completely positive map $\Phi:M\rightarrow M$ set
\[
\Phi^{h}(x)=h^{-1}\Phi(hxh)h^{-1}\text{.}
\]
Then, the $L^{2}$-implementation $\Phi^{(2)}$ of $\Phi$ with respect
to $\tau$ exists if and only if the $L^{2}$-implementation $(\Phi^{h})^{(2)}$
of $\Phi^{h}$ with respect to $h\tau h$ exists. Further, $\Phi^{(2)}\in\mathcal{K}(M,N,\tau)$
if and only if $(\Phi^{h})^{(2)}\in\mathcal{K}(M,N,h\tau h)$.
\end{lem}

\begin{proof}
Note first that the assumptions on $h$ imply that
$\mathbb{E}_{N}(h^{2})$ is a positive boundedly invertible element
of the center $Z(N)$ of $N$. Indeed, we have for all $n\in N$ the
equality $n\mathbb{E}_{N}(h^{2})=\mathbb{E}_{N}(nh^{2})=\mathbb{E}_{N}(h^{2}n)=\mathbb{E}_{N}(h^{2})n$,
and if $h$ is boundedly invertible, then $h^2\geq c1_{M}$ for some $c>0$, hence $\mathbb{E}_{N}(h^{2})\geq c 1_{M}$.

The map $\mathbb{E}_{N}^{h}:x\mapsto\mathbb{E}_{N}(h^{2})^{-1/2}\mathbb{E}_{N}(hxh)\mathbb{E}_{N}(h^{2})^{-1/2}$
is the unique normal $h\tau h$-preserving conditional expectation
onto $N$. Indeed,  we can verify it is an idempotent, normal, ucp map with image equal to $N$ and for any $x\in M$ we have
\begin{eqnarray}
\nonumber
(h\tau h)(\mathbb{E}_{N}^{h}(x)) &=& \tau(h\mathbb{E}_{N}(h^{2})^{-1/2}\mathbb{E}_{N}(hxh)\mathbb{E}_{N}(h^{2})^{-1/2}h)\\
\nonumber
&=& \tau(h^2\mathbb{E}_{N}(h^{2})^{-1}\mathbb{E}_{N}(hxh))\\
\nonumber
&=& \tau(\mathbb{E}_{N}(h^{2}\mathbb{E}_{N}(h^{2})^{-1}\mathbb{E}_{N}(hxh)))\\
\nonumber
&=& \tau(\mathbb{E}_{N}(h^{2})\mathbb{E}_{N}(h^{2})^{-1}\mathbb{E}_{N}(hxh))\\
\nonumber
&=&\tau(\mathbb{E}_{N}(hxh))\\
\nonumber
&=&\tau(hxh)\\
\nonumber
&=& (h\tau h)(x)\text{.}
\end{eqnarray}

Now assume that the $L^{2}$-implementation $\Phi^{(2)}$ of $\Phi$
with respect to $\tau$ exists, i.e. that there exists a constant
$C>0$ such that $\tau(\Phi(x)^{\ast}\Phi(x))\leq C\tau(x^{\ast}x)$
for all $x\in M$. Then 
\begin{eqnarray}
\nonumber
& & (h\tau h)(\Phi^{h}(x)^{\ast}\Phi^{h}(x))=\tau(\Phi(hx^{\ast}h)h^{-2}\Phi(hxh)) \leq \left\Vert h^{-2}\right\Vert \tau(\Phi(hx^{\ast}h)\Phi(hxh)) \\
\nonumber
& \leq & C\left\Vert h^{-2}\right\Vert \tau(hx^{\ast}hhxh) \leq C\left\Vert h^{-2}\right\Vert \left\Vert h^{2}\right\Vert \tau(hx^{\ast}xh)= C\left\Vert h^{-2}\right\Vert \left\Vert h^{2}\right\Vert (h\tau h)(x^{\ast}x)
\end{eqnarray}
for all $x\in M$, so the $L^{2}$-implementation $(\Phi^{h})^{(2)}$
exists as well.

The converse implication follows, as $\Phi = ({\Phi^h})^{h^{-1}}$.



For elements $a,b,x\in M$ the equality
\begin{eqnarray}
\nonumber
(a\mathbb{E}_{N}b)^{h}(x) &=& h^{-1}a\mathbb{E}_{N}(bhxh)h^{-1}\\
\nonumber
&=& h^{-1}a\mathbb{E}_{N}(h^{2})^{1/2}\mathbb{E}_{N}^{h}(h^{-1}bhx)\mathbb{E}_{N}(h^{2})^{1/2}h^{-1}\\
\nonumber
&=& (h^{-1}a\mathbb{E}_{N}(h^{2})h^{-1})\mathbb{E}_{N}^{h}(h^{-1}bhx) \\
\nonumber
&=& (h^{-1}a\mathbb{E}_{N}(h^{2})h^{-1}\mathbb{E}_{N}^{h}h^{-1}bh)(x)
\end{eqnarray}
implies by taking linear combinations and approximation that if $\Phi^{(2)}\in\mathcal{K}(M,N,\tau)$,
then $(\Phi^{h})^{(2)}\in\mathcal{K}(M,N,h\tau h)$. The converse
statement follows as before, which finishes the proof.
\end{proof}

Now, for a triple $(M,N,\varphi)$
let $h$ be the unique (possibly unbounded) positive self-adjoint operator affiliated
with $M\rtimes_{\sigma^{\varphi}}\mathbb{R}$ such that $h^{it}=\lambda_{t}$
for all $t\in\mathbb{R}$. If we further assume that $N\subseteq M^{\sigma^{\varphi}}$  (which implies that $N$ is finite with a tracial state $\varphi|_N$)
we have for $x\in N$ that $\lambda_{t}x\lambda_{t}^{\ast}=\sigma_{t}^{\varphi}(x)=x$
and hence $\lambda_{t}\in N^{\prime}\cap(M\rtimes_{\sigma^{\varphi}}\mathbb{R})$.
This implies that $h$ is affiliated with $N^{\prime}\cap(M\rtimes_{\sigma^{\varphi}}\mathbb{R})$
and so its finite spectral projections are elements in $N^{\prime}\cap(M\rtimes_{\sigma^{\varphi}}\mathbb{R})$.
Set for $k\in\mathbb{N}$
\[
p_{k}=\chi_{[k^{-1},k]}(h)\qquad\text{ and }\qquad h_{k}=hp_{k}\text{.}
\]
Here $\chi_{[k^{-1},k]}$ denotes the indicator function of $[k^{-1},k]\subseteq\mathbb{R}$
and $p_{k}$ is the corresponding spectral projection. Then, for every $k\in\mathbb{N}$,
$h_{k}$ is boundedly invertible in the corner algebra $p_{k}(M\rtimes_{\sigma^{\varphi}}\mathbb{R})p_{k}$
and we write $h_{k}^{-1}$ for its inverse which we view as an operator
in $M\rtimes_{\sigma^{\varphi}}\mathbb{R}$.

Denote by $\widehat{\varphi}:=\varphi\circ T_{\theta}$ the dual weight
of $\varphi$ and let $\tau_{\rtimes}$ be the unique faithful normal
semi-finite weight on $M\rtimes_{\sigma^{\varphi}}\mathbb{R}$ whose
Connes cocycle derivative satisfies $(D\widehat{\varphi}/D\tau_{\rtimes})_{t}=h^{it}$
for all $t\in\mathbb{R}$ (we refer to \cite[Lemma 5.2]{HaaOVW}; the proofs
below stay within the realm of bounded functionals). It is a trace
on $M\rtimes_{\sigma^{\varphi}}\mathbb{R}$ which is formally given
by
\[
\tau_{\rtimes}(x)=\varphi\circ T_{\theta}(h^{-\frac{1}{2}}xh^{-\frac{1}{2}})\text{, }\qquad x\in (M\rtimes_{\sigma^{\varphi}}\mathbb{R})^+\text{.}
\]
By construction we have

\begin{equation} \label{Eqn=HkRadonNykodym}
\widehat{\varphi}_{j}(p_{k}xp_{k})=\tau_{\rtimes}(h_{k}^{\frac{1}{2}}\lambda(f_{j})^{\ast}x\lambda(f_{j})h_{k}^{\frac{1}{2}}),\qquad x\in M\rtimes_{\sigma^{\varphi}}\mathbb{R}\text{.}
\end{equation}
for all $j\in\mathbb{N}$, where $\widehat{\varphi}_{j}$
and $f_{j}$ are defined as in Subsection  \ref{PassageCrossedProducts}.
Further note that the operators $\lambda(f_{j})$ and $h_{k}$ commute.

\begin{rmk} \label{Rmk=Fourier2}
Following Remark \ref{Rmk=Fourier},
for $k\in\mathbb{N}$ the operators $p_{k}$ and $h_{k}$ can be described
in terms of multiplication operators conjugated with the Fourier unitary
$\mathcal{F}_{2}$. Indeed, $\mathcal{F}_{2} \lambda_{t}\mathcal{F}_{2}^{\ast}$
is the multiplication operator on $L^{2}(\mathbb{R},\mathcal{H})$
with the function $(s\mapsto e^{its})$ and therefore (under proper
identification of the domains) $\mathcal{F}_{2} h\mathcal{F}_{2}^{\ast}$
coincides with the multiplication operator with $(s\mapsto e^{s})$.
It follows that for all $k\in\mathbb{N}$, $\mathcal{F}_{2}p_{k}\mathcal{F}_{2}^{\ast}$
is the multiplication with $(I_{k}:s\mapsto\chi_{[-\log(k),\log(k)]}(s)  )$
and $\mathcal{F}_{2}h_{k}\mathcal{F}_{2}^{\ast}$ is the multiplication
with $(J_{k}:s\mapsto\chi_{[-\log(k),\log(k)]}(s)e^{s})$. Therefore,
by Remark \ref{Rmk=Fourier},
\[
p_{k}=\lambda(\widehat{I_{k}}),\qquad h_{k}=\lambda(\widehat{J_{k}}),\qquad\textrm{ and }\qquad h_{k}^{-1}=\lambda(\widehat{J_{k}^{-1}})\text{,}
\]
where $J_{k}^{-1}$ is the function $(s\mapsto\chi_{[-\log(k),\log(k)]}e^{-s})$.
We also have that
\begin{equation} \label{Eqn=PosBdInv}
\lambda(f_{j})h_{k}=\lambda(f_{j})\lambda(\widehat{J_{k}})=\lambda(f_{j}\ast\widehat{J_{k}})=\mathcal{F}_{2}^{\ast}\widehat{f}_{j}J_{k}\mathcal{F}_{2},
\end{equation}
where we view the product $\widehat{f}_{j}J_{k}$ as a multiplication
operator. Since the Fourier transform of $f_{j}$ is Gaussian we see
that $ \mathcal{F}_2^\ast \widehat{f}_{j}J_{k} \mathcal{F}_2$ is positive and boundedly invertible  in the corner algebra $ p_k(M \rtimes_{\sigma^\varphi} \R)p_k$.
Further, by \eqref{Eqn=HaagerupComputation} and the Plancherel identity,
\[
T_{\theta}(h_{k}^{-1})  = T_\theta (\lambda(\widehat{J_k^{-1/2}}) \lambda(\widehat{J_k^{-1/2}})) =\Vert\widehat{J_{k}^{-1/2}}\Vert_{2}^{2}=\Vert J_{k}^{-1/2}\Vert_{2}^{2}=k-k^{-1}. 
\]
It follows that
\[
\tau_{\rtimes}(p_{k})=\varphi(T_{\theta}(h^{-1/2}p_{k}h^{-1/2}))=\varphi(T_{\theta}(h_{k}^{-1}))=  k - k^{-1}.
\]
 In particular, $\tau_{\rtimes}(p_{k})<\infty$. 
Since $\tau_\rtimes$ is tracial we also have for $x \in  M\rtimes_{\sigma^{\varphi}} \mathbb{R}$, 
\begin{equation}\label{Eqn=TauExpressionExtra}
\tau_{\rtimes}(p_k x p_k) = \varphi \circ T_\theta(h_k^{-1} p_k x p_k).
\end{equation}
\end{rmk}

In the next statements it is notationally more convenient to work
with property $\text{(rHAP)}$ (resp. property $\text{(rHAP)}^{-}$)
for general faithful normal positive functionals instead of just states,
see Remark \ref{Rmk=HAPforPosFunctionals}. Note that $p_{k}\widehat{\varphi}_{j}p_{k}$,
$j\in\mathbb{N}$ is not a state, but a positive scalar multiple of a state.

We shall use the fact that the unique faithful normal $\widehat{\varphi}_{j}$-preserving
conditional expectation $\mathbb{E}_{N}^{\widehat{\varphi}_{j}}$
of $M\rtimes_{\sigma^{\varphi}}\mathbb{R}$ onto $N$ is given by
$\mathbb{E}_{N}^{\widehat{\varphi}_{j}}=\mathbb{E}_{N}\circ T_{f_{j}}$.
This fact was used in the proof of Proposition \ref{Prop=Phi} already.

\begin{lem} \label{Lem=NewExpectation}
For every $k\in\mathbb{N}$,
$j\in\mathbb{N}$ there is a faithful normal $p_{k}\widehat{\varphi}_{j}p_{k}$-preserving
conditional expectation of $p_{k}(M\rtimes_{\sigma^{\varphi}}\mathbb{R})p_{k}$
onto $p_{k}Np_{k}$ given by
\begin{equation} \label{Eqn=CEV}
x\mapsto\mu_{k}^{-1}p_{k}\mathbb{E}_{N}(T_{f_{j}}(x))p_{k}=\mu_{k}^{-1}p_{k}\mathbb{E}_{N}^{\widehat{\varphi}_{j}}(x)p_{k}
\end{equation}
where $\mu_{k}:=T_{f_{j}}(p_{k})=\Vert\widehat{f}_{j}\chi_{[-\log(k),\log(k)]}\Vert_{2}^{2}$.
In particular, $T_{f_{j}}(p_{k})$ is a scalar multiple of the identity.
\end{lem}

\begin{proof}
First note that by Remark \ref{Rmk=Fourier} and Remark
\ref{Rmk=Fourier2} the operator $p_{k}\lambda(f_{j})$ coincides
with $\lambda(g_{j,k})$ where $g_{j,k}$ is the inverse Fourier transform
of the function $\widehat{f}_{j}\chi_{[-\log(k),\log(k)]}$.  The equality \eqref{Eqn=HaagerupComputation} then implies that
\begin{equation} \label{Eqn=LambdaScalar}
T_{f_{j}}(p_{k})=T_{\theta}(\lambda(f_{j})^{\ast}p_{k}\lambda(f_{j}))=  T_\theta(\lambda(g_{j,k})^* \lambda(g_{j,k}))=\Vert g_{j,k}\Vert_{2}^{2}
=\mu_{k}
\end{equation}
is a multiple of the identity.

For $x\in p_{k}(M\rtimes_{\sigma^{\varphi}}\mathbb{R})p_{k}$ expand
\begin{eqnarray}
\nonumber
(p_{k}\widehat{\varphi}_{j}p_{k})(p_{k}\mathbb{E}_{N}^{\widehat{\varphi}_{j}}(x)p_{k}) &=& \widehat{\varphi}_{j}(p_{k}\mathbb{E}_{N}^{\widehat{\varphi}_{j}}(x)p_{k})\\
\nonumber
& = & (\varphi\circ T_{f_{j}})(p_{k}\mathbb{E}_{N}(T_{f_{j}}(x))p_{k})\\
\nonumber
& = & (\varphi\circ T_{\theta})(\lambda(f_{j})^{\ast}p_{k}\mathbb{E}_{N}(T_{f_{j}}(x))p_{k}\lambda(f_{j}))\text{.}
\end{eqnarray}
 Since  $N\subseteq M^{\sigma^{\varphi}}$ we see that
\begin{eqnarray}
\nonumber
(p_{k}\widehat{\varphi}_{j}p_{k})(p_{k}\mathbb{E}_{N}^{\widehat{\varphi}_{j}}(x)p_{k}) &
 = & (\varphi\circ T_{\theta})(\mathbb{E}_{N}(T_{f_{j}}(x))\lambda(f_{j})^{\ast}p_{k}\lambda(f_{j}))\\
\nonumber
&= & \varphi(\mathbb{E}_{N}(T_{f_{j}}(x))T_{f_{j}}(p_{k}))\text{.}
\end{eqnarray}
With \eqref{Eqn=LambdaScalar} we can continue as follows:
\[
(p_{k}\widehat{\varphi}_{j}p_{k})(p_{k}\mathbb{E}_{N}^{\widehat{\varphi}_{j}}(x)p_{k})=\mu_{k}\varphi(\mathbb{E}_{N}(T_{f_{j}}(x)))=\mu_{k}\varphi(T_{f_{j}}(x))=\mu_{k}\widehat{\varphi}_{j}(x)=\mu_{k}\widehat{\varphi}_{j}(p_{k}xp_{k})\text{.}
\]
This proves that \eqref{Eqn=CEV} is $p_{k}\widehat{\varphi}_{j}p_{k}$-preserving,
as claimed. For $x\in N\subseteq M^{\sigma^{\varphi}}$ we have that
$x$ and $p_{k}$ commute. Therefore, using the $N$-module property
of the maps involved,
\[
p_{k}\mathbb{E}_{N}(T_{f_{j}}(p_{k}xp_{k}))p_{k}=p_{k}xp_{k}\mathbb{E}_{N}(T_{f_{j}}(p_{k}))=\mu_{k}p_{k}xp_{k}.
\]
 This shows that the map $x\mapsto\mu_{k}^{-1}p_{k}\mathbb{E}_{N}(T_{f_{j}}(x))p_{k}$
is a unital (the unit being $p_{k}$) normal completely positive projection
onto $p_{k}Np_{k}$ (see \cite[Theorem 1.5.10]{NateTaka}).
\end{proof}

\begin{lem} \label{Lem=TomitaTakesaki}
Let $ N\subseteq M$
be a unital inclusion of von Neumann algebras which admits a faithful
normal conditional expectation $\mathbb{E}_{N}$. Assume that $N$
is finite and let $\tau\in N_{\ast}$ be a faithful normal tracial
state that we extend to a state $\varphi:=\tau\circ\mathbb{E}_{N}$
on $M$. Then we have $\mathbb{E}_{N}(T_{f_{j}}(xa))=\mathbb{E}_{N}(T_{f_{j}}(ax))$  and   $\mathbb{E}_{N}(T_{\theta}(xa))=\mathbb{E}_{N}(T_{\theta}(ax))$
for every $j\in\mathbb{N}$, $a\in\mathcal{L}(\mathbb{R})$
and $x\in M\rtimes_{\sigma^{\varphi}}\mathbb{R}$.
\end{lem}

\begin{proof}


We first prove that $\mathbb{E}_{N}(T_{f_{j}}(xa))=\mathbb{E}_{N}(T_{f_{j}}(ax))$. 
 Suppose $a= \lambda( k)$ and $x= y \lambda(g)$ for $y \in M$, $k \in L^1(\R)$ and $g \in C_c(\R)$. Let us first compute $ T_{f_{j}}(xa)$ and $T_{f_{j}}(ax)$. By the formula \eqref{Eqn=HaagerupComputation} we have
\[
T_{f_{j}}(xa) = \int_{\mathbb{R}} f_{j}^{\ast}(-t)  (g \ast  k \ast f_{j})(t) \sigma^{\varphi}_{-t}(y) dt.
\]
By a similar computation we get
\[
T_{f_{j}}(ax) = \int_{\mathbb{R}} (f_{j}^{\ast}\ast  k)(-t) (g\ast f_{j})(t) \sigma^{\varphi}_{-t}(y) dt.
\]
We now apply $\mathbb{E}_{N}$ to these expressions and use the fact that $N$
is contained in the centralizer of $\varphi$, so $\mathbb{E}_{N} (\sigma^{\varphi}_{-t}(y)) = \mathbb{E}_{N} (y)$. It therefore suffices to prove the equality of the integrals $\int_{\mathbb{R}} f_{j}^{\ast}(-t)  (g \ast  k \ast f_{j})(t) dt$ and $\int_{\mathbb{R}} (f_{j}^{\ast}\ast  k)(-t) (g\ast f_{j})(t) dt$; Using the commutativity of the convolution on $\mathbb{R}$, we can rewrite the first one as
\[
\int_{\mathbb{R}} \int_{\mathbb{R}} f_{j}^{\ast}(-t) (g\ast f_{j})(t-s) {k }(s) ds dt
\]
and the second one is equal to
\[
\int_{\mathbb{R}} \int_{\mathbb{R}} f_{j}^{\ast}(-t-s) k (s) (g\ast f_{j}(t)) ds dt.
\]
In the second integral we can introduce a new variable $t':= t+s$ and it transforms into
\[
\int_{\mathbb{R}} \int_{\mathbb{R}} f_{j}^{\ast}(-t') (g\ast f_{j})(t'-s) k (s) ds dt',
\]
which is equal to the first one. For arbitrary $a \in \mathcal{L}(\mathbb{R})$ and $x \in M\rtimes_{\sigma^{\varphi}}\mathbb{R}$ we can find bounded nets $(a_i)_{i \in I}$ and $(x_i)_{i\in I}$ formed by linear combinations of elements of the form discussed above that converge strongly to $a$ and $x$, respectively, as a consequence of Kaplansky's density theorem. As multiplication is strongly continuous on bounded subsets, we have strong limits $\lim_{i\in I} a_{i} x_{i} = ax$ and $\lim_{i\in I} x_{i} a_{i}$. As both $\mathbb{E}_{N}$ and $T_{f_{j}}$ are strongly continuous on bounded subsets, we may conclude.

The equality  $\mathbb{E}_{N}(T_{\theta}(xa))=\mathbb{E}_{N}(T_{\theta}(ax))$ follows by a similar computation.
\end{proof}

The ideas appearing in the proof of the next statements are of a similar type.

\begin{prop} \label{Prop=ToCorner}
Let $N\subseteq M$
be a unital inclusion of von Neumann algebras which admits a faithful
normal conditional expectation $\mathbb{E}_{N}$. Assume that $N$
is finite and let $\tau\in N_{\ast}$ be a faithful normal tracial
state that we extend to a state $\varphi:=\tau\circ\mathbb{E}_{N}$
on $M$. Then, for every $j\in\mathbb{N}$, the following
statements hold:

\begin{enumerate}
\item  \label{Eqn=HardHAPEq1} The triple $(M\rtimes_{\sigma^{\varphi}}\mathbb{R},N,\widehat{\varphi}_{j})$
satisfies property $\text{(rHAP)}$ if and only if the triple $(p_{k}(M\rtimes_{\sigma^{\varphi}}\mathbb{R})p_{k},p_{k}Np_{k},p_{k}\widehat{\varphi}_{j}p_{k})$
satisfies property $\text{(rHAP)}$ for every $k\in\mathbb{N}$. Moreover,
the $\text{(rHAP)}$  may be witnessed by contractive maps,
i.e. we may assume that (1$^{\prime}$) holds.
\item  \label{Eqn=HardHAPEq2} If for every $k\in\mathbb{N}$ the property
$\text{(rHAP)}$ of the triple $(p_{k}(M\rtimes_{\sigma^{\varphi}}\mathbb{R})p_{k},p_{k}Np_{k},p_{k}\widehat{\varphi}_{j}p_{k})$
is witnessed by unital $p_{k}\widehat{\varphi}_{j}p_{k}$-preserving
approximating maps, then the relative Haagerup property of $(M\rtimes_{\sigma^{\varphi}}\mathbb{R},N,\widehat{\varphi}_{j})$
is witnessed by unital $\widehat{\varphi}_{j}$-preserving maps.
\item  \label{Eqn=HardHAPEq3} If the triple $(M\rtimes_{\sigma^{\varphi}}\mathbb{R},N,\widehat{\varphi}_{j})$
satisfies property $\text{(rHAP)}^{-}$ then the triple $(p_{k}(M\rtimes_{\sigma^{\varphi}}\mathbb{R})p_{k},p_{k}Np_{k},p_{k}\widehat{\varphi}_{j}p_{k})$
satisfies property $\text{(rHAP)}^{-}$ as well for every $k\in\mathbb{N}$.
\end{enumerate}
\end{prop}

\begin{proof}
\emph{First part of \eqref{Eqn=HardHAPEq1}}: For
the ``$\Rightarrow$'' direction assume that $(M\rtimes_{\sigma^{\varphi}}\mathbb{R},N,\widehat{\varphi}_{j})$
satisfies property $\text{(rHAP)}$, that it is witnessed by a net
of maps $(\Phi_{i})_{i\in I}$ and fix $k\in\mathbb{N}$. We will
show that $(p_{k}\Phi_{i}(\:\cdot\:)p_{k})_{i\in I}$ is a net of
approximating maps witnessing the relative Haagerup property of $(p_{k}(M\rtimes_{\sigma^{\varphi}}\mathbb{R})p_{k},p_{k}Np_{k},p_{k}\widehat{\varphi}_{j}p_{k})$:

It is clear that for every $i\in I$ the map $p_{k}\Phi_{i}(\:\cdot\:)p_{k}$
is completely positive, that the net $(p_{k}\Phi_{i}(\:\cdot\:)p_{k})_{i\in I}$
admits a uniform bound on its norms and that $p_{k}\Phi_{i}(\:\cdot\:)p_{k}\rightarrow\text{id}$
in the point-strong topology in $i$ as maps on $p_{k}(M\rtimes_{\sigma^{\varphi}}\mathbb{R})p_{k}$.

By our assumptions, $N\subseteq M^{\sigma^{\varphi}}$ and hence $p_{k}$
and $N$ commute. Hence for $a,b\in N$, $x\in p_{k}(M\rtimes_{\sigma^{\varphi}}\mathbb{R})p_{k}$ we have
\[
p_{k}\Phi_{i}(p_{k}ap_{k}xp_{k}bp_{k})p_{k}=p_{k}\Phi_{i}(ap_{k}xp_{k}b)p_{k}=p_{k}a\Phi_{i}(p_{k}xp_{k})bp_{k}=p_{k}ap_{k}\Phi_{i}(p_{k}xp_{k})p_{k}bp_{k},
\]
 which shows that $p_{k}\Phi_{i}(\:\cdot\:)p_{k}$ is a $p_{k}Np_{k}$-$p_{k}Np_{k}$-bimodule
map for every $i\in I$.

 We have by \cite[Theorem VIII.3.19.(vi)]{Takesaki2},  \cite[Theorem X.1.17.(ii)]{Takesaki2} and the fact that $p_k$ and $\lambda(f_j)$ commute that
\[
\sigma^{\widehat{\varphi}_j}_t(p_k) =  \lambda(f_j)^{it} \sigma^{\widehat{\varphi}}_t(p_k) \lambda(f_j)^{-it} =  \lambda(f_j)^{it} p_k \lambda(f_j)^{-it} = p_k.
\]
Therefore by \cite[Lemma 2.3]{CS-IMRN}, for $x\in p_{k}(M\rtimes_{\sigma^{\varphi}}\mathbb{R})p_{k}$
positive,
\[
(p_{k}\widehat{\varphi}_{j}p_{k})(p_{k}\Phi_{i}(x)p_{k})=\widehat{\varphi}_{j}(p_{k}\Phi_{i}(x)p_{k})\leq\widehat{\varphi}_{j}(\Phi_{i}(x))\leq\widehat{\varphi}_{j}(x)=(p_{k}\widehat{\varphi}_{j}p_{k})(x)\text{,}
\]
i.e. $(p_{k}\widehat{\varphi}_{j}p_{k})\circ(p_{k}\Phi_{i}(\:\cdot\:)p_{k})\leq p_{k}\widehat{\varphi}_{j}p_{k}$.

Now, for every map $\Phi$ on $M\rtimes_{\sigma^{\varphi}}\mathbb{R}$
of the form $\Phi=a\mathbb{E}_{N}^{\widehat{\varphi}_{j}}(\cdot)b$
with $a,b\in M\rtimes_{\sigma^{\varphi}}\mathbb{R}$ and $x\in p_{k}(M\rtimes_{\sigma^{\varphi}}\mathbb{R})p_{k}$
we have, using Lemma \ref{Lem=TomitaTakesaki} (recalling that $\mathbb{E}_{N}^{\widehat{\varphi}_{j}} = \mathbb{E}_{N} \circ T_{f_j}$) and the fact that $p_{k}$
commutes with $N$, that
\[
p_{k}\Phi(x)p_{k}=p_{k}a\mathbb{E}_{N}^{\widehat{\varphi}_{j}}(bp_{k}xp_{k})p_{k}=p_{k}a\mathbb{E}_{N}^{\widehat{\varphi}_{j}}(p_{k}bp_{k}x)p_{k}=(p_{k}ap_{k})\mathbb{E}_{N}^{\widehat{\varphi}_{j}}(p_{k}bp_{k}x)\text{.}
\]
Lemma \ref{Lem=NewExpectation} then implies that $(p_{k}\Phi(\:\cdot\:)p_{k})^{(2)}\in\mathcal{K}_{00}(p_{k}(M\rtimes_{\sigma^{\varphi}}\mathbb{R})p_{k},p_{k}Np_{k},p_{k}\widehat{\varphi}_{j}p_{k})$.
By taking linear combinations and approximating we see that if $\Phi$
is a map on $M\rtimes_{\sigma^{\varphi}}\mathbb{R}$ with $\Phi^{(2)}\in\mathcal{K}(M\rtimes_{\sigma^{\varphi}}\mathbb{R},N,\widehat{\varphi}_{j})$
then  $(p_{k}\Phi(\:\cdot\:)p_{k})^{(2)}\in\mathcal{K}(p_{k}(M\rtimes_{\sigma^{\varphi}}\mathbb{R})p_{k},p_{k}Np_{k},p_{k}\widehat{\varphi}_{j}p_{k})$.
Therefore for the approximating maps $\Phi_{i}$, $i\in I$ we conclude
that
\[
(p_{k}\Phi_{i}(\:\cdot\:)p_{k})^{(2)}\in\mathcal{K}(p_{k}(M\rtimes_{\sigma^{\varphi}}\mathbb{R})p_{k},p_{k}Np_{k},p_{k}\widehat{\varphi}_{j}p_{k})\text{.}
\]
This shows that $(p_{k}\Phi_{i}(\:\cdot\:)p_{k})_{i\in I}$ indeed
witnesses the relative Haagerup property of the triple $(p_{k}(M\rtimes_{\sigma^{\varphi}}\mathbb{R})p_{k},p_{k}Np_{k},p_{k}\widehat{\varphi}_{j}p_{k})$.\\

\emph{\eqref{Eqn=HardHAPEq3}}: Note that if $(M\rtimes_{\sigma^{\varphi}}\mathbb{R},N,\widehat{\varphi}_{j})$
has property $\text{(rHAP)}^{-}$ witnessed by the net $(\Phi_{i})_{i\in I}$,
then property $\text{(rHAP)}^{-}$ of $(p_{k}(M\rtimes_{\sigma^{\varphi}}\mathbb{R})p_{k},p_{k}Np_{k},p_{k}\widehat{\varphi}_{j}p_{k})$
follows in a very similar way as above. The only condition that  remains to be
checked is that the $L^{2}$-implementation $(p_{k}\Phi_{i}(\:\cdot\:)p_{k})^{(2)}$
exists. For this, assume that there exists $C>0$ with $\widehat{\varphi}_{j}(\Phi_{i}(x)^{\ast}\Phi_{i}(x))\leq C\widehat{\varphi}_{j}(x^{\ast}x)$
for all $x\in M\rtimes_{\sigma^{\varphi}}\mathbb{R}$. Then, using again \cite[Lemma2.3]{CS-IMRN} for the second inequality,

\begin{eqnarray}
\nonumber
(p_{k}\widehat{\varphi}_{j}p_{k})((p_{k}\Phi_{i}(x)p_{k})^{\ast}(p_{k}\Phi_{i}(x)p_{k}))&=& \widehat{\varphi}_{j}(p_{k}\Phi_{i}(x^{\ast})p_{k}\Phi_{i}(x)p_{k}) \\
\nonumber
&\leq& \widehat{\varphi}_{j}(p_{k}\Phi_{i}(x)^{\ast}\Phi_{i}(x)p_{k}) \\
\nonumber
&\leq& \widehat{\varphi}_{j}(\Phi_{i}(x)^{\ast}\Phi_{i}(x)) \\
\nonumber
&\leq& C\widehat{\varphi}_{j}(x^{\ast}x) \\
\nonumber
&=& C(p_{k}\widehat{\varphi}_{j}p_{k})(x^{\ast}x)
\end{eqnarray}
for all $x\in p_{k}(M\rtimes_{\sigma^{\varphi}}\mathbb{R})p_{k}$.
The claim follows. \\

\emph{Second part of \eqref{Eqn=HardHAPEq1}}: For the ``$\Leftarrow$''
direction assume that for every $k\in\mathbb{N}$ the triple $(p_{k}(M\rtimes_{\sigma^{\varphi}}\mathbb{R})p_{k},p_{k}Np_{k},p_{k}\widehat{\varphi}_{j}p_{k})$
satisfies property $\text{(rHAP)}$ witnessed by approximating maps
$(\Phi_{k,i})_{i\in I_{k}}$. We wish to apply Lemma \ref{Lem=ContractiveMaps}
for which we check the conditions. By $N\subseteq M^{\sigma^{\varphi}}$
we have that $N$ and $\lambda_{t}$ commute for every $t\in\mathbb{R}$
and hence so do $N$ and $h_{k}$. In particular, $h_{k}\in(p_{k}Np_{k})^{\prime}\cap p_{k}(M\rtimes_{\sigma^{\varphi}}\mathbb{R})p_{k}$.
By \eqref{Eqn=PosBdInv} and the remarks after it, it follows that
$\lambda(f_{j})h_{k}$ is positive and boundedly invertible. Now,
from \eqref{Eqn=HkRadonNykodym} we see that the conditions of Lemma
\ref{Lem=ContractiveMaps} are fulfilled and this lemma shows that
the maps of the net $(\Phi_{k,i})_{i\in I_{k}}$ can be chosen contractively,
i.e. we may assume that (1$^{\prime}$) holds.

We shall prove that $(\Phi_{k,i}(p_{k}\:\cdot\:p_{k}))_{k\in\mathbb{N},i\in I_{k}}$
induces a net witnessing property $\text{(rHAP)}$ of $(M\rtimes_{\sigma^{\varphi}}\mathbb{R},N,\widehat{\varphi}_{j})$.
This in particular shows that we may assume (1$^{\prime}$).

By the contractivity of the $\Phi_{k,i}$ it is clear that the maps
$\Phi_{k,i}(p_{k}\:\cdot\:p_{k})$ are completely positive with a
uniform bound on their norms.

Since $N$ and $p_{k}$ commute we see that for $a,b\in N$ and $x\in M\rtimes_{\sigma^{\varphi}}\mathbb{R}$
\begin{eqnarray}
\nonumber
\Phi_{k,i}(p_{k}axbp_{k}) &= & \Phi_{k,i}(p_{k}ap_{k}xp_{k}bp_{k}) \\
\nonumber
&=& p_{k}ap_{k}\Phi_{k,i}(p_{k}xp_{k})p_{k}bp_{k}\\
\nonumber
&= & ap_{k}\Phi_{k,i}(p_{k}xp_{k})p_{k}b \\
\nonumber
&=& a\Phi_{k,i}(p_{k}xp_{k})b\text{.}
\end{eqnarray}
 Therefore $\Phi_{k,i}(p_{k}\:\cdot\:p_{k})$ is an $N$-$N$ bimodule
map for every $k\in\mathbb{N}$, $i\in I_{k}$.

We have, using again \cite[Lemma 2.3]{CS-IMRN}, that for $x\in(M\rtimes_{\sigma^{\varphi}}\mathbb{R})^{+}$

\begin{eqnarray}
\nonumber
\widehat{\varphi}_{j}(\Phi_{k,i}(p_{k}xp_{k}))&=&\widehat{\varphi}_{j}(p_{k}\Phi_{k,i}(p_{k}xp_{k})p_{k})\\
\nonumber
&=& (p_{k}\widehat{\varphi}_{j}p_{k})(\Phi_{k,i}(p_{k}xp_{k})) \\
\nonumber
&\leq& (p_{k}\widehat{\varphi}_{j}p_{k})(p_{k}xp_{k})\\
\nonumber
&=& \widehat{\varphi}_{j}(p_{k}xp_{k})\leq\widehat{\varphi}_{j}(x).
\end{eqnarray}
i.e. $\widehat{\varphi}_{j}\circ\Phi_{k,i}(p_{k}\:\cdot\:p_{k})\leq\widehat{\varphi}_{j}$.

We claim that $(\Phi_{k,i}(p_{k}\:\cdot\:p_{k}))^{(2)}\in\mathcal{K}(M\rtimes_{\sigma^{\varphi}}\mathbb{R},N,\widehat{\varphi}_{j})$
for all $k\in\mathbb{N}$, $i\in I_{k}$. Indeed, take an arbitrary
map $\Phi$ of the form $\Phi(x)=p_{k}ap_{k}\mathbb{E}_{N}(T_{f_{j}}(p_{k}bp_{k}x))$
for $x\in M\rtimes_{\sigma^{\varphi}}\mathbb{R}$ where $a,b\in M\rtimes_{\sigma^{\varphi}}\mathbb{R}$.
The $L^{2}$-implementations of such operators span $\mathcal{K}_{00}(p_{k}(M\rtimes_{\sigma^{\varphi}}\mathbb{R})p_{k},p_{k}Np_{k},p_{k}\widehat{\varphi}_{j}p_{k})$
by Lemma \ref{Lem=NewExpectation}. Lemma \ref{Lem=TomitaTakesaki}
and the fact that $p_{k}$ and $N$ commute show that for $x\in M\rtimes_{\sigma^{\varphi}}\mathbb{R}$,
\[
\Phi(p_{k}xp_{k})=p_{k}ap_{k}\mathbb{E}_{N}(T_{f_{j}}(p_{k}bp_{k}xp_{k}))=p_{k}ap_{k}\mathbb{E}_{N}(T_{f_{j}}(p_{k}bp_{k}x)).
\]
 Then, since $\mathbb{E}_{N}\circ T_{f_{j}}$ is the faithful normal
$\widehat{\varphi}_{j}$-preserving conditional expectation of $M\rtimes_{\sigma^{\varphi}}\mathbb{R}$
onto $N$, this implies that $(\Phi(p_k \cdot p_k))^{(2)} \in\mathcal{K}_{00}(M\rtimes^{\sigma^{\varphi}}\mathbb{R},N,\widehat{\varphi}_{j})$.  
By taking linear combinations and approximation we see that if $\Phi^{(2)}\in\mathcal{K}(p_{k}(M\rtimes_{\sigma^{\varphi}}\mathbb{R})p_{k},p_{k}Np_{k},p_{k}\widehat{\varphi}_{j}p_{k})$,
then $( \Phi(p_k\cdot p_{k}))^{(2)} \in\mathcal{K}(M\rtimes^{\sigma^{\varphi}}\mathbb{R},N,\widehat{\varphi}_{j})$.
We conclude that
\[
(\Phi_{k,i}(p_{k}\:\cdot\:p_{k}))^{(2)}\in\mathcal{K}(M\rtimes^{\sigma^{\varphi}}\mathbb{R},N,\widehat{\varphi}_{j}).
\]

Now, for $x\in M\rtimes_{\sigma^{\varphi}}\mathbb{R}$ we see that
\[
\lim_{k\rightarrow\infty}\lim_{i\in I_{k}}\Phi_{k,i}(p_{k}xp_{k})=x,
\]
in the strong topology. Then a variant of Lemma \ref{Lem=ApproxSOT}
shows that there is a directed set $\mathcal{F}$ and a function $(\widetilde{k},\widetilde{i}):\mathcal{F}\rightarrow\{(k,i)\mid k\in\mathbb{N},i\in I_{k}\}$,
$F\mapsto(\widetilde{k}(F),\widetilde{i}(F))$ such $(\Phi_{\widetilde{k}(F),\widetilde{i}(F)})_{F\in\mathcal{F}}$
witnesses the relative Haagerup property of $(M\rtimes_{\sigma^{\varphi}}\mathbb{R},N,\widehat{\varphi}_{j})$.\\

\emph{\eqref{Eqn=HardHAPEq2}}: It only remains to show that if for
every $k\in\mathbb{N}$ the property $\text{(rHAP)}$ of the triple
$(p_{k}(M\rtimes_{\sigma^{\varphi}}\mathbb{R})p_{k},p_{k}Np_{k},p_{k}\widehat{\varphi}_{j}p_{k})$
is witnessed by unital $p_{k}\widehat{\varphi}_{j}p_{k}$-preserving
approximating maps, then the relative Haagerup property of $(M\rtimes_{\sigma^{\varphi}}\mathbb{R},N,\widehat{\varphi}_{j})$
is witnessed by unital $\widehat{\varphi}_{j}$-preserving maps. For
this, assume that the maps $(\Phi_{k,i})_{i\in I}$ from before are
unital and $p_{k}\widehat{\varphi}_{j}p_{k}$-preserving and choose
a sequence $(\epsilon_{k})_{k\in\mathbb{N}}\subseteq\left(0,1\right)$
with $\epsilon_{k}\rightarrow0$. Recall that $p_{k}\in N^{\prime}\cap(M\rtimes_{\sigma^{\varphi}}\mathbb{R})$
and note that $\mathbb{E}_{N}^{\widehat{\varphi}_{j}}(1-(1-\epsilon_{k})p_{k})\geq\epsilon_{k}$.
We then have $\mathbb{E}_{N}^{\widehat{\varphi}_{j}}(1-(1-\epsilon_{k})p_{k})\in N\cap N^{\prime}$,
the inverse $(\mathbb{E}_{N}^{\widehat{\varphi}_{j}}(1-(1-\epsilon_{k})p_{k}))^{-1}\in N\cap N^{\prime}$
exists and $a_{k}:=(1-(1-\epsilon_{k})p_{k})(\mathbb{E}_{N}^{\widehat{\varphi}_{j}}(1-(1-\epsilon_{k})p_{k}))^{-1}\in N^{\prime}\cap(M\rtimes_{\sigma^{\varphi}}\mathbb{R})$
is positive. Set $b_{k}:=1-(1-\epsilon_{k})p_{k} geq 0$. Define the maps
\begin{eqnarray}
\nonumber
\widetilde{\Phi}_{k,i}(\: \cdot \: ):=(1-\epsilon_{k})\Phi_{k,i}(p_{k}\:\cdot\:p_{k})+a_{k}\mathbb{E}_{N}^{\widehat{\varphi}_{j}}(b_{k}^{1/2}\:\cdot\:b_{k}^{1/2}).
\end{eqnarray}
Obviously $\widetilde{\Phi}_{k,i}$ is normal, completely positive and $N$-$N$-bimodular.  We may finish the proof as in Theorem \ref{BannonFangApproach} now; since the statement of that theorem is not directly applicable here we will give the complete proof for the convenience of the reader.

We have
\begin{eqnarray}
\nonumber
\widetilde{\Phi}_{k,i}(1)=(1-\epsilon_{k})\Phi_{k,i}(p_{k})+a_{k}\mathbb{E}_{N}^{\widehat{\varphi}_{j}}(b_{k})=(1-\epsilon_{k})p_{k}+(1-(1-\epsilon_{k})p_{k})=1\text{.}
\end{eqnarray}
Now, since $\Phi_{k,i}$ is $p_{k}\widehat{\varphi}_{j}p_{k}$-preserving
we have that $\widehat{\varphi}_{j}\circ\Phi_{k,i}(p_{k}xp_{k})=\widehat{\varphi_{j}}(p_{k}xp_{k})$
for all $x\in M\rtimes_{\sigma^{\varphi}}\mathbb{R}$, and hence with
Lemma \ref{Lem=TomitaTakesaki} we deduce that
\begin{eqnarray}
\nonumber
\widehat{\varphi}_{j}\circ\widetilde{\Phi}_{k,i}(x) &=& (1-\epsilon_{k})\widehat{\varphi}_{j}(\Phi_{k,i}(p_{k}xp_{k}))+\widehat{\varphi}_{j}(a_{k}\mathbb{E}_{N}^{\widehat{\varphi}_{j}}(b_{k}^{1/2}xb_{k}^{1/2}))\\
\nonumber
&=& (1-\epsilon_{k})\widehat{\varphi}_{j}(p_{k}xp_{k})+\widehat{\varphi}_{j}(\mathbb{E}_{N}^{\widehat{\varphi}_{j}}(a_{k})b_{k}^{1/2}xb_{k}^{1/2}) \\
\nonumber
&=& (1-\epsilon_{k})\widehat{\varphi}_{j}\circ\mathbb{E}_{N}^{\widehat{\varphi}_{j}}(p_{k}xp_{k})+\widehat{\varphi}_{j}\circ\mathbb{E}_{N}^{\widehat{\varphi}_{j}}(b_{k}^{1/2}xb_{k}^{1/2})\\
\nonumber
&=& (1-\epsilon_{k})\widehat{\varphi}_{j}\circ\mathbb{E}_{N}^{\widehat{\varphi}_{j}}(p_{k}x)+\widehat{\varphi}_{j}\circ\mathbb{E}_{N}^{\widehat{\varphi}_{j}}(b_{k}x)\\
\nonumber
&=& (1-\epsilon_{k})\widehat{\varphi}_{j}(p_{k}x)+\widehat{\varphi}_{j}(b_{k}x) \\
\nonumber
&=& \widehat{\varphi}_{j}(x)\text{.}
\end{eqnarray}
By the fact that  $(\Phi_{k,i}(p_{k}\:\cdot\:p_{k}))^{(2)}\in\mathcal{K}(M\rtimes_{\sigma^{\varphi}}\mathbb{R},N,\widehat{\varphi}_{j})$
and by Lemma \ref{Lem=TomitaTakesaki}, we have
\begin{eqnarray}
\nonumber
\widetilde{\Phi}_{k,i}^{(2)}=(1-\epsilon_{k})(\Phi_{k,i}(p_{k}\:\cdot\:p_{k}))^{(2)}+a_{k}e_{N}^{\widehat{\varphi}_{j}}b_{k}\in\mathcal{K}(M\rtimes_{\sigma^{\varphi}}\mathbb{R},N,\widehat{\varphi}_{j})\text{.}
\end{eqnarray}
Further, for every $x\in(M\rtimes_{\sigma^{\varphi}}\mathbb{R})_{+}$,
\begin{eqnarray}
\nonumber
\widetilde{\Phi}_{k,i}(x)-(1-\epsilon_{k})\Phi_{k,i}(p_{k}xp_{k}) &=& a_{k}\mathbb{E}_{N}^{\widehat{\varphi}_{j}}(b_{k}^{1/2}xb_{k}^{1/2}) \\
\nonumber
&\leq& \left\Vert x\right\Vert a_{k}\mathbb{E}_{N}^{\widehat{\varphi}_{j}}(b_{k})\\
\nonumber
&=& \left\Vert x\right\Vert (1-(1-\epsilon_{k})p_{k})\text{,}
\end{eqnarray} 
from which we deduce that $\lim_{F\in\mathcal{F}}\widetilde{\Phi}_{\widetilde{k}(F),\widetilde{i}(F)}=\text{id}_{M\rtimes_{\sigma^{\varphi}}\mathbb{R}}$.
This implies that the net $(\widetilde{\Phi}_{\widetilde{k}(F),\widetilde{i}(F)})_{F\in\mathcal{F}}$
of unital $\widehat{\varphi}_{j}$-preserving maps witnesses the relative
Haagerup property of $(M\rtimes_{\sigma^{\varphi}}\mathbb{R},N,\widehat{\varphi}_{j})$.
\end{proof}

We are now ready to formulate the key statement of this section. Note
that for every $k\in\mathbb{N}$ the von Neumann algebra $p_{k}(M\rtimes_{\sigma^{\varphi}}\mathbb{R})p_{k}$
is finite with a faithful normal tracial state $p_{k}\tau_{\rtimes}p_{k}$.

\begin{thm} \label{Prop=HAPtoCorner}
Let $N\subseteq M$
be a unital inclusion of von Neumann algebras which admits a faithful
normal conditional expectation $\mathbb{E}_{N}$. Assume that $N$
is finite and let $\tau\in N_{\ast}$ be a faithful normal tracial
state that we extend to a state $\varphi:=\tau\circ\mathbb{E}_{N}$
on $M$. Then the following are equivalent:

\begin{enumerate}
\item   \label{Eqn=HAPEq1} The triple $(M,N,\varphi)$ has property
$\text{(rHAP)}$;
\item   \label{Eqn=HAPEq2} $(M\rtimes_{\sigma^{\varphi}}\mathbb{R},N,\widehat{\varphi}_{j})$
has property $\text{(rHAP)}$ for every $j\in\mathbb{N}$;
\item   \label{Eqn=HAPEq3} $(p_{k}(M\rtimes_{\sigma^{\varphi}}\mathbb{R})p_{k},p_{k}Np_{k},p_{k}\tau_{\rtimes}p_{k})$
has property $\text{(rHAP)}$ for every $k\in\mathbb{N}$.
\end{enumerate}
Further, the following statement holds:
\begin{enumerate} \setcounter{enumi}{3}
\item  \label{Eqn=HAPEq4} If the triple $(M,N,\varphi)$ has property
$\text{(rHAP)}^{-}$, then $(p_{k}(M\rtimes_{\sigma^{\varphi}}\mathbb{R})p_{k},p_{k}Np_{k},p_{k}\tau_{\rtimes}p_{k})$
has property $\text{(rHAP)}^-$ for every $k\in\mathbb{N}$.
\end{enumerate}
\end{thm}

\begin{proof}
The equivalence ``\eqref{Eqn=HAPEq1} $\Leftrightarrow$
\eqref{Eqn=HAPEq2}'' was proved in Theorem \ref{Thm=NormalToCross}.

``\eqref{Eqn=HAPEq2} $\Rightarrow$ \eqref{Eqn=HAPEq3}'': Assume
that for $j\in\mathbb{N}$ the triple $(M\rtimes_{\sigma^{\varphi}}\mathbb{R},N,\widehat{\varphi}_{j})$
has property $\text{(rHAP)}$ and fix $k \in \mathbb{N}$. Then by Proposition \ref{Prop=ToCorner}, the triple $(p_{k}(M\rtimes_{\sigma^{\varphi}}\mathbb{R})p_{k},p_{k}Np_{k},p_{k}\widehat{\varphi}_{j}p_{k})$ also has the (rHAP). Let $(\Phi_{i})_{i\in I}$ be a net
of suitable approximating maps and define the
self-adjoint boundedly invertible operator $A_{j,k}:=\lambda(f_{j})h_{k}^{1/2}\in(p_{k}Np_{k})^{\prime}\cap(p_{k}(M\rtimes_{\sigma^{\varphi}}\mathbb{R})p_{k})$.
By \eqref{Eqn=HkRadonNykodym} for every $x\in p_{k}(M\rtimes_{\sigma^{\varphi}}\mathbb{R})p_{k}$
the equality
\[
(p_{k}\widehat{\varphi}_{j}p_{k})(x)=\tau_{\rtimes}(A_{j,k}^{\ast}xA_{j,k})=(A_{j,k}p_{k}\tau_{\rtimes}p_{k}A_{j,k})(x)
\]
holds and hence Lemma \ref{Lem=CompactTwist} implies that the $L^{2}$-implementation
of the map $\Phi_{i}^{\prime}(\: \cdot \:)  :=A_{j,k}\Phi_{i}(A_{j,k}^{-1}\:\cdot\:A_{j,k}^{-1})A_{j,k}$
exists and is contained in $\mathcal{K}(p_{k}(M\rtimes_{\sigma^{\varphi}}\mathbb{R})p_{k},p_{k}Np_{k},p_{k}\tau_{\rtimes}p_{k})$.
Similarly to the proof of Proposition \ref{Prop=ToCorner} one checks
that the net $(\Phi_{i}^{\prime})_{i\in I}$ witnesses property $\text{(rHAP)}$
of $(p_{k}(M\rtimes_{\sigma^{\varphi}}\mathbb{R})p_{k},p_{k}Np_{k},p_{k}\tau_{\rtimes}p_{k})$.
We omit the details.

``\eqref{Eqn=HAPEq2} $\Leftarrow$ \eqref{Eqn=HAPEq3}'' Now assume
that the triple $(p_{k}(M\rtimes_{\sigma^{\varphi}}\mathbb{R})p_{k},p_{k}Np_{k},p_{k}\tau_{\rtimes}p_{k})$
has property $\text{(rHAP)}$ for every $k\in\mathbb{N}$. It suffices
to show that the triple $(p_{k}(M\rtimes_{\sigma^{\varphi}}\mathbb{R})p_{k},p_{k}Np_{k},p_{k}\widehat{\varphi}_{j}p_{k})$
has property $\text{(rHAP)}$ as it implies the desired statement
by Proposition \ref{Prop=ToCorner}. So let $(\Phi_{i})_{i\in I}$
be a net that witnesses property $\text{(rHAP)}$ of $(p_{k}(M\rtimes_{\sigma^{\varphi}}\mathbb{R})p_{k},p_{k}Np_{k},p_{k}\tau_{\rtimes}p_{k})$
and set $\Phi_{i}^{\prime}:=A_{j,k}^{-1}\Phi_{i}(A_{j,k}\:\cdot\:A_{j,k})A_{j,k}^{-1}\text{.}$
Lemma \ref{Lem=CompactTwist} and \eqref{Eqn=HkRadonNykodym} imply
that for every $i\in I$ the $L^{2}$-implementation $(\Phi_{i}^{\prime})^{(2)}$
of $\Phi_{i}^{\prime}$ with respect to the positive functional $p_{k}\widehat{\varphi_{j}}p_{k}$
is contained in $\mathcal{K}(p_{k}(M\rtimes_{\sigma^{\varphi}}\mathbb{R})p_{k},p_{k}Np_{k},p_{k}\widehat{\varphi}_{j}p_{k})$.
Again, similarly to the proof of Proposition \ref{Prop=ToCorner}
one checks that the net $(\Phi_{i}^{\prime})_{i\in I}$ witnesses
property $\text{(rHAP)}$.

It remains to show  \eqref{Eqn=HAPEq4}. The statement easily follows
from Proposition \ref{Thm=NormalToCross},  Proposition \ref{Prop=ToCorner} and the arguments used in the proof of the implication ``\eqref{Eqn=HAPEq2} $\Rightarrow$ \eqref{Eqn=HAPEq3}''.
\end{proof}

\vspace{1mm}


\section{Main results}  \label{MainResults}

After the main work has been done in Section \ref{Preparation} we
can now put the pieces together. This allows us to show that in the
case of a finite von Neumann subalgebra the notion of relative Haagerup
property is independent of the choice of the corresponding faithful
normal conditional expectation, that the approximating maps may be
chosen to be unital and state-preserving and that property $\text{(rHAP)}$
and property $\text{(rHAP)}^{-}$ are equivalent. The general notation
will be the same as in Section \ref{Preparation}.


\subsection{Independence of the conditional expectation}

Let $ N\subseteq M$ be a unital inclusion of von Neumann
algebras for which $N$ is finite with a faithful normal tracial state
$\tau\in N_{\ast}$. Let further $\mathbb{E}_{N},\mathbb{F}_{N}:M\rightarrow N$
be two faithful normal conditional expectations and extend $\tau$
to states $\varphi:=\tau\circ\mathbb{E}_{N}$ and $\psi:=\tau\circ\mathbb{F}_{N}$
on $M$. In this subsection we will prove that the triple $(M,N,\mathbb{E}_{N})$
has property $\text{(rHAP)}$ if and only if the triple $(M,N,\mathbb{F}_{N})$
does, i.e.\ the relative Haagerup property is an intrinsic invariant
of the inclusion $N\subseteq M$.   Let us first introduce some notation.

As in Section \ref{Preparation} consider the crossed product von
Neumann algebra $M\rtimes_{\sigma^{\varphi}}\mathbb{R}$ which contains
the projections $p_{k}\in N^{\prime}\cap(M\rtimes_{\sigma^{\varphi}}\mathbb{R})$,
$k\in\mathbb{N}$ and carries the canonical normal semi-finite tracial
weight $\tau_{\rtimes}$ which we will from now on denote by $\tau_{\rtimes,1}$.
For $t\in\mathbb{R}$ write $\lambda_{t}^{\varphi}$ for the left
regular representation operators in $M\rtimes_{\sigma^{\varphi}}\mathbb{R}$.
Similarly, we write $\tau_{\rtimes,2}$ for the canonical normal semi-finite
tracial weight on $M\rtimes_{\sigma^{\psi}}\mathbb{R}$ and denote
the corresponding left regular representation operators by $\lambda_{t}^{\psi}$,
$t\in\mathbb{R}$.

For $t\in\mathbb{R}$ let $u_{t}:=(D\varphi/D\psi)_{t}\in M$ be the
Connes cocycle Radon-Nikodym derivative, so in particular $u_{t}\sigma_{t}^{\varphi}(u_{s})=u_{t+s}$ and $\sigma_t^\psi(x) = u_t^* \sigma_t^\varphi(x) u_t$ hold
for all $s,t\in\mathbb{R}$. Then (see \cite[Proof of Theorem X.1.7]{Takesaki2}) there exists an isomorphism $\rho:M\rtimes_{\sigma^{\psi}}\mathbb{R}\rightarrow M\rtimes_{\sigma^{\varphi}}\mathbb{R}$
of von Neumann algebras which restricts to the identity on $M$ and
for which $\rho(\lambda_{t}^{\psi})=u_{t}\lambda_{t}^{\varphi}$ for
all $t\in\mathbb{R}$. This implies that the dual actions $\theta^{\varphi}$
and $\theta^{\psi}$ of $\sigma^{\varphi}$ and $\sigma^{\psi}$ respectively
are related by the equality $\theta_{t}^{\varphi}\circ\rho=\rho\circ\theta_{t}^{\psi}\text{, }t\in\mathbb{R}$.
Further, $\tau_{\rtimes,1}\circ\rho=\tau_{\rtimes,2}$ (see the footnote \footnote{
This is well-known to specialists, but it seems that the statement does not appear explicitly in \cite{Takesaki2}. The argument goes as follows.  Firstly, as $\rho$ intertwines the dual actions on $M\rtimes_{\sigma^{\psi}}\mathbb{R}$ and $M\rtimes_{\sigma^{\varphi}}\mathbb{R}$  we find that $\widehat{\varphi} \circ \rho$ is the dual weight of $\varphi$ in the crossed product $M\rtimes_{\sigma^{\psi}}\mathbb{R}$. Let  $t \in \mathbb{R}$. By \cite[Theorem X.1.17]{Takesaki2} we have  Connes cocycle derivative $\left( \frac{  D  \widehat{\psi}  }{  D \widehat{\varphi} \circ \rho  } \right)_t = u_t = \rho(u_t)$. Then by the chain rule \cite[Theorem VIII.3.7]{Takesaki2},
\[
\begin{split}
\left( \frac{  D\tau_{\rtimes, 2}  }{ D  \tau_{\rtimes, 1} \circ \rho  } \right)_t =  &
\left( \frac{  D\tau_{\rtimes, 2}  }{ D \widehat{\psi}   } \right)_t
\left( \frac{  D  \widehat{\psi}  }{  D \widehat{\varphi} \circ \rho  } \right)_t
\left( \frac{  D \widehat{\varphi} \circ \rho   }{ D  \tau_{\rtimes, 1} \circ \rho  } \right)_t
=   \lambda^{\psi}_{-t}  \rho^{-1}( u_{t}  \lambda^\varphi_t) = 1.
\end{split}
\]
Hence $\tau_{\rtimes, 1} \circ \rho = \tau_{\rtimes, 2} $.}).
Denote by $h_{\psi}$ the unique unbounded self-adjoint positive operator
affiliated with $M\rtimes_{\sigma^{\psi}}\mathbb{R}$ such that $h_{\psi}^{it}=\lambda_{t}^{\psi}$
for all $t\in\mathbb{R}$ and set
\[
p_{\psi,k}:=\chi_{[k^{-1},k]}(h_{\psi})\qquad\text{ and }\qquad q_{k}:=\rho(p_{\psi,k})\text{.} \qquad
\]
for $k\in\mathbb{N}$. Further, define
\[
h_{\psi,k}:=\rho(\chi_{[k^{-1},k]}(h_{\psi})h_{\psi})=\rho(p_{\psi,k}h_{\psi}).
\]
Recall that  for all $k \in \mathbb{N}$ and $t \in \mathbb{R}$ we write $h^{it} = \lambda_t^{\varphi}$, $p_k:= \chi_{[k^{-1},k]}(h)$, and $h_{k} := p_k h$.

The following statement compares to Lemma \ref{Lem=NewExpectation}.

\begin{lem} \label{Lem=CEVComp1}
For every $k\in\mathbb{N}$
there is a (unique) faithful normal $p_{k}\tau_{\rtimes,1}p_{k}$-preserving
conditional expectation $\mathbb{E}_{1,k}:p_{k}(M\rtimes_{\sigma^{\varphi}}\mathbb{R})p_{k}\rightarrow p_{k}Np_{k}$
given by
\begin{equation}\label{Eqn=PostArXivCEV}
\nonumber
x\mapsto\nu_{k}^{-1}p_{k}\mathbb{E}_{N}(T_{\theta^{\varphi}}(h_{k}^{-1}x))p_{k}\text{,}
\end{equation}
where $\nu_{k}:=T_{\theta^{\varphi}}(h_{k}^{-1})=k-k^{-1}$.
In particular, $T_{\theta^{\varphi}}(h_{k}^{-1})$ is a scalar multiple
of the identity.
\end{lem}

\begin{proof}
The proof is essentially the same as that of Lemma
\ref{Lem=NewExpectation}. First note that by Remark \ref{Rmk=Fourier2}
the operator $h_{k}$ coincides with $\lambda(\widehat{J_{k}})$ where
$J_{k}(s) =  \chi_{[-\log(k),\log(k)]}e^{s}$ and that $\nu_{k}=T_{\theta^{\varphi}}(h_{k}^{-1})= k-k^{-1}$
is a multiple of the identity. For $x\in p_{k}(M\rtimes_{\sigma^{\varphi}}\mathbb{R})p_{k}$
one checks using \eqref{Eqn=TauExpressionExtra} for the second and last equality, that
\begin{eqnarray}
\nonumber
(p_{k}\tau_{\rtimes,1}p_{k})(p_{k}\mathbb{E}_{N}(T_{\theta^{\varphi}}(h_{k}^{-1}x))p_{k})&=&\tau_{\rtimes,1}(p_{k}\mathbb{E}_{N}(T_{\theta^{\varphi}}(h_{k}^{-1}x))p_{k})\\
\nonumber
&=&\varphi\circ T_{\theta^{\varphi}}(p_{k}h_{k}^{-1}\mathbb{E}_{N}(T_{\theta^{\varphi}}(h_{k}^{-1}x))p_{k})\\
\nonumber
&=&\varphi\circ T_{\theta^{\varphi}}(h_{k}^{-1}\mathbb{E}_{N}(T_{\theta^{\varphi}}(h_{k}^{-1}x)))\\
\nonumber
&=&\varphi\left(T_{\theta^{\varphi}}(h_{k}^{-1})\mathbb{E}_{N}(T_{\theta^{\varphi}}(h_{k}^{-1}x))\right)\\
\nonumber
&=&\nu_{k}\varphi\left(\mathbb{E}_{N}(T_{\theta^{\varphi}}(h_{k}^{-1}x))\right)\\
\nonumber
&=&\nu_{k}\varphi\circ T_{\theta^{\varphi}}(h_{k}^{-1}x))\\
\nonumber
&=&\nu_{k}\tau_{\rtimes,1}(p_{k}xp_{k})\text{,}
\end{eqnarray}
hence $\mathbb{E}_{1,k}$ is indeed $p_{k}\tau_{\rtimes,1}p_{k}$-preserving.
Here we used in the fourth line that $N$ is invariant under the dual
action $\theta^{\varphi}$ and in the fifth line that $T_{\theta^{\varphi}}(h_{k}^{-1})$
is a multiple of the identity.

From Lemma \ref{Lem=TomitaTakesaki} we see that 
\[
\nu_{k}^{-1}p_{k}\mathbb{E}_{N}(T_{\theta^{\varphi}}(h_{k}^{-1} \: \cdot \: ))p_{k} = \nu_{k}^{-1}p_{k}\mathbb{E}_{N}(T_{\theta^{\varphi}}(h_{k}^{-1/2}  \: \cdot \:   h_{k}^{-1/2}))p_{k},
\]
and from the right hand side of this expression it is clear that \eqref{Eqn=PostArXivCEV} is completely positive.
The remaining statements (i.e. that
$\mathbb{E}_{1,k}$ is a unital faithful normal  
$p_{k}Np_{k}$-$p_{k}Np_{k}$-bimodule map) are then easy to check.

\end{proof}

The following lemma provides the analogous statement for the functional
$q_{k}\tau_{\rtimes,1}q_{k}$ and the inclusion $q_{k}Nq_{k}\subseteq q_{k}(M\rtimes_{\sigma^{\varphi}}\mathbb{R})q_{k}$.
We omit the proof.

\begin{lem} \label{Lem=CEVComp2}
For every $k\in\mathbb{N}$
there is a (unique) faithful normal $q_{k}\tau_{\rtimes,1}q_{k}$-preserving
conditional expectation $\mathbb{E}_{2,k}:q_{k}(M\rtimes_{\sigma^{\varphi}}\mathbb{R})q_{k}\rightarrow q_{k}Nq_{k}$
given by
\begin{equation}
\nonumber
x\mapsto\nu_{k}^{-1}q_{k}\mathbb{E}_{N}(T_{\theta^{\varphi}}(h_{\psi,k}^{-1}x))q_{k}\text{,}
\end{equation}
where $\nu_{k}:= k-k^{-1}$ as before.
\end{lem}

\begin{prop} \label{Prop=CornerToCorner}
Let $N\subseteq M$
be a unital inclusion of von Neumann algebras for which $N$ is finite
with a faithful normal tracial state $\tau\in N_{\ast}$. Let further
$\mathbb{E}_{N},\mathbb{F}_{N}:M\rightarrow N$ be two faithful
normal conditional expectations and extend $\tau$ to states $\varphi:=\tau\circ\mathbb{E}_{N}$
and $\psi:=\tau\circ\mathbb{F}_{N}$ on $M$. Then the following statements
are equivalent:
\begin{enumerate}
\item   \label{Eqn=CornerHAPEq1} For every $k\in\mathbb{N}$ the triple
$(p_{k}(M\rtimes_{\sigma^{\varphi}}\mathbb{R})p_{k},p_{k}Np_{k},p_{k}\tau_{\rtimes,1}p_{k})$
has property $\text{(rHAP)}$.
\item   \label{Eqn=CornerHAPEq2} For every $k\in\mathbb{N}$ the triple
$(q_{k}(M\rtimes_{\sigma^{\varphi}}\mathbb{R})q_{k},q_{k}Nq_{k},q_{k}\tau_{\rtimes,1}q_{k})$
has property $\text{(rHAP)}$.
\end{enumerate}
\end{prop}

\begin{proof}
By symmetry it suffices to consider the direction
``\eqref{Eqn=CornerHAPEq2} $\Rightarrow$ \eqref{Eqn=CornerHAPEq1}''.
For this, fix $k,l\in\mathbb{N}$ and let $(\Phi_{l,i})_{i\in I_{l}}$
be a net of maps witnessing the relative Haagerup property of the
triple $(q_{l}(M\rtimes_{\sigma^{\varphi}}\mathbb{R})q_{l},q_{l}Nq_{l},q_{l}\tau_{\rtimes,1} q_{l})$,
which we can assume to be contractive by Lemma \ref{Lem=ContractiveMaps}.
Define for $i\in I_{l}$ the normal completely positive contractive
map
\[
\Phi_{k,l,i}^{\prime}:p_{k}(M\rtimes_{\sigma^{\varphi}}\mathbb{R})p_{k}\rightarrow p_{k}(M\rtimes_{\sigma^{\varphi}}\mathbb{R})p_{k}:
 x \mapsto p_{k}\Phi_{l,i}(q_{l}xq_{l})p_{k}.
\]
As $N\subseteq M^{\sigma^{\varphi}}$ and $N\subseteq M^{\sigma^{\psi}}$,
$N$ commutes with both $q_{l}$ and $p_{k}$. Thus we have that for $x\in p_{k}(M\rtimes_{\sigma^{\varphi}}\mathbb{R})p_{k}$
and $a,b\in N$
\begin{eqnarray}
\nonumber
 \Phi_{k,l,i}^{\prime}(p_{k}ap_{k}xp_{k}bp_{k}) &=& p_{k}\Phi_{l,i}(q_{l}p_{k}ap_{k}xp_{k}bp_{k}q_{l})p_{k}= p_{k}\Phi_{l,i}(q_{l}axbq_{l})p_{k} = p_{k}\Phi_{l,i}(q_{l}aq_{l}xq_{l}bq_{l})p_{k}\\
\nonumber
&=&  p_{k}q_{l}aq_{l}\Phi_{l,i}(q_{l}xq_{l})q_{l}bq_{l}p_{k} = p_{k}a\Phi_{l,i}(q_{l}xq_{l})bp_{k} = p_{k}ap_{k}\Phi_{l,i}(q_{l}xq_{l})p_{k}bp_{k} \\
\nonumber
&=& p_{k}ap_{k}\Phi_{k,l,i}^{\prime}(x)p_{k}bp_{k}\text{,}
\end{eqnarray}
i.e. $\Phi_{k,l,i}'$ is $p_{k}Np_{k}$-$p_{k}Np_{k}$-bimodular. Further,
$(p_{k}\tau_{\rtimes,1}p_{k})\circ\Phi_{k,l,i}^{\prime}\leq p_{k}\tau_{\rtimes,1}p_{k}$
since for all positive $x\in p_{k}(M\rtimes_{\sigma^{\varphi}}\mathbb{R})p_{k}$ we have
\begin{eqnarray}
\nonumber
(p_{k}\tau_{\rtimes,1}p_{k})\circ\Phi_{k,l,i}^{\prime}(x) &=& \tau_{\rtimes_{1}}(p_{k}\Phi_{l,i}(q_{l}p_{k}xp_{k}q_{l})p_{k})\leq \tau_{\rtimes,1}(\Phi_{l,i}(q_{l}p_{k}xp_{k}q_{l})) \\
\nonumber
&=& \tau_{\rtimes,1}(q_{l}\Phi_{l,i}(q_{l}p_{k}xp_{k}q_{l})q_{l}) \leq \tau_{\rtimes,1}(q_{l}p_{k}xp_{k}q_{l}) \\
\nonumber
&\leq& (p_{k}\tau_{\rtimes,1}p_{k})(x).
\end{eqnarray}

For every map $\Phi$ on $q_{l}(M\rtimes_{\sigma^{\varphi}}\mathbb{R})q_{l}$
of the form $\Phi=a\mathbb{E}_{2,l}b$ with $a,b\in q_{l}(M\rtimes_{\sigma^{\varphi}}\mathbb{R})q_{l}$
and $x\in p_{k}(M\rtimes_{\sigma^{\varphi}}\mathbb{R})p_{k}$ we have
by Lemma \ref{Lem=CEVComp2} that {
\begin{eqnarray}
\nonumber
p_{k}\Phi(q_{l}xq_{l})p_{k}&=&p_{k}a\mathbb{E}_{2,l}(bq_{l}xq_{l})p_{k} \\
\nonumber
&=& \nu_{l}^{-1}p_{k}aq_{l}\mathbb{E}_{N}(T_{\theta^{\varphi}}(h_{\psi,l}^{-1}bq_{l}xq_{l}))q_{l}p_{k}\\
\nonumber
&=& \nu_{l}^{-1}p_{k}ap_{k}\mathbb{E}_{N}(T_{\theta^{\varphi}}(h_{\psi,l}^{-1}bq_{l}xq_{l})).\\
\nonumber
\end{eqnarray}
Now we may use the isomorphism $\rho$ and apply Lemma \ref{Lem=TomitaTakesaki}  to $M \rtimes_{\sigma^\psi} \mathbb{R}$ to get
\begin{eqnarray*}
p_{k}\Phi(q_{l}xq_{l})p_{k}&=& 
 \nu_{l}^{-1}p_{k}ap_{k}\mathbb{E}_{N}(T_{\theta^{\psi}}(  p_{\psi,l} h_{\psi}^{-1}  \rho^{-1}(b)   p_{\psi, l} \rho^{-1}(x)   p_{\psi,l})) \\
&=&
 \nu_{l}^{-1}p_{k}ap_{k}\mathbb{E}_{N}(T_{\theta^{\psi}}(  p_{\psi,l} h_{\psi}^{-1}  \rho^{-1}(b)   p_{\psi, l} \rho^{-1}(x)   ) \\
 &=&
 \nu_{l}^{-1}p_{k}ap_{k}\mathbb{E}_{N}(T_{\theta^{\varphi}}(q_{l}h_{\psi,l}^{-1}bq_{l}x)). 
\end{eqnarray*}
Then  by  Lemma \ref{Lem=TomitaTakesaki} applied to  $M \rtimes_{\sigma^\varphi} \mathbb{R}$ for the second equality    and  Lemma  \ref{Lem=CEVComp1} for the last equality, we find
\begin{eqnarray*}
p_{k}\Phi(q_{l}xq_{l})p_{k}  &=&
 \nu_{l}^{-1}p_{k}ap_{k}\mathbb{E}_{N}(T_{\theta^{\varphi}}(q_{l}h_{\psi,l}^{-1}bq_{l}x p_k)) \\
&=& \nu_{l}^{-1}p_{k}ap_{k}\mathbb{E}_{N}(T_{\theta^{\varphi}}(h_{k}^{-1}(h_{k}q_{l}h_{\psi,l}^{-1}bq_{l}x)))p_{k}\\
\nonumber
&=& \nu_{k}\nu_{l}^{-1}p_{k}a\mathbb{E}_{1,k}((h_{k}q_{l}h_{\psi,l}^{-1}bq_{l})x).
\end{eqnarray*}} 
Thus $(p_{k}\Phi(q_{l}\:\cdot\:q_{l})p_{k})^{(2)}\in\mathcal{K}_{00}(p_{k}(M\rtimes_{\sigma^{\varphi}}\mathbb{R})p_{k},p_{k}Np_{k},p_{k}\tau_{\rtimes,1}p_{k})$. By taking linear combinations and approximation we see that if $\Phi^{(2)}\in\mathcal{K}(q_{l}(M\rtimes_{\sigma^{\varphi}}\mathbb{R})q_{l},q_{l}Nq_{l},q_{l}\tau_{\rtimes, 1}q_{l})$,
then also  $p_{k}\Phi(q_{l}\:\cdot\:q_{l})p_{k})^{(2)}\in\mathcal{K}(p_{k}(M\rtimes_{\sigma^{\varphi}}\mathbb{R})p_{k},p_{k}Np_{k},p_{k}\tau_{\rtimes,1}p_{k})$.
In particular, $(\Phi_{k,l,i}^{\prime})^{(2)}\in\mathcal{K}(p_{k}(M\rtimes_{\sigma^{\varphi}}\mathbb{R})p_{k},p_{k}Np_{k},p_{k}\tau_{\rtimes,1}p_{k})$
for $k,l\in\mathbb{N}$ and $i\in I_{l}$.

For every $x\in p_{k}(M\rtimes_{\sigma^{\varphi}}\mathbb{R})p_{k}$
we have that
\[
\lim_{l\rightarrow\infty}\lim_{i\in I_{l}}\Phi_{k,l,i}^{\prime}(x)=x
\]
 in the strong topology. A variant of Lemma \ref{Lem=ApproxSOT} then shows that
there is a directed set $\mathcal{F}$ and an increasing function
$(\widetilde{l},\widetilde{i}):\mathcal{F}\rightarrow\{(l,i)\mid k\in\mathbb{N},i\in I_{l}\}$,
$F\mapsto(\widetilde{l}(F),\widetilde{i}(F))$ such that $(\Phi_{k,\widetilde{l}(F),\widetilde{i}(F)}^{\prime})_{F\in\mathcal{F}}$
witnesses the relative Haagerup property of $(p_{k}(M\rtimes_{\sigma^{\varphi}}\mathbb{R})p_{k},p_{k}Np_{k},p_{k}\tau_{\rtimes,1}p_{k})$.
\end{proof}

\begin{thm} \label{thm:condexpindep}
Let $ N\subseteq M$ be a unital inclusion
of von Neumann algebras with $N$  finite. Let   $\mathbb{E}_{N},\mathbb{F}_{N}:M\rightarrow N$
be two faithful normal conditional expectations. Then the triple $(M,N,\mathbb{E}_{N})$
has  $\text{(rHAP)}$   if and only if the triple $(M,N,\mathbb{F}_{N})$
has  $\text{(rHAP)}$.
\end{thm}

\begin{proof}
Assume that the triple $(M,N,\mathbb{E}_{N})$ has
the relative Haagerup property. Let $\tau$ be a faithful normal
tracial state on $N$ that we extend to a state $\varphi:=\tau\circ\mathbb{E}_{N}$
on $M$. Theorem \ref{Prop=HAPtoCorner} implies that for every $k\in\mathbb{N}$
the triple $(p_{k}(M\rtimes_{\sigma^{\varphi}}\mathbb{R})p_{k},p_{k}Np_{k},p_{k}\tau_{\rtimes,1}p_{k})$
has the $\text{(rHAP)}$. With Proposition \ref{Prop=CornerToCorner}
we get that for every $k\in\mathbb{N}$ the triple $(q_{k}(M\rtimes_{\sigma^{\varphi}}\mathbb{R})q_{k},q_{k}Nq_{k},q_{k}\tau_{\rtimes,1}q_{k})$
has the $\text{(rHAP)}$. The isomorphism $\rho$
restricts to an isomorphism $q_{k}(M\rtimes_{\sigma^{\varphi}}\mathbb{R})q_{k}\cong p_{\psi,k}(M\rtimes_{\sigma^{\psi}}\mathbb{R})p_{\psi,k}$
which maps $q_{k}Nq_{k}$ onto $p_{\psi,k}Np_{\psi,k}$ and for which
$(q_{k}\tau_{\rtimes,1}q_{k})\circ\rho=p_{\psi,k}\tau_{\rtimes,2}p_{\psi,k}$.
Combining this with   Theorem \ref{Prop=HAPtoCorner}
implies that $(M,N,\mathbb{F}_{N})$ has the  $\text{(rHAP)}$.
\end{proof}


\subsection{Unitality and state-preservation of the approximating maps}

The following theorem states that for  triples $(M,N,\varphi)$ with $N$ finite  
the approximating maps may be assumed to be unital and state-preserving.
The proof combines the passage to suitable crossed products and corners
of crossed products from Section \ref{Preparation} with the case
considered in Subsection \ref{SpecialCase}.

\begin{thm} \label{Unitality+State-Preservation}
Let $ N\subseteq M$
be a unital inclusion of von Neumann algebras which admits a faithful
normal conditional expectation $\mathbb{E}_{N}$. Assume that $N$
is finite. Let $\tau \in N_{\ast}$ be a faithful normal (possibly non-tracial) state
that we extend to a state $\varphi:=\tau \circ \mathbb{E}_{N}$ on $M$
and assume that the triple $(M,N,\varphi)$ has property $\text{(rHAP)}$.
Then property $\text{(rHAP)}$ may be witnessed by a net of unital
and $\varphi$-preserving approximating maps, i.e. we may assume (1$^{\prime\prime}$)
and (4$^{\prime}$).
\end{thm}

\begin{proof}
First assume that $\tau$ is tracial.
Since the triple $(M,N,\varphi)$ has property $\text{(rHAP)}$
we get with Theorem \ref{Thm=NormalToCross} and Proposition \ref{Prop=ToCorner}
that for all $j\in\mathbb{N}$, $k\in\mathbb{N}$ the triple
$(p_{k}(M\rtimes_{\sigma^{\varphi}}\mathbb{R})p_{k},p_{k}Np_{k},p_{k}\widehat{\varphi}_{j}p_{k})$
has property $\text{(rHAP)}$ as well and that it may be witnessed
by a net of contractive approximating maps. As we have seen before,
for every $k\in\mathbb{N}$ the element $h_{k}^{1/2}\lambda(f_{j})\in p_{k}(M\rtimes_{\sigma^{\varphi}}\mathbb{R})p_{k}$
is positive and boundedly invertible in $p_{k}(M\rtimes_{\sigma^{\varphi}}\mathbb{R})p_{k}$.
Further, by \eqref{Eqn=HkRadonNykodym} and \cite[Theorem VIII.2.11]{Takesaki2}
the equality
\begin{eqnarray}
\nonumber
\sigma_{t}^{p_{k}\widehat{\varphi_{j}}p_{k}}(x)=(h_{k}^{1/2}\lambda(f_{j}))^{it}x(h_{k}^{1/2}\lambda(f_{j}))^{-it}
\end{eqnarray}
holds for all $x\in p_{k}(M\rtimes_{\sigma^{\varphi}}\mathbb{R})p_{k}$,
$t\in\mathbb{R}$. Theorem \ref{BannonFangApproach} then implies
that property $\text{(rHAP)}$ of $(p_{k}(M\rtimes_{\sigma^{\varphi}}\mathbb{R})p_{k},p_{k}Np_{k},p_{k}\widehat{\varphi}_{j}p_{k})$
may for every $j,k \in\mathbb{N}$ be witnessed
by a net of unital $(p_{k}\widehat{\varphi}_{j}p_{k})$-preserving
maps. By applying the converse directions of Proposition \ref{Prop=ToCorner}
and Theorem \ref{Thm=NormalToCross} we deduce the claimed statement.

Now we show that we may replace $\tau$ by any non-tracial faithful state   in $N_\ast$.  Let still $\tau \in N_\ast$ be a faithful tracial state.
 Let $(\Phi_i)_{i \in I}$ be approximating maps witnessing the $\text{(rHAP)}$ for $(M,N,\tau \circ \mathbb{E}_N)$ which by the previous paragraph  may be taken unital and $\tau \circ \mathbb{E}_N$-preserving. The proof of Theorem  \ref{StateIndependence}, exploiting Lemmas   \ref{Lem=StatePreserving} and \ref{Lem=StandardForm} shows that $(\Phi_i)_{i \in I}$ also witness the $\text{(rHAP)}$  for  $(M,N, \varphi \circ \mathbb{E}_N)$ for any faithful state $\varphi \in N_\ast$. Further Lemma \ref{Lem=StatePreserving} shows that  $\Phi_i$ is $\varphi \circ  \mathbb{E}_N$-preserving and we are done.
\end{proof}


\subsection{Equivalence of $\text{(rHAP)}$ and $\text{(rHAP)}^{-}$}

In \cite{BannonFang} among other things Bannon and Fang prove that
for triples $(M,N,\tau)$ of finite von Neumann algebras with a tracial
state $\tau\in M_{\ast}$ the subtraciality condition in Popa's notion
of the relative Haagerup property is redundant. It is easy to check
that their proof translates into our setting, which leads to the following
variation of \cite[Theorem 2.2]{BannonFang}.

\begin{thm}[Bannon-Fang] \label{BaFaTheorem2.2}
Let $M$ be a finite von Neumann algebra equipped with a faithful normal tracial
state $\tau\in M_{\ast}$ and let $ N\subseteq M$ be a unital
inclusion of von Neumann algebras. If the triple $(M,N,\tau)$ has
property $\text{(rHAP)}^{-}$, then it has property $\text{(rHAP)}$.
Further, property $\text{(rHAP)}$ may be witnessed by unital and
trace-preserving approximating maps.
\end{thm}

In combination with Theorem \ref{Unitality+State-Preservation} the
following theorem provides a generalisation of Theorem \ref{BaFaTheorem2.2}.

\begin{thm} \label{rHAP=rHAP-}
Let $ N\subseteq M$
be a unital inclusion of von Neumann algebras  which admits a faithful
normal conditional expectation $\mathbb{E}_{N}$. Assume that $N$
is finite.  Let $\tau \in N_{\ast}$ be a faithful normal state
that we extend to a state $\varphi:=\tau \circ \mathbb{E}_{N}$ on $M$.
Then the triple $(M,N,\varphi)$ has property $\text{(rHAP)}$ if
and only if it has property $\text{(rHAP)}^{-}$.
\end{thm}

\begin{proof}
 By Theorem  \ref{StateIndependence} we may without loss of generality assume that $\tau$ is tracial on $N$.
It is clear that property $\text{(rHAP)}$ implies
property $\text{(rHAP)}^{-}$. Conversely, if the triple $(M,N,\varphi)$
has property $\text{(rHAP)}^{-}$, then we deduce from Theorem \ref{Prop=HAPtoCorner}
that for every $k\in\mathbb{N}$ the triple $(p_{k}(M\rtimes_{\sigma^{\varphi}}\mathbb{R})p_{k},p_{k}Np_{k},p_{k}\tau_{\rtimes,1}p_{k})$
has property $\text{(rHAP)}^{-}$ as well. Recall that $p_{k}(M\rtimes_{\sigma^{\varphi}}\mathbb{R})p_{k}$
is finite since $p_{k}\tau_{\rtimes}p_{k}$ is a faithful normal tracial
state. We can hence apply Theorem \ref{BaFaTheorem2.2} to deduce that $(p_{k}(M\rtimes_{\sigma^{\varphi}}\mathbb{R})p_{k},p_{k}Np_{k},p_{k}\tau_{\rtimes,1}p_{k})$ has $\text{(rHAP)}$ for every $k \in \mathbb{N}$. In combination
with Theorem \ref{Prop=HAPtoCorner} this implies that the triple
$(M,N,\varphi)$ has property $\text{(rHAP)}$.
\end{proof}

\vspace{1mm}

We finish this subsection with an easy lemma which will be needed later on. It could be formulated in a greater generality, but this is the form we will use in Section \ref{amfreeprod}.

	\begin{lem} \label{lem:condexp}
		Let $ N\subseteq M_1 \subset M$ be a unital inclusion
		of von Neumann algebras with $N$  finite. Assume that we have faithful normal conditional expectations $\mathbb{E}_1:M_1 \to N$ and $\mathbb{F}_1:M \to M_1$ and a faithful tracial state $\tau \in N_*$. Set $\varphi = \tau \circ \mathbb{E}_1 \circ  \mathbb{F}_1$. Then if the triple $(M,N,\varphi)$ has property $\text{(rHAP)}$	then the triple $(M_1,N,\varphi|_{M_1})$ also has  	property $\text{(rHAP)}$.	
		\end{lem}

\begin{proof}
Suppose that $(\Phi_i)_{i \in I}$ is a net of approximations (unital, $\varphi$-preserving maps on $M$) satisfying the conditions in the property 	$\text{(rHAP)}$ for the triple  $(M,N,\varphi)$. For each $i \in I$ define $\Psi_i:= \mathbb{F}_1 \circ \Phi_i|_{M_1}$. Our conditions guarantee that $\mathbb{F}_1$ is $\varphi$-preserving, so $\Psi_i$ is a normal, ucp, $N$-bimodular, $\varphi|_{M_1}$ preserving map on $M_1$. Due to the last theorem, we need only to check that $(\Psi_i)_{i \in I}$ satisfy the conditions in the property $\text{(rHAP)}^{-}$ (for the triple $(M_1,N,\varphi|_{M_1})$). Condition (iii) holds as for $x \in M_1$ we have $\Psi_i(x) - x = \mathbb{F}_1 (\Phi_i(x) - x)$ and $\mathbb{F}_1^{(2)}$ is the orthogonal projection from $L^2(M, \varphi)$ onto $L^2(M_1, \varphi|_{M_1})$.

To verify the last condition we assume first that $\Phi_i$ is of the form $a (\mathbb{E}_1 \circ  \mathbb{F}_1)(b\cdot)$ for some $a, b \in M$. But then for $x \in M_1$ we have
\begin{align*} \Psi_i(x) = \mathbb{F}_1 (a (\mathbb{E}_1 \circ  \mathbb{F}_1)(bx))) =
\mathbb{F}_1 (a) (\mathbb{E}_1 \circ  \mathbb{F}_1)(bx))) = \mathbb{F}_1 (a) \mathbb{E}_1 (\mathbb{F}_1(b)x),  \end{align*}
so we get that $\Psi_i^{(2)} \in \mathcal{K}_{00}(M_1,N,\varphi|_{M_1})$. Taking linear combinations and approximation ends the proof.
\end{proof}

\section{First examples}\label{Sect=Examples}

In this section we first put our definitions and main results in concrete context, discussing examples of the Haagerup (and non-Haagerup) inclusions arising in the framework of Cartan subalgebras, as studied in \cite{Jolissaint}, \cite{Ueda} and \cite{ClaireEquiv}, and then present the case of the big algebra being just $\mathcal{B}(\mathcal{H})$. The examples related to the latter situation show that the relative Haagerup property is not implied by coamenability as defined in \cite{PopaCorr}.

\subsection{Examples from equivalence relations and groupoids} \label{Subsec:Examples}

 In this subsection we will discuss examples of inclusions of von Neumann algebras which satisfy the relative Haagerup property and have already appeared in the literature. As mentioned in the introduction, the notion of the Haagerup property  regarding the von Neumann inclusions beyond the finite context first appeared  in the study of von Neumann algebras associated with groupoids/equivalence relations.

The first result here is due to \cite{Jolissaint}, still in the finite context. Note that Jolissaint uses the definition of the Haagerup inclusion $N \subset M$ due to Popa in \cite{Popa06}, namely the one using the larger ideal of `generalised compacts' than the one employed in this paper, but also note that due to \cite[Proposition 2.2]{Popa06} both notions coincide if $N' \cap M \subset N$, so for example if $N$ is a maximal abelian subalgebra in $M$, which is the case of interest for the result below.
	
	\begin{thm}\cite[Theorem 2.1]{Jolissaint}
		Let $\mathcal{R}$ be a measure preserving standard equivalence relation on a set $X$ (with the measure $\nu$ on $\mathcal{R}$ induced by the invariant probability measure $\mu$ on $X$). Then the following are equivalent:
		\begin{itemize}
			\item[(i)] $\mathcal{R}$ has the Haagerup property, i.e.\ it admits a sequence of positive-definite functions $(\varphi_n:\mathcal{R} \to \mathbb{C})_{n \in \mathbb{N}}$ which are bounded by $1$ on the diagonal, converge to $1$ $\nu$-almost everywhere and satisfy the `vanishing property': for every $n \in \mathbb{N}$ and $\epsilon>0$ there is \[\nu(\{(x,y) \in \mathcal{R}: |\varphi_n(x,y)| > \epsilon\}) < \infty;\]
			\item[(ii)] the von Neumann inclusion (of finite von Neumann algebras)
			\[ L^\infty(X, \mu)\subset \mathcal{L}(\mathcal{R})\]
			has the relative Haagerup property.
		\end{itemize}	
	\end{thm}
	
	The definition beyond the finite case has first been considered in \cite{Ueda}; a more detailed study has been conducted by Anantharaman-Delaroche in \cite{ClaireEquiv}. Note that both these papers use the notion of the relative Haagerup property for arbitrary (expected) von Neumann inclusions identical to the one studied here. We will now describe the setup.
	
	Let $\mathcal{G}$ be a measured groupoid with countable fibers, equipped with a quasi-invariant probability measure $\mu$ on the unit space $\mathcal{G}^{(0)}$ (note that a measure preserving standard equivalence relation as considered above is one source of such examples). Again $\mu$ induces a measure $\nu$ on $\mathcal{G}$; we further obtain a (not necessarily finite) von Neumann algebra $  \mathcal{L}(\mathcal{G}) \subset \mathcal{B}(L^2(\mathcal{G}, \nu))$. The following result holds.
	
	\begin{thm}\cite[Theorem 1]{ClaireEquiv}
		Let $\mathcal{G}$ be a measured groupoid with countable fibers, as above. 	Then the following conditions are equivalent:
		\begin{itemize}
			\item[(i)] $\mathcal{G}$ has the Haagerup property, i.e.\ it admits a sequence of positive-definite functions $(F_n:\mathcal{G} \to \mathbb{C})_{n \in \mathbb{N}}$ which are equal to $1$ on $\mathcal{G}^{(0)}$, converge to $1$ $\nu$-almost everywhere and satisfy the `vanishing property': for every $n \in \mathbb{N}$ and $\epsilon>0$ there is \[\nu(\{g \in \mathcal{G}: |\varphi_n(g)| > \epsilon\}) < \infty;\]
			\item[(ii)] the von Neumann inclusion
			\[ L^\infty(\mathcal{G}^{(0)}, \mu)\subset \mathcal{L}(\mathcal{G})\]
			has the relative Haagerup property.
		\end{itemize}		
	\end{thm}
	
Ueda shows in \cite[Lemma 5]{Ueda} (and then Anantharaman-Delaroche reproves it in \cite[Theorem 3]{ClaireEquiv}) that a property of a groupoid as above called \emph{treeability} implies the Haagerup property.  \cite[Theorem 5]{ClaireEquiv} also shows that for ergodic measured groupoid with countable fibers the Haagerup property is incompatible with Property (T); we are however not aware of explicit examples of such Property (T) groupoids leading to von Neumann algebras which are not finite, and a general intuition regarding Property (T) objects says that these should naturally lead to finite von Neumann algebras (for example discrete property (T) quantum groups are necessarily unimodular, see \cite{Fima}).

 \subsection{Examples and counterexamples with $M =\mathcal{B}(\mathcal{H})$}

We end this section with the example where $M = \mathcal{B}(\mathcal{H})$ and study which triples $(\mathcal{B}(\mathcal{H}), N, \mathbb{E}_N)$ have (rHAP). Since the conditional expectation $\mathbb{E}_N$ is assumed to be normal it follows by a result of Tomiyama from \cite{Tomiyama} that $N$ must be a direct sum of type I factors, so $N \simeq \oplus_{i \in I} \mathcal{B}(\mathcal{K}_i )$ for some index set $I$. Note that each $\mathcal{B}(\mathcal{K}_i )$ may occur in $\mathcal{B}(\mathcal{H})$ with a certain multiplicity $m_i \in \mathbb{N} \cup \{ \infty \}$. In general, we have that $N$ is spatially isomorphic to $\oplus_{i \in I} \mathcal{B}(\mathcal{K}_i ) \otimes \mathbb{C} 1_{m_i}$ where $1_{m_i}$ is the identity acting on a Hilbert space of dimension $m_i$. For simplicity in the examples below we assume that all multiplicities $m_i$ equal 1 and ignore the spatial isomorphism. In that case the normal conditional expectation of $\mathcal{B}(\mathcal{H})$ onto $\oplus_{i \in I} \mathcal{B}(\mathcal{K}_i )$ is unique and determined by $\mathbb{E}_N(x) = \sum_{i \in I} p_i x p_i$ where $p_i$ is the projection onto $\mathcal{K}_i$. Therefore, in this case we can speak not only of the Haagerup property of the inclusion $N \subseteq \mathcal{B}(\mathcal{H})$, but also about maps being compact and of finite index relative to this inclusion.

\begin{thm}\label{Thm=BHcase}
Assume that $\mathcal{H}$ is a separable Hilbert space, that $\mathcal{H} = \bigoplus_{i \in I} \mathcal{K}_i$, where $I$ is an index set and that the dimension of $\mathcal{K}_i$ does not depend on $i\in I$. Put $N = \bigoplus_{i \in I}  \mathcal{B}(\mathcal{K}_i) \subset \mathcal{B}(\mathcal{H})$. Then the triple $(\mathcal{B}(\mathcal{H}), N, \mathbb{E}_N)$  has the property (rHAP).
\end{thm}

\begin{proof}
We may assume that $\mathcal{K}_{i} = \mathcal{K}$ for a single (separable) Hilbert space $\mathcal{K}$. 
The inclusion $N \subseteq \mathcal{B}(\mathcal{H})$ is then 
isomorphic to the inclusion  $\ell^\infty(I) \otimes \mathcal{B}(  \mathcal{K}  ) \subseteq \mathcal{B}(\ell^2( I )) \otimes   \mathcal{B}( \mathcal{K}  )$. In the case where $I$ is finite $\ell^\infty(I) \subseteq \mathcal{B}(\ell^2(I))$ is a finite dimensional inclusion which clearly has (rHAP). In the case where $I$ is infinite we may assume that $I = \mathbb{Z}$ and the inclusion $\ell^\infty(\mathbb{Z} ) \subseteq \mathcal{B}( \ell^2( \mathbb{Z} ) )$ has the (rHAP) with   approximating maps given by the (Fej\'er-)Herz-Schur multipliers $T_n$ with
\[
T_n(  (x_{i,j})_{i,j \in \mathbb{Z}}  ) = (   W(i - j)    x_{i,j})_{i,j \in \mathbb{Z}}, \qquad W(k) := \max(1 - \frac{\vert k \vert}{n} , 0  ).
\]
Since $W = \frac{1}{n} (\chi_{[0,n]})^\ast \ast \chi_{[0,n]}$ is positive definite and converges to the identity pointwise it follows that $T_n$ is completely positive and $T_n^{(2)}$ converges to the identity strongly. Further $T_n^{(2)}$ is finite rank relative to $\ell^\infty( \mathbb{Z} )$, so certainly compact.   In both cases ($I$ being finite or infinite), we tensor the approximating maps with ${\rm Id}_{ \mathcal{B}(  \mathcal{K}  ) }$ and find that  $\ell^\infty( I  ) \otimes \mathcal{B}(  \mathcal{K}  ) \subseteq \mathcal{B}(\ell^2(I)) \otimes \mathcal{B}(  \mathcal{K}  )$ has (rHAP).

\end{proof}

With a bit more work Theorem \ref{Thm=BHcase} could be proved in larger generality by relaxing  the assumption that the multiplicities are trivial and that the dimension is constant (as opposed to say for example  uniformly bounded). However, we cannot admit just any subalgebra $N$ as the following counterexample shows.

\begin{thm}\label{Thm=NonRHap}
Let $\mathcal{H} = \mathcal{K}_1 \oplus \mathcal{K}_2$, where $\mathcal{K}_1, \mathcal{K}_2$ are Hilbert spaces such that  $\dim( \mathcal{K}_1 ) < \infty$ and $\dim( \mathcal{K}_2 ) = \infty$. Set $N= \mathcal{B}(\mathcal{K}_1) \oplus \mathcal{B}(\mathcal{K}_2)$.  Then  the triple $(\mathcal{B}(\mathcal{H}), N, \mathbb{E}_N)$ does not have the property (rHAP).
\end{thm}

\begin{proof}
Let $p$ be the projection of $\mathcal{H}$ onto $\mathcal{K}_1$. Let $\Phi: \mathcal{B}(\mathcal{H}) \rightarrow \mathcal{B}(\mathcal{H})$ be a normal linear map. The proof is based on two claims.

\vspace{0.3cm}

\noindent {\bf Claim 1:} If $\Phi$ is an $N$-$N$ bimodule map then  $\mathcal{B}(\mathcal{H})p$ is an invariant subspace. Moreover, the restriction of $\Phi$ to $\mathcal{B}(\mathcal{H})p$ lies in the linear span of the two maps $ x p \mapsto p xp$ and $xp \mapsto (1-p) x p$.

\vspace{0.3cm}

\noindent {\it Proof of Claim 1.} Note that $p$ is contained in $N$ from which the first statement follows. For the second part let $E^{i}_{k,l}$ be matrix units with respect to some basis of $\mathcal{K}_i$. Then for $x \in \mathcal{B}(\mathcal{H})$ we have  $\Phi(E^{i}_{k,k} xE^{i}_{l,l}) = E^{i}_{k,k} \Phi( x ) E^{i}_{k,k}$ so that $E^{i}_{k,k} \mathcal{B}(\mathcal{H})E^{i}_{k,k}$ is an eigenspace of $\Phi$ (i.e. $\Phi$ is a Herz-Schur multiplier). Moreover $\Phi(E^{i}_{k',k'} xE^{i}_{l',l'}) =  E^{i}_{k',k}  \Phi(E^{i}_{k,k'} x E^{i}_{l',l} ) E^{i}_{l,l'}$ so that the eigenvalues of these spaces only depend on $i$. This in particular implies the claim.

\vspace{0.3cm}

\noindent {\bf Claim 2:} If $\Phi$ is compact relative to the inclusion $N \subseteq \mathcal{B}(\mathcal{H})$ then  $\mathcal{B}(\mathcal{H})p$ is an invariant subspace. Moreover, the restriction of $\Phi$ to $\mathcal{B}(\mathcal{H})p$ is compact (in the non-relative sense).

\vspace{0.3cm}

\noindent {\it Proof of Claim 2.} By approximation it suffices to prove Claim 2 with `compact' replaced by `finite rank'. So assume that $\Phi = a \mathbb{E}_N b$ with $a,b \in \mathcal{B}(\mathcal{H})$. Note that $p \in N \cap N'$ and therefore $a \mathbb{E}_N( b x p) = a p \mathbb{E}_N( b x)p = a  \mathbb{E}_N(p  b x p )$. The first of these equalities shows that $\mathcal{B}(\mathcal{H})p$ is   invariant. Further $x \mapsto (p  x  p )$ is finite rank as $p$ projects onto a finite dimensional space. This proves the claim.

\vspace{0.3cm}

\noindent {\it Remainder of the proof.} Suppose that $\Phi$ is both $N$-$N$ bimodular and compact relative to $N$. By Claim 1 we know that there are scalars $\lambda_1, \lambda_2 \in \mathbb{C}$ such that $\Phi(  x p) = \lambda_1 pxp + \lambda_2 (1-p) x p$. If $\lambda_2 \not = 0$ then the associated $L^2$-map is not compact (in the non-relative sense) since $(1-p)$ projects onto an infinite dimensional Hilbert space. This contradicts Claim 2 because the restriction of $\Phi$ to  $\mathcal{B}(\mathcal{H})p$  is compact.  We conclude that $\lambda_2 =0$ for any normal map $\Phi: M \rightarrow M$ that is $N$-$N$-bimodular and compact relative to $N$. But then we can never find a net of such maps that approximates the identity map on $\mathcal{B}(\mathcal{H})$ in the point-strong topology. Hence the inclusion $N \subseteq \mathcal{B}(\mathcal{H})$ fails to have (rHAP).
\end{proof}

\begin{rmk}
 Recall that a unital inclusion of von Neumann algebras
 $N \subseteq M$ is said to be \emph{co-amenable} if there exists a (not necessarily normal) conditional expectation from  $N'$ onto $M'$, where the commutants are taken with respect to any Hilbert space realization of $M$.  Theorem  \ref{Thm=NonRHap} shows -- surprisingly -- that a co-amenable inclusion in general need not have (rHAP).

 Note that this also means that a naive extension of the definition of relative Haagerup property in terms of correspondences, modelled on the notion of \emph{strictly mixing bimodules} \cite[Theorem 9]{RuiTakaReiji} valid for the non-relative Haagerup property, cannot be equivalent to the definition studied in our paper. Indeed, the last fact, together with the examples above, would contradict \cite[Theorem 2.4]{BannonMarrakchiOzawa}.

\end{rmk}

\section{Property \text{(RHAP)} for finite-dimensional subalgebras}\label{Sect=FD}

In this section we consider  the case of finite-dimensional subalgebras and show equivalence of the relative Haagerup property and the non-relative Haagerup
property. For this, we fix a unital inclusion $N\subseteq M$
of von Neumann algebras and assume that it admits a faithful normal
conditional expectation $\mathbb{E}_{N}$. Assume that $N$ is finite-dimensional
and let $\tau\in N_{\ast}$ be a faithful normal tracial
state on $N$ that we extend to a state $\varphi:=\tau\circ\mathbb{E}_{N}$
on $M$. We will prove that the triple $(M,N,\varphi)$ has property
$\text{(rHAP)}$ if and only if $(M,\mathbb{C},\varphi)$ does. Recall
that by Theorem \ref{StateIndependence} the Haagerup
property of $(M,N,\varphi)$ does not depend on the choice of the
 state $\tau$.

Denote by $z_{1},...,z_{n}\in\mathcal{Z}(N)$ the minimal central projections
of $N$. There exist natural numbers $n_{1},...,n_{k}\in\mathbb{N}$
such that $z_{k}N\cong M_{n_{k}}(\mathbb{C})$ for $k=1,...,n$. Let
$(f_{i}^{k})_{1\leq i\leq n_{k}}$ be an orthonormal basis of $\mathbb{C}^{n_{k}}$,
write $E_{i,j}^{k}$, $1\leq i,j\leq n_{k}$ for the matrix units
with respect to this basis and set $E_{i}^{k}:=E_{i,i}^{k}$ for the
diagonal projections. We have that $E_{i,j}^{k}f_{l}^{k}=\delta_{j,l}f_{i}^{k}$
for all $k\in\mathbb{N}$, $1\leq i,j,l\leq n_{k}$ and $\sum_{k=1}^{n}\sum_{i=1}^{n_{k}}E_{i}^{k}=1$.
Set $d:=\sum_{k=1}^{n}n_{k}$, choose an orthonormal basis $(f_{k,i})_{1\leq k\leq n\text{, }1\leq i\leq n_{k}}$
of $\mathbb{C}^{d}$ with corresponding matrix units $e_{(k,i),(l,j)}\in M_{d}(\mathbb{C})$
where $1\leq k,l\leq n$, $1\leq i\leq n_{k}$, $1\leq j\leq n_{l}$
and define
\begin{equation}\label{Eqn=PDefinition}
p:=\sum_{k=1}^{n}E_{1}^{k}.
\end{equation}

For a general linear map $\Phi\text{: }pMp\rightarrow pMp$ we may define
a  linear map $\widetilde{\Phi}\text{: }M\rightarrow M$
by
\begin{equation}\label{Eqn=PhiTilde}
\widetilde{\Phi}(E_{i}^{k}xE_{j}^{l}):=E_{i,1}^{k}\Phi(E_{1,i}^{k}xE_{j,1}^{l})E_{1,j}^{l}
\end{equation}
 for all $1\leq k,l\leq n$, $1\leq i\leq n_{k}$, $1\leq j\leq n_{l}$
and $x\in M$.

Let us study the properties of $\widetilde{\Phi}$.

\begin{lem} \label{Lem=Stinespring}
Let $\Phi:pMp\rightarrow pMp$
be a linear map. Define
\[
U:=\sum_{k=1}^{n}\sum_{i=1}^{n_{k}}e_{(1,1),(k,i)}\otimes E_{i,1}^{k}\in M_{d}(\mathbb{C})\otimes M,\qquad V:=\sum_{k=1}^{n}\sum_{i=1}^{n_{k}}f_{k,i}\otimes E_{1,i}^{k}\in\mathbb{C}^{d}\otimes M.
\]
 Then,
\[
\widetilde{\Phi}(x)=V^{\ast}(\text{id}_{\mathcal{B}(L^{2}(N,\tau))}\otimes\Phi)\left(U^{\ast}(1\otimes x)U\right)V\text{.}
\]
\end{lem}

\begin{proof}
We have for $x\in z_{k}Mz_{l}$ with $1\leq k,l\leq n$ that
\[
U^{\ast}(1\otimes x)U=\sum_{i=1}^{n_{k}}\sum_{j=1}^{n_{l}}e_{(k,i),(l,j)}\otimes E_{1,i}^{k}xE_{j,1}^{l}
\]
so that
\[
V^{\ast}(\text{id}_{\mathcal{B}(L^{2}(N,\tau))}\otimes\Phi)\left(U^{\ast}(1\otimes x)U\right)V=\sum_{i=1}^{n_{k}}\sum_{j=1}^{n_{l}}E_{i,1}^{k}\Phi(E_{1,i}^{k}xE_{j,1}^{l})E_{1,j}^{l}.
\]
By definition this expression coincides with $\widetilde{\Phi}(x)$.
The claim follows.
\end{proof}

\begin{lem} \label{Lem=NormalUcp}
If $\Phi\text{: }pMp\rightarrow pMp$ is a unital normal completely positive map, then $\widetilde{\Phi}$
is contractive, normal and completely positive.
\end{lem}

\begin{proof}
The normality and the complete positivity follow from
Lemma \ref{Lem=Stinespring}. We further have
\begin{eqnarray}
\nonumber
\Vert \widetilde{\Phi}\Vert &=& \widetilde{\Phi}(1)=\widetilde{\Phi}\left(\sum_{k=1}^{n}\sum_{i=1}^{n_{k}}E_{i}^{k}\right)=\sum_{k=1}^{n}\sum_{i=1}^{n_{k}}E_{i,1}^{k}\Phi(E_{1}^{k})E_{1,i}^{k} \\
\nonumber
&\leq& \sum_{k=1}^{n}\sum_{i=1}^{n_{k}}E_{i,1}^{k}E_{1,i}^{k}=\sum_{k=1}^{n}\sum_{i=1}^{n_{k}}E_{i}^{k}=1\text{,}
\end{eqnarray}
i.e. $\widetilde{\Phi}$ is contractive.
\end{proof}

\begin{lem} \label{Lem=NNBimodule}
Let $\Phi\text{: }pMp\rightarrow pMp$
be a linear map. Then $\widetilde{\Phi}$ is an $N$-$N$-bimodule
map.
\end{lem}

\begin{proof}
Let $x\in M$. For $1\leq l,k,m\leq n$ and $1\leq r,s\leq n_{l}$,
$1\leq i\leq n_{k}$, $1\leq j\leq n_{m}$ we have
\begin{eqnarray}
\nonumber
E_{r,s}^{l}\widetilde{\Phi}(E_{i}^{k}xE_{j}^{m}) &=&E_{r,s}^{l}E_{i,1}^{k}\Phi(E_{1,i}^{k}xE_{j,1}^{m})E_{1,j}^{m}\\
\nonumber
&=& \delta_{s,i}\delta_{l,k}E_{r,1}^{k}\Phi(E_{1,i}^{k}xE_{j,1}^{m})E_{1,j}^{m}\\
\nonumber
&=& E_{r,1}^{k}\Phi(E_{1,r}^{k}E_{r,s}^{l}E_{i}^{k}xE_{j}^{m}E_{j,1}^{m})E_{1,j}^{m}.
\end{eqnarray}
We hence find that for $y\in E_{i}^{k}xE_{j}^{m}$, $E_{r,s}^{l}\widetilde{\Phi}(y)=\widetilde{\Phi}(E_{r,s}^{l}y)$.
The linearity of $\widetilde{\Phi}$ then implies that it is a left
$N$-module map. A similar argument applies to the right-handed case.
\end{proof}

\begin{prop} \label{Prop=Expect}
Define
the map
\[
\text{Diag: }pMp\rightarrow pMp\text{, }x\mapsto\sum_{k=1}^{n}\frac{\varphi(E_{1}^{k}xE_{1}^{k})}{\varphi(E_{1}^{k})}E_{1}^{k}\text{.}
\]
Then $\widetilde{\text{Diag}}=\mathbb{E}_{N}$.
\end{prop}

\begin{proof}
It is clear that the map $\text{Diag}$ is linear
unital normal and completely positive. Hence, by Lemma \ref{Lem=NormalUcp}
and Lemma \ref{Lem=NNBimodule}, $\widetilde{\text{Diag}}$
is contractive normal completely positive and $N$-$N$-bimodular. It is easy to check that $\widetilde{\text{Diag}}$ is even unital. In particular,
$\widetilde{\text{Diag}}$ restricts to the identity on $N$. It is
further clear that $\widetilde{\text{Diag}}$ is faithful and that
it maps $M$ onto $N$, so $\widetilde{\text{Diag}}$ is a faithful
normal conditional expectation. For $x\in M$ and $1\leq k,l\leq n$,
$1\leq i\leq n_{k}$, $1\leq j\leq n_{l}$ we have
\begin{eqnarray}
\nonumber
\varphi\circ\widetilde{\text{Diag}}(E_{i}^{k}xE_{j}^{l}) &=& \varphi\left(E_{i,1}^{k}\text{Diag}(E_{1,i}^{k}xE_{j,1}^{l})E_{1,j}^{l}\right)\\
\nonumber
&=& \sum_{m=1}^{n}\frac{\varphi(E_{1}^{m}E_{1,i}^{k}xE_{j,1}^{l}E_{1}^{m})}{\varphi(E_{1}^{m})}\varphi\left(E_{i,1}^{k}E_{1}^{m}E_{1,j}^{l}\right)\\
\nonumber
&=& \frac{\varphi(E_{1}^{l}E_{1,i}^{k}xE_{j,1}^{l}E_{1}^{l})}{\varphi(E_{1}^{l})}\varphi(E_{i,1}^{k}E_{1}^{l}E_{1,j}^{l})\\
\nonumber
&=& \delta_{k,l}\frac{\varphi(E_{1,i}^{l}xE_{j,1}^{l})}{\varphi(E_{1}^{l})}\varphi(E_{i,1}^{l}E_{1,j}^{l})\\
\nonumber
&=& \delta_{k,l}\frac{\varphi(E_{1,i}^{l}xE_{j,1}^{l})}{\tau(E_{1}^{l})}\tau(E_{i,j}^{l})\text{.}
\end{eqnarray}
But then, since $\tau$ is tracial,
\begin{eqnarray}
\nonumber
\varphi\circ\widetilde{\text{Diag}}(E_{i}^{k}xE_{j}^{l}) &=& \delta_{i,j}\delta_{k,l}\varphi(E_{1,i}^{l}xE_{i,1}^{l})=\delta_{i,j}\delta_{k,l}\tau(\mathbb{E}_{N}(E_{1,i}^{l}xE_{i,1}^{l})) \\
\nonumber
&=& \delta_{i,j}\delta_{k,l}\tau(E_{1,i}^{l}\mathbb{E}_{N}(x)E_{i,1}^{l})=\tau(E_{i}^{k}\mathbb{E}_{N}(x)E_{j}^{l})\\
\nonumber
&=& \varphi(E_{i}^{k}xE_{j}^{l})\text{,}
\end{eqnarray}
i.e. $\widetilde{\text{Diag}}$ is $\varphi$-preserving. Since $\mathbb{E}_{N}$
is the unique faithful normal $\varphi$-preserving conditional expectation
onto $N$, we get that $\widetilde{\text{Diag}}=\mathbb{E}_{N}$.
\end{proof}

\begin{lem} \label{Lem=FiniteTilde}
Let $\Phi\text{: }pMp\rightarrow pMp$
be a normal completely positive map with $\varphi\circ\Phi\leq\varphi$
and assume that the $L^{2}$-implementation $\Phi^{(2)}$ of $\Phi$
with respect to $\varphi|_{pMp}$ is a compact operator. Then $\widetilde{\Phi}$
satisfies $\varphi\circ\widetilde{\Phi}\leq\varphi$ and $(\widetilde{\Phi})^{(2)}\in\mathcal{K}(M,N,\varphi)$.
\end{lem}

\begin{proof}
For $1\leq k,l\leq n$, $1\leq i\leq n_{k}$, $1\leq j\leq n_{l}$
and $x\in M$ positive we have by the traciality of $\tau$,
\begin{eqnarray}
\nonumber
\varphi\circ\widetilde{\Phi}(E_{i}^{k}xE_{j}^{l}) &=& \varphi\left(E_{i,1}^{k}\Phi(E_{1,i}^{k}xE_{j,1}^{l})E_{1,j}^{l}\right)\\
\nonumber
&=& \tau\left(E_{i,1}^{k}\mathbb{E}_{N}\left(\Phi(E_{1,i}^{k}xE_{j,1}^{l})\right)E_{1,j}^{l}\right)\\
\nonumber
&=& \delta_{i,j}\delta_{k,l}\tau\left(E_{1}^{k}\mathbb{E}_{N}(\Phi(E_{1,i}^{k}xE_{i,1}^{k}))\right)\text{,}
\end{eqnarray}
so in particular $\varphi\circ\widetilde{\Phi}(E_{i}^{k}xE_{j}^{l})\geq0$.
We get (as $N$ is contained in the centralizer $M^\varphi$)
\begin{eqnarray}
\nonumber
\varphi\circ\widetilde{\Phi}(E_{i}^{k}xE_{j}^{l})&=&\delta_{i,j}\delta_{k,l}\tau\left(E_{1}^{k}\mathbb{E}_{N}(\Phi(E_{1,i}^{k}xE_{i,1}^{k}))\right) \\
\nonumber
&\leq& \delta_{i,j}\delta_{k,l}\tau\left(\mathbb{E}_{N}(\Phi(E_{1,i}^{k}xE_{i,1}^{k}))\right) \\
\nonumber
&\leq& \delta_{i,j}\delta_{k,l}\varphi(E_{1,i}^{k}xE_{i,1}^{k})\\
\nonumber
&=& \varphi(E_{i}^{k}xE_{j}^{l})\text{.}
\end{eqnarray}
This implies that $\widetilde{\Phi}$ indeed satisfies $\varphi\circ\widetilde{\Phi}\leq\varphi$.
In particular, the $L^{2}$-implementation of $\widetilde{\Phi}$
with respect to $\varphi$ exists.

It remains to show that $(\widetilde{\Phi})^{(2)}\in\mathcal{K}(M,N,\varphi)$.
For this, let $\Psi\text{: }pMp\rightarrow pMp$ be a map with $\Psi^{(2)}=ae_{\mathbb{C}}b$
where $a,b\in pMp$ and $e_{\mathbb{C}}$ denotes the rank one projection $(\varphi|_{pMp}(\cdot)p)^{(2)}\in\mathcal{B}(L^{2}(pMp,\varphi|_{pMp}))$.
For $1\leq k,l\leq n$, $1\leq i\leq n_{k}$, $1\leq j\leq n_{l}$
and $x\in M$ we then have
\begin{eqnarray}
\nonumber
\widetilde{\Psi}(E_{i}^{k}xE_{j}^{l})=E_{i,1}^{k}a\varphi(bE_{1,i}^{k}xE_{j,1}^{l})E_{1,j}^{l}=E_{i,1}^{k}a\varphi(bE_{1,i}^{k}xE_{j,1}^{l})E_{1}^{l}E_{1,j}^{l}\text{.}
\end{eqnarray}
Note that by Proposition \ref{Prop=Expect},

\begin{eqnarray}
\nonumber
\mathbb{E}_{N}(bE_{1,i}^{k}xE_{j,1}^{l}) &=& \sum_{r=1}^{n}\mathbb{E}_{N}(E_{1}^{r}bE_{1,i}^{k}xE_{j,1}^{l}) = \sum_{r=1}^{n}\widetilde{\text{Diag}}(E_{1}^{r}bE_{1,i}^{k}xE_{j,1}^{l}) \\
\nonumber
&=& \sum_{r=1}^{n}E_{1}^{r}\text{Diag}(E_{1}^{r}bE_{1,i}^{k}xE_{j,1}^{l})E_{1}^{l} = \sum_{r=1}^{n}\sum_{m=1}^{n}\frac{\varphi(E_{1}^{m}E_{1}^{r}bE_{1,i}^{k}xE_{j,1}^{l}E_{1}^{m})}{\varphi(E_{1}^{m})}E_{1}^{r}E_{1}^{m}E_{1}^{l}  \\
\nonumber
&=& \frac{\varphi(E_{1}^{l}bE_{1,i}^{k}xE_{j,1}^{l})}{\varphi(E_{1}^{l})}E_{1}^{l}= \frac{\varphi(bE_{1,i}^{k}xE_{j,1}^{l})}{\varphi(E_{1}^{l})}E_{1}^{l}\text{,}
\end{eqnarray}
where in the last equality we again used that $\tau$ is tracial. Hence
\begin{equation} \label{scalar}
\widetilde{\Psi}(E_{i}^{k}xE_{j}^{l})=\varphi(E_{1}^{l})E_{i,1}^{k}a\mathbb{E}_{N}(bE_{1,i}^{k}xE_{j,1}^{l})E_{1,j}^{l}
=\varphi(E_{1}^{l})E_{i,1}^{k}a\mathbb{E}_{N}(bE_{1,i}^{k}E_{i}^{k}xE_{j}^{l})\text{.}
\end{equation}
Fix now suitable $t_0,j_0,k_0, l_0$ and $x \in M$ and compute the following expression:
\begin{align*} \left(\sum_{k=1}^{n}\sum_{l=1}^{n}\sum_{r=1}^{n}\sum_{t=1}^{n_{k}}\right.&\left.\varphi(E_{1}^{r})E_{t,1}^{k}aE_{1}^{l}e_{N}E_{1}^{r}bE_{1,t}^{k} \right)(E_{t_0}^{k_0} x E_{j_0}^{l_0} \Omega_\varphi)  \\&=
\sum_{l=1}^{n}\sum_{r=1}^{n} \varphi(E_{1}^{r})
E_{t_0,1}^{k_0} a E_1^l \mathbb{E}_N (E_1^r b E_{1,t_0}^{k_0} x E_{j_0}^{l_0} )\Omega_\varphi
\\&= \sum_{l=1}^{n} \varphi(E_{1}^{l}) E_{t_0,1}^{k_0} a E_1^l
\mathbb{E}_N (b E_{1,t_0}^{k_0} x E_{j_0,1}^{l_0} )E_{1,j_0}^{l_0}\Omega_\varphi
 \end{align*}
Now the equality \eqref{scalar} implies that the value of the conditional expectation appearing in the last formula is a scalar multiple of $E_1^{l_0}$, so the whole expression equals
\begin{align*}
\varphi(E_{1}^{l_0}) E_{t_0,1}^{k_0} a
\mathbb{E}_N (b E_{1,t_0}^{k_0} x E_{j_0}^{l_0} )\Omega_\varphi
= \widetilde{\Psi}(E_{t_0}^{k_0} x E_{j_0}^{l_0}) \Omega_\varphi.
\end{align*}
Hence we arrive at
\[
(\widetilde{\Psi})^{(2)}=\sum_{k=1}^{n}\sum_{l=1}^{n}\sum_{r=1}^{n}\sum_{t=1}^{n_{k}}\varphi(E_{1}^{r})E_{t,1}^{k}aE_{1}^{l}e_{N}E_{1}^{r}bE_{1,t}^{k} \in\mathcal{K}_{00}(M,N,\varphi)\text{.}
\]

By taking linear combinations this implies that for every map $\Psi$
with $\Psi^{(2)}\in\mathcal{K}_{00}(pMp,\mathbb{C},\varphi)$ the
$L^{2}$-implementation of $\widetilde{\Psi}$ is contained in $\mathcal{K}_{00}(M,N,\varphi)$.
Via approximation we then see that $(\widetilde{\Phi})^{(2)}\in\mathcal{K}(M,N,\varphi)$.
\end{proof}

We are now ready to prove the main theorem of this section.

\begin{thm} \label{Thm=HAPrelHAP} Let $ N\subseteq M$
be a unital inclusion of von Neumann algebras and assume that it admits
a faithful normal conditional expectation $\mathbb{E}_{N}\text{: }M\rightarrow N$.
Assume that $N$ is finite-dimensional and let $\tau\in N_{\ast}$
be a faithful  state on $N$ that we extend
to a state $\varphi:=\tau\circ\mathbb{E}_{N}$ on $M$. Then $M$ has
the Haagerup property (in the sense that the triple $(M,\mathbb{C},\varphi)$
has the relative Haagerup property) if and only if the triple $(M,N,\varphi)$
has the relative Haagerup property.
\end{thm}

\begin{proof}
By Theorem \ref{StateIndependence} we may assume without loss of generality that $\tau$ is tracial.

``$\Leftarrow$'' Assume that the triple $(M,N,\varphi)$
has the relative Haagerup property and let $(\Phi_{i})_{i\in I}$
be a net of normal completely positive maps witnessing it. Since $N$
is finite dimensional, $e_{N}$ is a finite rank projection. In particular,
$\mathcal{K}_{00}(M,N,\varphi)$ consists of finite rank operators
and hence $\mathcal{K}(M,N,\varphi)\subseteq\mathcal{K}(M,\mathbb{C},\varphi)$.
In particular, $\Phi_{i}^{(2)}\in\mathcal{K}(M,\mathbb{C},\varphi)$
for every $i\in I$. Further, $\Phi_{i}(x)\rightarrow x$ strongly
for every $x\in M$. This implies that the net $(\Phi_{i})_{i\in I}$
also witnesses the relative Haagerup property of the triple $(M,\mathbb{C},\varphi)$.

``$\Rightarrow$'' Assume that $M$ has the Haagerup property. Recall that the projection $p$ was defined in \eqref{Eqn=PDefinition}.  By
\cite[Lemma 4.1]{CS-IMRN} the triple
$(pMp,\mathbb{C},\varphi|_{pMp})$ also has the relative Haagerup
property and by Theorem \ref{Unitality+State-Preservation}
we find a net $(\Phi_{i})_{i\in I}$ of unital normal completely positive
$\varphi$-preserving maps witnessing it. By Lemma \ref{Lem=NormalUcp},
Lemma \ref{Lem=NNBimodule} and Lemma \ref{Lem=FiniteTilde}
we find that $\widetilde{\Phi}_{i}$ is a contractive normal completely
positive $N$-$N$-bimodule map with $\varphi\circ\widetilde{\Phi}\leq\varphi$
and $(\widetilde{\Phi})^{(2)}\in\mathcal{K}(M,N,\varphi)$ for every
$i\in I$. It follows directly from the prescription \eqref{Eqn=PhiTilde} that $\widetilde{\Phi}_{i}(x)\rightarrow x$
strongly for every $x\in M$.
It follows that the net $(\widetilde{\Phi}_i)_{i\in I}$ witnesses the
relative Haagerup property of the triple $(M,N,\varphi)$.
\end{proof}

\vspace{1mm}


\section{The relative Haagerup property for amalgamated free products} \label{amfreeprod}

In this section we study the notion of the relative Haagerup property
in the context of amalgamated free products of von Neumann algebras.
We will further apply our results to the class of virtually free Hecke-von
Neumann algebras.


\subsection{Preservation under amalgamated free products}

The following theorem demonstrates that in the setting of Section
\ref{MainResults} the relative Haagerup property is preserved under
taking amalgamated free products (for details on operator algebraic
amalgamated free products see \cite{Voiculescu} or also \cite{NateTaka}).
For finite inclusions of von Neumann algebras this has been proved
in \cite[Proposition 3.9]{Boca}.

\begin{thm} \label{Thm=amalgamated}
Let $ N\subseteq M_{1}$
and $N\subseteq M_{2}$ be unital embeddings of
von Neumann algebras which admit faithful normal conditional expectations
$\mathbb{E}_{1}:M_{1}\rightarrow N$, $\mathbb{E}_{2}:M_{2}\rightarrow N$
and for which $N$ is finite. Denote by $M:=(M_{1},\mathbb{E}_{1})\ast_{N}(M_{2},\mathbb{E}_{2})$
the amalgamated free product von Neumann algebra of $M_{1}$ and $M_{2}$
with respect to the expectations $\mathbb{E}_{1}$, $\mathbb{E}_{2}$
and let $\mathbb{E}_{N}$ be the corresponding conditional expectation
of $M$ onto $N$. Then $(M_{1},N,\mathbb{E}_{1})$ and $(M_{2},N,\mathbb{E}_{2})$
have the relative Haagerup property if and only if the triple $(M,N,\mathbb{E}_{N})$
has the relative Haagerup property.
\end{thm}

\begin{proof}
``$\Rightarrow$'' Assume that both $(M_{1},N, \mathbb{E}_{1})$
and $(M_{2},N,\mathbb{E}_{2})$ have the relative Haagerup property, let $\tau\in N_{\ast}$ be a faithful normal tracial state and set $\varphi_{1}:=\tau\circ\mathbb{E}_{1}$, $\varphi_{2}:=\tau\circ\mathbb{E}_{2}$.
Then the triples $(M_{1},N,\varphi_{1})$
and $(M_{2},N,\varphi_{2})$ have the relative Haagerup property.
Without loss of generality we can assume that the corresponding nets $(\Phi_{i,1})_{i\in I}$
and $(\Phi_{i,2})_{i\in I}$ witnessing the relative Haagerup property
are indexed by the same set $I$. By Theorem \ref{Unitality+State-Preservation}
we can also assume that the maps are unital with $\varphi_{1}\circ\Phi_{i,1}=\varphi_{1}$,
$\varphi_{2}\circ\Phi_{i,2}=\varphi_{2}$ for all $i\in I$, which then
implies that $\Phi_{i,1}|_{N}=\Phi_{i,2}|_{N}=\text{id}_{N}$ and
that $\mathbb{E}_{1}\circ\Phi_{i,1}=\mathbb{E}_{1}$, $\mathbb{E}_{2}\circ\Phi_{i,2}=\mathbb{E}_{2}$.
Choose a net $(\epsilon_{i})_{i\in I}$ (we can use the same indexing set, modifying it  if necessary) with $\epsilon_{i}\rightarrow0$
and define unital normal completely positive $N$-$N$-bimodular maps
$\Phi_{i,1}^{\prime}:=\frac{1}{1+\varepsilon_{i}}(\Phi_{i,1}+\epsilon_{i}\mathbb{E}_{1})$,
$\Phi_{i,2}^{\prime}:=\frac{1}{1+\varepsilon_{i}}(\Phi_{i,2}+\epsilon_{i}\mathbb{E}_{2})$.

In the following we will need to work with certain sets of multi-indices: for each $n \in \mathbb{N}$ set $\mathcal{J}_n=\{\bj=(j_1, \ldots, j_n):j_k\in \{1,2\} \textup{ and } j_k \neq j_{k+1} \textup{ for } k=1,\ldots,n-1\}$; put also $\mathcal{J} = \bigcup_{n \in \N} \mathcal{J}_n$.

Set $\varphi:=\tau\circ\mathbb{E}_{N}$, let $\Psi_{i}:=\Phi_{i,1}\ast\Phi_{i,2}\text{: }M\rightarrow M$
be the unital normal completely positive map with $\Psi_{i}|_{N}=\text{id}_{N}$
and $\Psi_{i}(x_{1}...x_{n})=\Phi_{i,j_{1}}^{\prime}(x_{1})...\Phi_{i,j_{n}}^{\prime}(x_{n})$
for $\bj \in \mathcal{J}_n$ and $ x_{k} \in M_{j_{k}}\cap\ker(\mathbb{E}_{j_{k}})$ 
for $k=1, \ldots, n$ (see \cite[Theorem 3.8]{BlanchardDykema})
and define $\Psi_{i}^{\prime}:=\Phi_{i,1}^{\prime}\ast\Phi_{i,2}^{\prime}$
analogously. We claim that the net $(\Psi_{i}^{\prime})_{i\in I}$
witnesses the relative Haagerup property of the triple $(M,N,\varphi)$.
Indeed, it is clear that the maps satisfy the conditions \eqref{Item=RHAP1},
\eqref{Eqn=RHAP2} and \eqref{Eqn=RHAP4} of Definition \ref{Dfn=RHAPState}.
It remains to show that $\Psi_{i}^{\prime}(x)\rightarrow x$ strongly
for every $x\in M$ and that the $L^{2}$-implementations $(\Psi_{i}^{\prime})^{(2)}$
are contained in $\mathcal{K}(M,N,\varphi)$.\\

Define for $n \in \N$ and $\bj \in \mathcal{J}_n$
the Hilbert subspace
\[
\mathcal{H}_{\bj}:=\overline{\text{Span}}\left\{ x_{1}...x_{n}\Omega_{\varphi}\mid x_{1}\in\ker(\mathbb{E}_{j_{1}}),...,x_{n}\in\ker(\mathbb{E}_{j_{n}})\right\} \subseteq L^{2}(M,\varphi)
\]
and let $P_{\bj}\in\mathcal{B}(L^{2}(M,\varphi))$ be
the orthogonal projection onto $\mathcal{H}_{\bj}$. Note
that these Hilbert subspaces are pairwise orthogonal for different multi-indices $\bj_1, \bj_2 \in \mathcal{J}$, orthogonal to
$N\Omega_{\varphi}\subseteq L^{2}(M,\varphi)$, one has inclusions
$\Psi_{i}^{(2)}\mathcal{H}_{\bj}\subseteq\mathcal{H}_{\bj}$,
$(\Psi_{i}^{\prime})^{(2)}\mathcal{H}_{\bj}\subseteq\mathcal{H}_{\bj}$
and the span of the union of all $\mathcal{H}_\bj$ ($\bj \in \mathcal{J}$) with $N\Omega_{\varphi}$ is dense in
$L^{2}(M,\varphi)$.\\

For the strong convergence it suffices to show that $\Vert\Psi_{i}^{(2)}\xi-\xi\Vert_{2}\rightarrow0$
for all $\xi\in\mathcal{H}_{\bj}$, $\bj \in \mathcal{J}$. So let $n \in \N$, $\bj \in \mathcal{J}_n$, $x_{1}\in\ker(\mathbb{E}_{j_{1}})$,
..., $x_{n}\in\ker(\mathbb{E}_{j_{n}})$. Then,
\begin{eqnarray}
\nonumber
& & \Vert  (\Psi_{i}')^{(2)}(x_{1}...x_{n}\Omega_{\varphi})-x_{1}...x_{n}\Omega_{\varphi}\Vert_{2}=\Vert\Phi_{i,j_{1}}^{\prime}(x_{1})...\Phi_{i,j_{n}}^{\prime}(x_{n})\Omega_{\varphi}-x_{1}...x_{n}\Omega_{\varphi}\Vert_{2}\\
\nonumber
&\leq& \Vert(\Phi_{i,j_{1}}^{\prime}(x_{1})-x_{1})\Phi_{i,j_{2}}^{\prime}(x_{2})...\Phi_{i,j_{n}}^{\prime}(x_{n})\Omega_{\varphi}\Vert_{2}+\Vert x_{1}\Vert\Vert\Phi_{i,j_{2}}^{\prime}(x_{2})...\Phi_{i,j_{n}}^{\prime}(x_{n})\Omega_{\varphi}-x_{2}...x_{n}\Omega_{\varphi}\Vert_{2}\\
\nonumber
&\leq& ... \: \leq\Vert(\Phi_{i,j_{1}}^{\prime}(x_{1})-x_{1})\Phi_{i,j_{2}}^{\prime}(x_{2})...\Phi_{i,j_{n}}^{\prime}(x_{n})\Omega_{\varphi}\Vert_{2}+\Vert x_{1}\Vert\Vert(\Phi_{i,j_{2}}^{\prime}(x_{2})-x_{2})\Phi_{i,j_{3}}^{\prime}(x_3)...\Phi_{i,j_{n}}^{\prime}(x_{n})\Omega_{\varphi}\Vert_{2}\\
\nonumber
& & \qquad \qquad +...+\Vert x_{1}\Vert...\Vert x_{n-1}\Vert\Vert\Phi_{i,j_{n}}^{\prime}(x_{n})\Omega_{\varphi}-x_{n}\Omega_{\varphi}\Vert_{2} \rightarrow 0\text{.}
\end{eqnarray}
This implies that indeed $\Psi_{i}(x)\rightarrow x$ strongly for
every $x\in M$.\\

To treat the relative compactness, express the operators $(\Phi_{i,1}^{\prime})^{(2)}\in\mathcal{K}(M_{1},N,\varphi_{1})$,
$(\Phi_{i,2}^{\prime})^{(2)}\in\mathcal{K}(M_{2},N,\varphi_{2})$
as norm-limits
\[
(\Phi_{i,1}^{\prime})^{(2)}=\lim_{l\rightarrow\infty}\sum_{k=1}^{N_{l}^{(i,1)}}a_{k,l}^{(i,1)}e_{N}^{\varphi_{1}}b_{k,j}^{(i,1)} \: \: \text{ and } \: \: (\Phi_{i,2}^{\prime})^{(2)}=\lim_{l\rightarrow\infty}\sum_{k=1}^{N_{j}^{(i,2)}}a_{k,l}^{(i,2)}e_{N}^{\varphi_{2}}b_{k,l}^{(i,2)}
\]
for suitable $N_{l}^{(i,1)},N_{j}^{(i,2)}\in\mathbb{N}$, $a_{k,l}^{(i,1)},b_{k,l}^{(i,1)}\in M_{1}$
and $a_{k,l}^{(i,2)},b_{k,l}^{(i,2)}\in M_{2}$ .\\

\noindent \emph{Claim}. For $n \in \N$, $\bj \in \mathcal{J}_n$,
we have
\begin{eqnarray} \label{AmalgamatedInequality}
\left\Vert (\Psi_{i}^{\prime})^{(2)}P_\bj\right\Vert \leq\left(\frac{1}{1+\epsilon_{i}}\right)^{n}
\end{eqnarray}
and
\begin{eqnarray} \label{AmalgamatedEquality}
(\Psi_{i}^{\prime})^{(2)}P_\bj=\lim_{l_{1},...,l_{n}\rightarrow\infty}\sum_{k_{1},...,k_{n}}a_{k_{1},l_{1}}^{(i,j_{1})}...a_{k_{n},l_{n}}^{(i,j_{n})}e_{N}b_{k_{n},l_{n}}^{(i,j_{n})}...b_{k_{1},l_{1}}^{(i,j_{1})},
\end{eqnarray}
where the convergence is in norm.

\noindent \emph{Proof of the claim}. For $x_{1}\in\ker(\mathbb{E}_{j_{1}})$,
..., $x_{n}\in\ker(\mathbb{E}_{j_{n}})$ one calculates
\begin{eqnarray}
\nonumber
(\Psi_{i}^{\prime})^{(2)}P_{\bj}(x_{1}...x_{n}\Omega_{\varphi}) &=& \Phi_{i,j_{1}}^{\prime}(x_{1})...\Phi_{i,j_{n}}^{\prime}(x_{n})\Omega_{\varphi}\\
\nonumber
&=& \left(\frac{1}{1+\epsilon_{i}}\right)^{n}\Phi_{i,j_{1}}(x_{1})...\Phi_{i,j_{n}}(x_{n})\Omega_{\varphi}\\
\nonumber
&=& \left(\frac{1}{1+\epsilon_{i}}\right)^{n}\Psi_{i}^{(2)}(x_{1}...x_{n}\Omega_{\varphi})
\end{eqnarray}
and hence $(\Psi_{i}^{\prime})^{(2)}P_{\bj}=(1+\epsilon_{i})^{-n}\Psi_{i}^{(2)} P_{\bj}$.
By the unitality of $\Phi_{i,1}$ and $\Phi_{i,2}$ the inequality \eqref{AmalgamatedInequality} then follows from
\[
\left\Vert (\Psi_{i}^{\prime})^{(2)}P_{\bj}\right\Vert =\left(\frac{1}{1+\epsilon_{i}}\right)^{n}\left\Vert \Psi_{i}^{(2)}P_{\bj}\right\Vert \leq \left(\frac{1}{1+\epsilon_{i}}\right)^{n}\left\Vert \Psi_{i}^{(2)}\right\Vert \leq \left(\frac{1}{1+\epsilon_{i}}\right)^{n}\left\Vert \Psi_{i}\right\Vert =\left(\frac{1}{1+\epsilon_{i}}\right)^{n}\text{.}
\]
We proceed by induction over $n$. For $n=1$ the equality \eqref{AmalgamatedEquality}
is clear. Assume that the equality \eqref{AmalgamatedEquality} holds
for $\bj \in \mathcal{J}_{n-1}$
and let $j_{n}\in\left\{ 1,2\right\} $ with $j_{n}\neq i_{n-1}$, $\bj' = (\bj,j_n)$.
One easily checks that the left- and right-hand side of \eqref{AmalgamatedEquality}
both vanish on the orthogonal complement of $\mathcal{H}_{\bj'}$.
Further, for $x_{1}\in\ker(\mathbb{E}_{j_{1}})$, ..., $x_{n}\in\ker(\mathbb{E}_{j_{n}})$,
we get by the assumption
\begin{eqnarray}
\nonumber
& & (\Psi_{i}^{\prime})^{(2)}(x_{1}...x_{n}\Omega_{\varphi})=\Psi_{i}^{\prime}(x_{1}...x_{n-1})\Phi_{i,j_{n}}^{\prime}(x_{n})\Omega_{\varphi}=\Psi_{i}^{\prime}(x_{1}...x_{n-1})(\Phi_{i,j_{n}}^{\prime})^{(2)}(x_{n}\Omega_{\varphi})\\
\nonumber
& & =\lim_{l_{1},...,l_{n}\rightarrow\infty}\sum_{k_{1},...,k_{n-1}}a_{k_{1},l_{1}}^{(i,j_{1})}...a_{k_{n-1},l_{n-1}}^{(i,j_{n-1})}\mathbb{E}_{N}\left(b_{k_{n-1},l_{n-1}}^{(i,j_{n-1})}...b_{k_{1},l_{1}}^{(i,j_{1})}x_{1}...x_{n-1}\right)\left(\sum_{k_{n}}a_{k_{n},l_{n}}^{(i,j_{n})}e_{N}b_{k_{n},l_{n}}^{(i,j_{n})}\right)x_{n}\Omega_{\varphi}\text{.}
\end{eqnarray}
Since the $\Phi_{i,1}^{\prime}$ and $\Phi_{i,2}^{\prime}$ are $N$-$N$-bimodular,
we have $(\Phi_{i,j_{n}}^{\prime})^{(2)}\in N^{\prime}\cap\left\langle N,M\right\rangle $
and hence
\[
(\Psi_{i}^{\prime})^{(2)}(x_{1}...x_{n}\Omega_{\varphi})=\lim_{l_{1},...,l_{n}\rightarrow\infty}\sum_{k_{1},...,k_{n}}a_{k_{1},l_{1}}^{(i,j_{1})}...a_{k_{n},l_{n}}^{(i,j_{n})}\mathbb{E}_{N}\left(b_{k_{n},l_{n}}^{(i,j_{n})}...b_{k_{1},l_{1}}^{(i,j_{1})}x_{1}...x_{n}\right)\Omega_{\varphi},
\]
i.e.
\[
\sum_{k_{1},...,k_{n}}a_{k_{1},l_{1}}^{(i,j_{1})}...a_{k_{n},l_{n}}^{(i,j_{n})}e_{N}b_{k_{n},l_{n}}^{(i,j_{n})}...b_{k_{1},l_{1}}^{(i,j_{1})}\rightarrow(\Psi_{i}^{\prime})^{(2)}
\]
 strongly in $l_{1},...,l_{n}$. The second part  of the claim, i.e. \eqref{AmalgamatedEquality},  then
follows from noticing that
\[
\left(\sum_{k_{1},...,k_{n}}a_{k_{1},l_{1}}^{(i,j_{1})}...a_{k_{n},l_{n}}^{(i,j_{n})}e_{N}b_{k_{n},l_{n}}^{(i,j_{n})}...b_{k_{1},l_{1}}^{(i,j_{1})}\right)_{l_{1},...,l_{n}}
\]
is a Cauchy sequence (compare with \cite[Section 3]{Boca}).\\

The (in)equalities \eqref{AmalgamatedInequality} and \eqref{AmalgamatedEquality} in particular
imply that $(\Psi_{i}^{\prime})^{(2)}$ can be expressed as a norm
limit
\[
(\Psi_{i}^{\prime})^{(2)}=e_{N}+\lim_{n\rightarrow\infty}\sum_{\bj \in \mathcal{J}_n}\Psi_{i}^{(2)}P_{\bj}
\]
 and hence $(\Psi_{i}^{\prime})^{(2)}\in\mathcal{K}(M,N,\varphi)$
for all $i\in I$. This finishes the direction ``$\Rightarrow$''.\\

``$\Leftarrow$''  It suffices to prove the result for $M_1$. Note first that  \cite[Lemma 3.5]{BlanchardDykema} shows that we have a normal conditional expectation $\mathbb{F}_1:M \to M_1$ such that $\mathbb{E}_1 \circ \mathbb{F}_1 = \mathbb{E}_N$. Hence Lemma \ref{lem:condexp} ends the proof. 
\end{proof}

In combination with Theorem \ref{Thm=HAPrelHAP}, Theorem \ref{Thm=amalgamated}
leads to the following corollary. This generalizes a result by Freslon \cite[Theorem 2.3.19]{Freslon} who showed this corollary in the realm of von Neumann algebras of discrete quantum groups, and the analogous property for classical groups was first shown in \cite{Jolissaint00} (see also \cite[Section 6]{HaagerupProperty}).
 To the authors' best knowledge even for inclusions of finite von Neumann algebras the statement of the following corollary
is new.

\begin{cor} \label{Cor=amalgamated}
Let $ N\subseteq M_{1}$
and $ N\subseteq M_{2}$ be unital embeddings of
von Neumann algebras which admit faithful normal conditional expectations
$\mathbb{E}_{1}:M_{1}\rightarrow N$, $\mathbb{E}_{2}:M_{2}\rightarrow N$
and assume that $N$ is finite-dimensional. Assume moreover that $M_{1}$
and $M_{2}$ have the Haagerup property. Then the amalgamated free product
von Neumann algebra $(M_{1}, \mathbb{E}_1) \ast_{N} (M_{2}, \mathbb{E}_2)$ has the Haagerup property as
well.
\end{cor}


\subsection{ Haagerup property for Hecke-von Neumann algebras of
virtually free Coxeter groups}

Let us now demonstrate the application of Corollary \ref{Cor=amalgamated} in the context of virtually free Hecke-von Neumann algebras.

A \emph{Coxeter system} $(W,S)$ consists of a set $S$ and a group
$W$ freely generated by $S$ with respect to relations of the form
$(st)^{m_{st}}=e$ where $m_{st}\in\{1,2,...,\infty\}$ with $m_{ss}=1$,
$m_{st}\geq2$ for all $s\neq t$ and $m_{st}=m_{ts}$. By $m_{st}=\infty$
we mean that no relation of the form $(st)^{m}=e$ with $m\in\mathbb{N}$
is imposed, i.e.\ $s$ and $t$ are free with respect to each other;  and the system is said to be \emph{right-angled} if $m_{st}\in \{2, \infty\}$ for all $s,t \in S$, $s \neq t$.
The system is of \emph{finite rank} if the generating set $S$ is
finite. A subgroup of $(W,S)$ is called {\it special} if it is generated by a subset of $S$.

With every Coxeter system one can associate a family of von Neumann
algebras, its \emph{Hecke-von Neumann algebras}, which can be viewed
as  $q$-deformations of the group von Neumann algebra $\mathcal{L}(W)$
of the Coxeter group $W$. For this, fix a multi-parameter $q:=(q_{s})_{s\in S}\in\mathbb{R}_{>0}^{S}$
with $q_{s}=q_{t}$ for all $s,t\in S$ which are conjugate in $W$.
Further, write $p_{s}(q):=(q_{s}-1)/\sqrt{q_{s}}$ for $s \in S$. Then the corresponding
Hecke-von Neumann algebra $\mathcal{N}_{q}(W)$ is the von Neumann
subalgebra of $\mathcal{B}(\ell^{2}(W))$   generated by the
operators $T_{s}^{(q)}$, $s\in S$ where $T_{s}^{(q)}:\mathcal{B}(\ell^{2}(W))\rightarrow\mathcal{B}(\ell^{2}(W))$
is defined by
\begin{eqnarray} \nonumber
T_{s}^{\left(q\right)}\delta_{\mathbf{w}}=\begin{cases}
\delta_{s\mathbf{w}} & \text{, if } |s \mathbf{w}|>|\mathbf{w}|\\
\delta_{s\mathbf{w}}+p_{s}(q)\delta_{\mathbf{w}} & \text{, if } |s \mathbf{w}|<|\mathbf{w}|
\end{cases}.
\end{eqnarray}
Here $\left|\cdot\right|$ denotes the word length function with respect
to the generating set $S$ and $(\delta_{\mathbf{w}})_{\mathbf{w}\in W}\subseteq\ell^{2}(W)$
is the canonical orthonormal basis of $\ell^{2}(W)$. For a group
element $\mathbf{w}\in W$ which can be expressed by a reduced expression
of the form $\mathbf{w}=s_{1}...s_{n}$ with $s_{1},...,s_{n}\in S$
we set $T_{\mathbf{w}}^{(q)}:=T_{s_{1}}^{(q)}...T_{s_{n}}^{(q)}\in\mathcal{N}_{q}(W)$.
This operator does not depend on the choice of the expression $s_{1}...s_{n}$
and the span of such operators is dense in $\mathcal{N}_{q}(W)$.
Further, the von Neumann algebra $\mathcal{N}_{q}(W)$ carries a canonical
faithful normal tracial state $\tau_{q}$ defined by $\tau_{q}(x):=\left\langle x\delta_{e},\delta_{e}\right\rangle $
for $x\in\mathcal{N}_{q}(W)$. For more details on Hecke-von Neumann
algebras see \cite[Chapter 20]{Davis}.

The aim of this subsection is to study the Haagerup property of Hecke-von
Neumann algebras of virtually free Coxeter groups. We will approach
this by decomposing these Hecke-von Neumann algebras as suitable amalgamated
free products over finite-dimensional subalgebras. In the case of
right-angled Hecke-von Neumann algebras the Haagerup property has been
obtained in \cite[Theorem 3.9]{Caspers}.

Fix a finite rank Coxeter system $(W,S)$. A subset $T\subseteq S$
is called \emph{spherical} if the \emph{special subgroup} $W_{T}\subseteq W$
generated by $T$ is finite. $(W,S)$ is called spherical if $S$ is a spherical subset.

If $W$ is an arbitrary group which decomposes as an amalgamated
free product $W=W_{1}\ast_{W_{0}}W_{2}$ where $(W_{1},S_{1})$, $(W_{2},S_{2})$
are Coxeter systems with $W_{0}=W_{1}\cap W_{2}$ and $S_{0}:=S_{1}\cap S_{2}$
generates $W_{0}$, then $(W,S_{1}\cup S_{2})$ is a Coxeter system
as well.  We may now define the class of {\it virtually free} Coxeter systems as the smallest class of Coxeter groups containing all spherical Coxeter groups and which is stable under taking amalgamated free products over special spherical subgroups. Note that the original definition of virtually free Coxeter systems is different, but by \cite[Proposition 8.8.5]{Davis} equivalent to the one used here.

Now, for a multi-parameter $q:=(q_{s})_{s\in S}\in\mathbb{R}_{>0}^{S}$
as above we have a natural unital embedding $\mathcal{N}_{q}(W_{0})\subseteq\mathcal{N}_{q}(W)$
(see \cite[19.2.2]{Davis}). Let $\mathbb{E}_{\mathcal{N}_{q}(W_{0})}:\mathcal{N}_{q}(W)\rightarrow\mathcal{N}_{q}(W_{0})$
be the unique faithful normal trace-preserving conditional expectation
onto $\mathcal{N}_{q}(W_{0})$. Then, for $\mathbf{w}\in W$ the equality
\[
\mathbb{E}_{\mathcal{N}_{q}(W_{0})}(T_{\mathbf{w}}^{(q)})=\left\{ \begin{array}{cc}
T_{\mathbf{w}}^{(q)}, & \text{ if }\mathbf{w}\in W_{0}\\
0, & \text{ if }\mathbf{w}\notin W_{0}
\end{array}\right.
\]
holds.

Let us show that the amalgamated free product decomposition of a Coxeter
group translates into the Hecke-von Neumann algebra setting. Note that the arguments using the (iterated) amalgamated free product description of Hecke-deformed Coxeter group $C^*$-algebras appear for example in \cite{SvenAdamK}, exploiting the earlier work on operator algebraic graph products in \cite{MartijnPierre}.

\begin{prop} \label{decomposition}
Let $(W,S)$ be a finite rank Coxeter system that decomposes as $W=W_{1}\ast_{W_{0}}W_{2}$
where $(W_{1},S_{1})$, $(W_{2},S_{2})$ are Coxeter systems with
$S=S_{1}\cup S_{2}$, $W_{0}=W_{1}\cap W_{2}$ such that $S_{0}:=S_{1}\cap S_{2}$
generates $W_{0}$. For a multi-parameter $q=(q_{s})_{s\in S}$ with
$q_{s}=q_{t}$ for all $s,t\in S$ which are conjugate in $W$ the
corresponding Hecke-von Neumann algebra $\mathcal{N}_{q}(W)$ decomposes
as an amalgamated free product of the form
\[
\mathcal{N}_{q}(W)=\mathcal{N}_{q_{1}}(W_{1})\ast_{\mathcal{N}_{q_{0}}(W_{0})}\mathcal{N}_{q_{2}}(W_{2})\text{,}
\]
where $q_{0}:=(q_{s})_{s\in S_{0}}$, $q_{1}:=(q_{s})_{s\in S_{1}}$
and $q_{2}:=(q_{s})_{s\in S_{2}}$. Here the decomposition is taken
with respect to the restricted conditional expectations $(\mathbb{E}_{\mathcal{N}_{q}(W_{0})})|_{\mathcal{N}_{q_{1}}(W_{1})}$
and $(\mathbb{E}_{\mathcal{N}_{q}(W_{0})})|_{\mathcal{N}_{q_{2}}(W_{2})}$.
\end{prop}

\begin{proof}
	We will use the multi-index notation from the proof of Theorem \ref{Thm=amalgamated}.
By the uniqueness of the amalgamated free product
construction in combination with our previous discussion, it suffices
to show that $\mathbb{E}_{\mathcal{N}_{q}(W_{0})}(a_{1}...a_{n})=0$
for all $n \in \N$, $\bj \in \mathcal{J}_n$ and   $a_{k}\in\mathcal{N}_{q_{j_{k}}}(W_{j_k}))\cap\ker(\mathbb{E}_{\mathcal{N}_{q}(W_{0})})$.  Let $\mathcal{N}_{q_{i}}(W_{i})_1$ denote the unit ball of $\mathcal{N}_{q_{i}}(W_{i})$.   Let  $\overline{\text{Span}}$ denote the strong closure of the linear span. By Kaplansky's density theorem,
\[
\mathcal{N}_{q_{1}}(W_{1})_1 \cap\ker(\mathbb{E}_{\mathcal{N}_{q}(W_{0})})= \mathcal{N}_{q_{1}}(W_{1})_1 \cap \overline{\text{Span}}\{ T_{\mathbf{w}}^{(q)}\mid\mathbf{w}\in W_{1}\setminus W_{0}\},
\]
and
\[
\mathcal{N}_{q_{2}}(W_{2})_1 \cap\ker(\mathbb{E}_{\mathcal{N}_{q}(W_{0})})= \mathcal{N}_{q_{2}}(W_{2})_1 \cap \overline{\text{Span}}\{ T_{\mathbf{w}}^{(q)}\mid\mathbf{w}\in W_{2}\setminus W_{0}\}\text{.}
\]
By \cite[Remark 4.3.1]{Murphy} the element $a_{1}...a_{n}$ can hence be approximated strongly by a bounded net of  linear combinations
of reduced expressions of the form $T_{\mathbf{w}_{1}}^{(q)}...T_{\mathbf{w}_{n}}^{(q)}$
with $\mathbf{w}_{k}\in W_{j_k}\setminus W_{0}$. 
But this expression coincides with $T_{\mathbf{w}_{1}...\mathbf{w}_{n}}^{(q)}$
where $\mathbf{w}_{1}...\mathbf{w}_{n}\in W\setminus W_{0}$ is non-trivial,
so $\mathbb{E}_{\mathcal{N}_{q}(W_{0})}(a_{1}...a_{n})=0$  since $\mathbb{E}_{\mathcal{N}_{q}}$ is normal and hence weakly continuous on bounded sets.
\end{proof}

The following corollary is an example of an application of Corollary
\ref{Cor=amalgamated} in a setting which is not covered by the results in \cite{Freslon}.

\begin{cor} \label{HeckeHaagerup}
Let $(W,S)$ be a finite rank Coxeter system, let $q=(q_{s})_{s\in S}\in\mathbb{R}_{>0}^{S}$
be a multi-parameter with $q_{s}=q_{t}$ if $s,t\in S$ are conjugate
in $W$ and assume that $W$ is virtually free. Then the corresponding
Hecke-von Neumann algebra $\mathcal{N}_{q}(W)$ has the Haagerup property.
\end{cor}

\begin{proof}
This follows from a combination of \cite[Theorem 3.9]{Caspers}, Proposition
\ref{decomposition} and Corollary \ref{Cor=amalgamated}.
\end{proof}

\section{Inclusions of finite index}\label{Sect=Quantum}
In this section we will discuss finite index inclusions for not necessarily tracial von Neumann algebras defined in \cite{finiteindex}. We will pick one of the (possibly nonequivalent) definitions, which is most suitable in our context, and then we will illustrate this notion using certain compact quantum groups, namely free orthogonal quantum groups.
\begin{dfn}
Let $N \subseteq M$ be an inclusion of von Neumann algebras with a faithful normal conditional expectation $\mathbb{E}_{N}: M \to N$. We say that a family of elements $(m_i)_{i\in I}$ is an \emph{orthonormal basis} of the right $N$-module $L^{2}(M)_{N}$ if
\begin{enumerate}
\item for each $i,j \in I$ we have $\mathbb{E}_{N}(m_{i}^{\ast} m_j) = \delta_{ij} p_{j}$, where $p_{j}$ is a projection in $N$;
\item $\overline{\sum_{i\in I} m_i N} = L^{2}(M)$.
\end{enumerate}
We say that the inclusion $N\subseteq M$ is strongly of finite index if it admits a finite orthonormal basis.
\end{dfn}

\begin{lem}
If an inclusion $N\subseteq M$ is strongly of finite index then it has the Haagerup property.
\end{lem}
\begin{proof}
Let $m_1,\dots, m_n$ be a finite orthonormal basis for our inclusion. It suffices to show that $x = \sum_{i=1}^{n} m_i \mathbb{E}_N( m_i^{\ast} x)$ for each $x\in M$. Indeed, this would show that the identity map on $L^{2}(M)$ is relatively compact with respect to $N$, so clearly the triple $(M,N, \mathbb{E}_N)$ satisfies the relative Haagerup property. The equality $x = \sum_{i=1}^{n} m_i \mathbb{E}_N( m_i^{\ast} x)$ has been already observed by Popa (see \cite[Section 1]{Popafiniteindex}) in a more general context.
\end{proof}
\subsection{Free orthogonal quantum groups}
We will now present a certain inclusion arising in the theory of compact quantum groups that has the relative Haagerup property. For information about compact quantum groups we refer the reader to the excellent book \cite{NeshveyevTuset}.

\begin{dfn}[\cite{OnF}]
Let $n\geqslant 2$ be an integer and let $F \in M_n(\mathbb{C})$ be a matrix such that $F \overline{F} = c\mathds{1}$ for some $c\in \mathbb{R}\setminus\{0\}$. Let $\Pol(O_F^{+})$ be the universal $\ast$-algebra generated by the entries of a unitary matrix $U\in M_{n}(\Pol(O_{F}^{+}))$, denoted $u_{ij}$, subject to the condition $U= F \overline{U} F^{-1}$, where $(\overline{U})_{ij}:= (u_{ij})^{\ast}$ for all $i,j=1,\ldots,n$.
Then the unique $\ast$-homomorphic extension of the map $\Delta(u_{ij}) := \sum_{k=1}^{n} u_{ik} \otimes u_{kj}$ makes $\Pol(O_{F}^{+})$ into a Hopf $\ast$-algebra, whose universal $C^{\ast}$-algebra completion yields a compact quantum group.
\end{dfn}
\begin{rem}
As every compact quantum group admits a Haar state, we can use the GNS construction to construct a von Neumann algebra $L^{\infty}(O_{F}^{+})$.
\end{rem}
In \cite{Banica} Banica classified irreducible representations of the compact quantum group $O_{F}^{+}$. He showed that they are indexed by natural numbers, $U^{k}$, where   $U^{0}$ is the trivial representation and $U^{1}=U$ is the fundamental representation $U$. Moreover, the fusion rules satisfied by these representations are the following:
\[
U^{k}\otimes U^{l} \simeq U^{k+l} \oplus U^{k+l-2} \oplus \dots \oplus U^{|k-l|},\;\;\; k, l \in \N,
\]
just like for the classical compact group $SU(2)$. From the fusion rules one can infer that the coefficients of representations indexed by even numbers form a subalgebra. Further, one can use the defining relation $U= F\overline{U} F^{-1}$ to show that they form a $\ast$-subalgebra.
\begin{dfn}
Let $M:= L^{\infty}(O_{F}^{+})$. We define the \emph{even subalgebra} $N$ to be the von Neumann subalgebra of $M$ generated by the elements $(u_{ij} u_{kl})_{1\leqslant i,j,k,l \leqslant n}$. It is equal to the von Neumann algebra generated by the coefficients of the even representations;  in fact it is related to the \emph{projective version} of $O_{F}^{+}$, usually denoted $PO_{F}^{+}$.
\end{dfn}
\begin{rem}
It has been shown by Brannan in \cite{Brannanthesis} that $N\subseteq M$ is a subfactor of index $2$ in case that $F=\mathds{1}$ (it is then an inclusion of finite von Neumann algebras).
\end{rem}
We now roughly outline Brannan's argument and then mention why it cannot immediately be translated into our setting. There is an automorphism $\Phi$ of $M$ such that $\Phi(u_{ij}) = -u_{ij}$; $\Phi$ can be first defined on $\Pol(O_{F}^+)$ by the universal property but it also preserves the Haar state, so can be extended to an automorphism of $L^{\infty}(O_{F}^+)$. The fixed point subalgebra of $\Phi$ is equal to the even subalgebra $N$ and therefore $\mathbb{E}_{N}:= \frac{1}{2}(\Id + \Phi)$ is a conditional expectation onto $N$ that preserves the Haar state. As a consequence $\mathbb{E}_{N} - \frac{1}{2} \Id $ is a completely positive map, so one can use the Pimsner-Popa inequality, which works for $\rm{II}_{1}$-factors, to conclude that the index of $N \subseteq M$ is at most $2$. On the other hand, any proper inclusion has index at least $2$, so the result follows. Unfortunately in the non-tracial case it is not clear if the condition that $\mathbb{E}_{N} - \frac{1}{2} \Id $ is completely positive implies that the inclusion $N\subseteq M$ is strongly of finite index; so far it is only known that it implies being of finite index in a weaker sense (see \cite[Th\'{e}or\`{e}me 3.5]{finiteindex}). Fortunately in our case it is possible to explicitly define a finite orthonormal basis.
\begin{prop}
	Let $n\geqslant 2$ be an integer and let $F \in M_n(\mathbb{C})$ be a matrix such that $F \overline{F} = c\mathds{1}$ for some $c\in \mathbb{R}\setminus\{0\}$.
Let $M:= L^{\infty}(O_{F}^{+})$ and let $N$ be the even von Neumann subalgebra of $M$. Then the inclusion $N\subseteq M$ is strongly of finite index. Moreover, one can find an orthonormal basis consisting of at most $n^{2}+1$ elements.
\end{prop}
\begin{proof}
	One can verify by an explicit computation  that $N$ is left globally invariant by the  modular automorphism group of the Haar state $h$ of $L^\infty(O_F^+)$, so we do have a faithful normal $h$-preserving conditional expectation $\mathbb{E}_N:M \to N$.
We start with $n^2+1$ elements  of $M$, namely $\mathds{1}$ and all the $u_{ij}$'s. Since we have all the coefficients of the fundamental representation, it follows from the fusion rules of $O_{F}^{+}$ that $N \oplus \sum_{i,j=1}^{n} u_{ij} N$ is a dense submodule of $L^{2}(M)_{N}$.

Note that all the elements $u_{ij}$ are odd, i.e. $\Phi(u_{ij}) = -u_{ij}$ for $i,j=1, \ldots,n$. Suppose that we have a family $x_1,\dots, x_k$ of odd elements. Then we can perform a Gram-Schmidt process to make this set orthonormal. To do it, first notice that $x_i^{\ast} x_i$ is an even element, hence so is $|x_i|$ -- we conclude that the partial isometry in the polar decomposition $x_i = v_i |x_i|$ is odd as well. Our process works as follows: we first replace $x_1$ by the corresponding partial isometry $v_1$. Then we define $\tilde{x}_2:= x_2 - v_1 v_1^{\ast} x_2$. Because $v_1$ is a partial isometry, we get $v_1^{\ast} \tilde{x}_2 = v_{1}^{\ast}x_2 - v_{1}^{\ast} v_1 v_{1}^{\ast} x_2 = 0$. We then define $v_2$ to be the partial isometry appearing in the polar decomposition of $\tilde{x}_2$; it still holds that $v_2$ is odd and $v_1^{\ast} v_2=0$. We can continue this process just like the usual Gram-Schmidt process and obtain an orthonormal set of odd partial isometries $v_{i}$ such that $\sum_{i=1}^{k} x_i N \subset \sum_{i=1}^{k} v_i N$; note that the projections $v_i^*v_i$ belong to $N$. If we apply this procedure to the family $(u_{ij})_{1\leqslant i,j\leqslant n}$, we obtain a finite orthonormal basis for the inclusion $N\subseteq M$.
\end{proof}

\begin{cor}
The inclusion $N\subseteq M:= L^{\infty}(O_{F}^{+})$ has the relative Haagerup property.
\end{cor}

\end{document}